\documentclass[12pt, reqno]{amsart}

\usepackage{
amsmath,
amssymb,
amsthm,
mathrsfs,
amsfonts,
ytableau,
enumitem,
comment,
upgreek,
soul,
yhmath,
mathtools,
shuffle,
kotex,
array,
caption,
hhline 
}

\usepackage[linktocpage]{hyperref}
\hypersetup{
    colorlinks=true,
    linkcolor = blue,
    citecolor=magenta,
    urlcolor=cyan,
}

\usepackage{float}
\usepackage[capitalise, nameinlink]{cleveref}
\crefname{equation}{}{}
\crefname{subsection}{Subsection}{Subsections}
\crefname{example}{Example}{Examples}
\crefname{theorem}{Theorem}{Theorems}
\crefname{lemma}{Lemma}{Lemmas}
\crefname{proposition}{Proposition}{Propositions}
\crefname{thm}{Theorem}{Theorems}
\crefname{lem}{Lemma}{Lemmas}
\crefname{prop}{Proposition}{Propositions}
\crefname{figure}{Figure}{Figures}
\crefname{fig}{Figure}{Figures}
\crefname{remark}{Remark}{Remarks}
\crefname{rem}{Remark}{Remarks}
\crefname{cor}{Corollary}{Corollaries}
\crefname{corollary}{Corollary}{Corollaries}
\crefname{conjecture}{Conjecture}{Conjectures}
\crefname{conj}{Conjecture}{Conjectures}
\crefname{ex}{Example}{Examples}

\usepackage{tikz,tikz-cd}

\usepackage[colorinlistoftodos, textwidth = 2.3cm]{todonotes}

\usepackage{bbm}
\usepackage{dsfont}

\ytableausetup{mathmode, boxsize=1.2em}

\title[Discovering product and coproduct Rules for Bases of $\QSym_F$]
{Discovering product and coproduct Rules for Bases of $\QSym_F$ through Supercharacters}

\author[W.-S. Jung]{Woo-Seok Jung}
\address{Department of Mathematics, Sogang University, Seoul 04107, Republic of Korea}
\email{jungws@sogang.ac.kr}

\author[Y.-T. Oh]{Young-Tak Oh}
\address{Department of Mathematics, Sogang University, Seoul 04107, Republic of Korea}
\email{ytoh@sogang.ac.kr}

\thanks{All authors were supported by the National Research Foundation of Korea (NRF) Grant funded by the Korean Government (NRF-2020R1F1A1A01071055).}

\date{\today}

\newtheorem{theorem}{Theorem}[section]
\newtheorem{proposition}[theorem]{Proposition}
\newtheorem{lemma}[theorem]{Lemma}
\newtheorem{corollary}[theorem]{Corollary}

\newtheorem*{claim*}{Claim}

\theoremstyle{definition}

\newtheorem{example}[theorem]{Example}
\newtheorem{definition}[theorem]{Definition}
\newtheorem{remark}[theorem]{Remark}

\numberwithin{equation}{section} \numberwithin{figure}{section}
\numberwithin{table}{section}

\newcommand{\nc}{\newcommand}

\nc{\wt}{\mathsf{wt}}
\nc{\sw}{\mathsf{sw}}
\nc{\sDes}{\mathrm{sDes}}
\nc{\Bre}{\mathrm{Bre}}
\nc{\SG}{\mathfrak{S}}
\nc{\frakR}{\mathfrak{R}}
\nc{\frakL}{\mathfrak{L}}
\nc{\PCT}{\mathrm{PCT}}
\nc{\SPCT}{\mathrm{SPCT}}
\nc{\RT}{\mathrm{RT}}
\nc{\SRT}{\mathrm{SRT}}
\nc{\RCT}{\mathrm{RCT}}
\nc{\SRCT}{\mathrm{SRCT}}
\nc{\SYRT}{\mathrm{SYRT}}
\nc{\SYCT}{\mathrm{SYCT}}
\nc{\SPYCT}{\mathrm{SPYCT}}
\nc{\tst}{\mathtt{st}}
\nc{\Span}{\mathrm{span}}
\nc{\comp}{\mathrm{comp}}
\nc{\rmst}{\mathrm{st}}
\nc{\std}{\mathsf{std}}
\nc{\Des}{\mathrm{Des}}
\nc{\set}{\mathrm{set}}
\nc{\cf}{\textsf{cf}}
\nc{\scf}{\textsf{scf}}
\nc{\ch}{\mathrm{ch}}
\nc{\cl}{\mathrm{cl}}
\nc{\id}{\mathrm{id}}
\nc{\sh}{\mathrm{sh}}
\nc{\Cop}{\mathrm{Cop}}
\nc{\bfS}{\mathbf{S}}
\nc{\bfm}{\mathbf{m}}
\nc{\hbfS}{\widehat{\mathbf{S}}}
\nc{\bfF}{\mathbf{F}}
\nc{\calB}{\mathcal{B}}
\nc{\calS}{\mathcal{S}}
\nc{\hcalS}{\widehat{\mathcal{S}}}
\nc{\alphamax}{\alpha_{\rm max}}
\nc{\brho}{\overline{\rho}}
\nc{\bphi}{\overline{\phi}}
\nc{\calV}{\mathcal{V}}
\nc{\calR}{\mathcal{R}}
\nc{\sfR}{\mathsf{R}}
\nc{\calG}{\mathcal{G}}
\nc{\tal}{\lambda(\alpha)}
\nc{\tbe}{\widetilde{\beta}}
\nc{\opi}{\overline{\pi}}
\nc{\calP}{\mathcal{P}}
\nc{\rmtop}{\mathrm{top}}
\nc{\rad}{\mathrm{rad}}
\nc{\bfP}{\mathbf{P}}
\nc{\SET}{\mathrm{SET}}
\nc{\SIT}{\mathrm{SIT}}
\nc{\rev}{\mathrm{r}}
\nc{\Th}{\theta}
\nc{\mPhi}{\Phi}
\nc{\mphi}{\phi}
\nc{\mPsi}{\Psi}
\nc{\hmPsi}{\widehat{\Psi}}
\nc{\mpsi}{\psi}
\nc{\mGam}{\Gamma}
\nc{\tcd}{\mathtt{cd}}
\nc{\trd}{\mathtt{rd}}
\nc{\trcd}{\mathtt{rcd}}
\nc{\rmr}{\mathrm{r}}
\nc{\rmc}{\mathrm{c}}
\nc{\rmt}{\mathrm{t}}
\nc{\bubact}{\,\scalebox{0.6}{$\bullet$}\,}
\nc{\hbubact}{\,\scalebox{0.6}{$\widehat{\bullet}$}\,}
\nc{\col}{\rm col}
\nc{\row}{\rm row}
\nc{\calE}{\mathcal{E}}
\nc{\calT}{\mathscr{T}}
\nc{\sfT}{\mathsf{T}}
\nc{\calEsa}{\mathcal{E}^\sigma(\alpha)}
\nc{\tauC}{\tau_{\scalebox{0.5}{$C$}}}
\nc{\sytabC}{\sytab_{\scalebox{0.5}{$C$}}}
\nc{\bbfP}{\overline{\bfP}}
\nc{\pr}{\mathbf{pr}}
\nc{\Ups}{\Upsilon}
\nc{\pact}{\diamond}
\nc{\tauE}{\tau_{\scalebox{0.5}{$E$}}}
\nc{\tauF}{\tau_{\scalebox{0.5}{$F$}}}
\nc{\tauG}{\tau_{\scalebox{0.5}{$G$}}}
\nc{\rtE}{T_{\scalebox{0.5}{$E$}}}
\nc{\rtF}{T_{\scalebox{0.5}{$F$}}}
\nc{\rtG}{T_{\scalebox{0.5}{$G$}}}
\nc{\oPaE}{\overline{\Phi}_{\alpha_E}}
\nc{\oPaF}{\overline{\Phi}_{\alpha_F}}
\nc{\oPaG}{\overline{\Phi}_{\alpha_G}}
\nc{\tab}{\tau}
\nc{\sytab}{\widehat{\tau}}
\nc{\hatE}{\widehat{E}}
\nc{\hati}{\hat{i}}
\nc{\hcalE}{\widehat{\calE}}
\nc{\hatC}{\widehat{C}}
\nc{\bal}{{\boldsymbol{\upalpha}}}
\nc{\bbe}{{\boldsymbol{\upbeta}}}

\nc{\bgam}{{\boldsymbol{\upgamma}}}
\nc{\bdel}{{\boldsymbol{\updelta}}}
\nc{\weakcon}{\odot}
\nc{\basisI}{I}

\nc{\ldalpha}{\lambda(\alpha)}
\nc{\SRIT}{\mathrm{SRIT}}
\nc{\re}{\mathrm{rev}}
\nc{\otau}{\overline{\tau}}
\nc{\rtop}{{\rm top}}
\nc{\sfc}{\mathsf{c}}
\nc{\sfr}{\mathsf{r}}
\nc{\tH}{\mathtt{H}}
\nc{\tV}{\mathtt{V}}
\nc{\rpi}{\mathring{\pi}}
\nc{\cpi}{\check{\pi}}
\nc{\frakm}{\mathfrak{m}}
\nc{\fke}{\mathfrak{e}}
\nc{\Hom}{\mathrm{Hom}}
\nc{\module}{\mathrm{mod} \, }

\nc{\SPCTsa}{\SPCT^\sigma(\alpha)}
\nc{\bfSsa}{\bfS_\alpha^\sigma}
\nc{\bfSsaC}{{\bfS}^\sigma_{\alpha,C}}
\nc{\hbfSsa}{\widehat{\bfS}_\alpha^\sigma}
\nc{\upineq}{\rotatebox{90}{$<$}}
\nc{\downineq}{\rotatebox{270}{$<$}}
\nc{\diagineq}{\rotatebox{135}{$<$}}
\nc{\frakB}{\mathfrak{B}}
\nc{\hxi}{\widehat{\xi}}
\nc{\hxidwJ}{\hxi_{\scalebox{0.55}{$J$}}}
\nc{\hxiupJ}{\hxi^{\scalebox{0.55}{$J$}}}
\nc{\scrS}{\mathscr{S}}
\nc{\bfT}{\mathbf{T}}

\nc{\ra}{\rightarrow}

\nc{\matr}[2]{\left( \hspace{-1ex} \begin{array}{c} #1 \\ #2 \end{array} \hspace{-1ex} \right)}

\newcommand{\QSym}{{\textsf {QSym}}}
\newcommand{\NSym}{{\textsf {NSym}}}
\newcommand{\FQSym}{{\textsf {FQSym}}}
\newcommand{\NCSym}{\textsf{NCSym}}
\newcommand{\Comp}{{\textsf {Comp}}}

\newcommand{\Irr}{{\rm Irr}}

\definecolor{wsgreen}{rgb}{0,0.5,0}

\nc{\DIRT}{\mathrm{DIRT}}
\nc{\hpi}{\pi}
\nc{\frakI}{\mathfrak{I}}
\nc{\hfrakI}{\widehat{\mathfrak{I}}}
\nc{\orho}{\overline{\rho}}
\nc{\autotheta}{\uptheta}
\nc{\autophi}{\upphi}
\nc{\autochi}{\upchi}
\nc{\autoomega}{\upomega}
\nc{\hIM}{\widehat{\frakB}}
\nc{\bfpi}{\boldsymbol{\uppi}}
\nc{\bfopi}{\overline{\boldsymbol{\uppi}}}
\nc{\ofrakB}{\overline{\frakB}}
\nc{\rmw}{\mathrm{w}}
\nc{\ostar}{\;\overline{*}\;}
\nc{\rank}{\mathrm{rank}}
\nc{\fkp}{\mathfrak{p}}
\nc{\bfR}{\mathbf{R}}

\nc{\upsig}{{\boldsymbol{\upsigma}}}
\nc{\bfSsaE}{{\bfS}^\upsig_{\alpha,E}}

\nc{\hfkp}{\widehat{\mathfrak{p}}}

\nc{\hautophi}{{\widehat{\autophi}}}
\nc{\hautotheta}{{\widehat{\autotheta}}}
\nc{\hautoomega}{{\widehat{\autoomega}}}
\nc{\rmperm}{\mathrm{perm}}

\nc{\bfsigJ}{\boldsymbol{\sigma}_{\scalebox{0.55}{$J$}}}
\nc{\bfrhoJ}{\boldsymbol{\rho}^{\scalebox{0.55}{$J$}}}

\nc{\pistar}[1]{\pi_{#1}^*}
\nc{\wfkp}{\widetilde{\mathfrak{p}}}
\nc{\bfpsi}{\boldsymbol{\uppsi}}

\nc{\yt}[1]{\todo[size=\tiny,color=blue!10]{#1 \\ \hfill --- Young-Tak}}
\nc{\YT}[1]{\todo[size=\tiny,inline,color=blue!10]{#1
		\\ \hfill --- Young-Tak}}

\nc{\ws}[1]{\todo[size=\tiny,color=green!10]{#1 \\ \hfill ---  Woo-Seok}}
\nc{\WS}[1]{\todo[size=\tiny,inline,color=green!10]{#1
		\\ \hfill --- Woo-Seok}}

\definecolor{purple}{rgb}{0.44, 0.0, 1.0}

\newenvironment{red}{\relax\color{red}}{\hspace*{.5ex}\relax}
\newenvironment{blue}{\relax\color{blue}}{\hspace*{.5ex}\relax}
\newenvironment{green}{\relax\color{wsgreen}}{\hspace*{.5ex}\relax}
\newenvironment{magenta}{\relax\color{magenta}}{\hspace*{.5ex}\relax}
\newenvironment{purple}{\relax\color{purple}}{\hspace*{.5ex}\relax}

\nc{\ber}{\begin{red}}
\nc{\er}{\end{red}}
\nc{\beb}{\begin{blue}}
\nc{\eb}{\end{blue}}
\nc{\bema}{\begin{magenta}}
\nc{\ema}{\end{magenta}}
\nc{\begr}{\begin{green}}
\nc{\egr}{\end{green}}

\nc{\bepu}{\begin{purple}}
\nc{\epu}{\end{purple}}
\nc{\lb}{\pmb{\left[\vphantom{\frac{1}{2}}\right.}}
\nc{\rb}{\pmb{\left.\vphantom{\frac{1}{2}}\right]}}
\nc{\slb}{\pmb{[\vphantom{\frac{1}{2}}}}
\nc{\srb}{\pmb{\vphantom{\frac{1}{2}}]}}

\nc{\tran}{\mathsf{tran}}
\nc{\e}{\mathsf{e}}
\nc{\Int}{\mathsf{Int}}

\begin{document}

\maketitle

\begin{abstract}
In this paper, we establish product and coproduct rules for three bases of the Hopf algebra $\textsf{QSym}_F$ of quasisymmetric functions over $F$, with $F$ being either $\mathbb{C}(q,t)$ or $\mathbb{C}(q)$. These results are derived through the categorizations of $\textsf{QSym}_{\mathbb{C}}$ obtained by utilizing the normal lattice supercharacter theories. Firstly, we deal with a basis $\{\mathcal{D}_{\alpha}(q,t) \mid \alpha \in \textsf{Comp}\}$ of $\textsf{QSym}_{\mathbb{C}(q,t)}$, where $\textsf{Comp}$ denotes the set of all compositions. This basis is obtained from the direct sum of specific supercharacter function spaces and consists of superclass identifier functions. Upon appropriate specializations of $q$ and $t$, it yields notable bases of $\textsf{QSym}_{\mathbb{C}}$ and $\textsf{QSym}_{\mathbb{C}(q)}$, including enriched $q$-monomial quasisymmetric functions introduced by Grinberg and Vassilieva. Secondly, we deal with the basis $\{G_{\alpha}(q) \mid \alpha \in \textsf{Comp}\}$ of $\textsf{QSym}_{\mathbb{C}(q)}$, where $G_{\alpha}(q)$ represents the quasisymmetric Hall--Littlewood function introduced by Hivert. Our product rule is new, whereas our coproduct rule turns out to be equivalent to the existing coproduct rule of Hivert. Finally, we consider a basis $\{M_{\alpha}(q) \mid \alpha \in \textsf{Comp}\}$ of $\textsf{QSym}_{\mathbb{C}(q)}$, where $M_{\alpha}(q)$ is a $q$-analogue of the monomial quasisymmetric function.
\end{abstract}

\tableofcontents

\section{Introduction}
Quasisymmetric functions were initially introduced by Gessel~\cite{gessel84} as a generalization of symmetric functions. Since their introduction, they have been extensively investigated in connection with various areas of mathematics. From an algebraic viewpoint, it is particularly intriguing to note that the algebra $\QSym_\mathbb C$ of quasisymmetric functions over $\mathbb{C}$ possesses a Hopf algebra structure. The pivotal role of $\QSym_\mathbb C$ within the category of combinatorial Hopf algebras has been established in~\cite{AS06}, where it has been proven to be the terminal object, emphasizing its fundamental significance in this category.

The main objective of this paper is to derive product and coproduct rules for specific bases of the Hopf algebra $\QSym_F$ of quasisymmetric functions over $F$, where $F=\mathbb C(q,t)$ or $\mathbb C(q)$, by identifying $\QSym:=\QSym_\mathbb C$ with the direct sum
\begin{equation}\label{directsum of supercharacter function spaces}
\bigoplus_{n \ge 0} \scf(\mathcal{N}_n(\nu))
\end{equation}
for each positive integer $v>1$. 
Here, $\scf(\mathcal{N}_n(\nu))$ represents the supercharacter function space of a particular supercharacter theory $\mathcal{N}_n(\nu)$ of the direct product of $n-1$ copies $\bigoplus_{n-1} C_{\nu}$ of the cyclic group $C_{\nu}$ of order $\nu$.

In~\cref{Section: categorification of QSym}, we introduce the supercharacter function space $\scf(\mathcal{N}_n(\nu))$ for every nonnegative integer $n$ and every positive integer $\nu>1$.
This space possesses two intrinsic bases: 
\begin{itemize}
    \item $\{ \chi^I(\nu) \mid I \subseteq [n-1] \} \quad $ (the basis of supercharacter functions)
    \item $\{\kappa_I(\nu) \mid I \subseteq [n-1] \} \quad $ (the basis of superclass identifier functions)
\end{itemize}
Then we define a product ${\bf m}$ and a coproduct $\blacktriangle$ on the $\mathbb C$-vector space given in~\cref{directsum of supercharacter function spaces}
as compositions of certain linear operators. 
Equipped with ${\bf m}$ and $\blacktriangle$, it is proven that 
this vector space
has a Hopf algebra structure and the map 
\begin{align*}
    \ch_\nu : \bigoplus_{n \ge 0} \scf(\mathcal{N}_n(\nu)) \to \QSym, \quad  \dot\chi^I(\nu)\mapsto L_{\comp(I)} \quad (I \subset [n-1])
\end{align*}
is an isomorphism of Hopf algebras as well.
Here, $\dot\chi^I(\nu)=\chi^{I}(\nu)/\chi^{I}(\nu)(\boldsymbol{0})$, $\comp(I)$ is the composition of $n$ corresponding to $I$, 
and $L_{\comp(I)}$ is the fundamental quasisymmetric function attached to $\comp(I)$
(\cref{thm: categorification of QSym}).
We finally provide product and coproduct rules for the basis $\{ \kappa_I(\nu)\}$ of $(\bigoplus_{n \ge 0} \scf(\mathcal{N}_n(\nu)), {\bf m},\blacktriangle)$
(\cref{thm: product formula of superclass identifier} and~\cref{thm: coproduct formula of superclass identifier}).

In~\cref{Section: new bases for NSym}, we present and investigate a new basis $\{ \mathcal{D}_{\alpha}(q,t)\mid \alpha \in \Comp \}$ of $\QSym_{\mathbb{C}(q,t)}$, where $\Comp$ represents the set of all compositions.
This basis possesses some notable properties.
For instance, the positivity phenomenon is demonstrated for the product and coproduct, and upon suitable specializations of $q$ and $t$, it yields notable bases of $\QSym$ and $\QSym_{\mathbb{C}(q)}$ as follows:
\begin{align*}
&\mathcal{D}_{\alpha}(1,0) = M_{\alpha}, \quad \mathcal{D}_{\alpha}(-1,1) = \Lambda^*_{\alpha}, \quad \mathcal{D}_{\alpha}(1,-1)= E_{\alpha},\\
&2^n \mathcal{D}_{\alpha}(2,-1) = \eta_{\alpha}, \quad (q+1)^n \mathcal{D}_{\alpha}(q+1,-1) = \eta^{(q)}_{\alpha}.
\end{align*}
Here, $M_{\alpha}$ is the monomial quasisymmetric function, $\Lambda^*_{\alpha}$ the dual of the elementary noncommutative symmetric function $\Lambda_{\alpha}$ in the Hopf algebra $\NSym$ of noncommutative symmetric functions, and $E_{\alpha}$ the essential quasisymmetric function introduced by Hoffman~\cite{H15}. In addition, $\eta_{\alpha}$ and $\eta^{(q)}_{\alpha}$ are the enriched monomial and enriched $q$-monomial quasisymmetric functions, respectively, as introduced by Grinberg and Vassilieva~\cite{GV21,GV22,GV23-1}.

The quasisymmetric functions $\mathcal{D}_{\alpha}(q,t)$ arise naturally in the context of our categorifications through 
the following relations:
$$\mathcal{D}_{\alpha}(-\nu, \nu-1) = \ch_\nu \left( \kappa_{\set(\alpha)^{\mathrm c}}(\nu) / (\nu-1)^{|\set(\alpha)^{\mathrm c}|}\right), \quad (\nu>1)$$ 
(\cref{cor: D and kappa}).  
Combining these relations with the product and coproduct rules for $\{\kappa_I(\nu)\}$ (\cref{thm: product formula of superclass identifier} and~\cref{thm: coproduct formula of superclass identifier}), we derive product and coproduct rules for $\{ \mathcal{D}_{\alpha}(q,t) \}$.
For the product rule, we introduce the combinatorial objects called \textit{two-way overlapping shuffles} of two compositions (\cref{two-way overlapping shuffle}).
Using these objects, we provide the product rule presented in the following  form:
for $\alpha,\beta \in \Comp$,
   \begin{align} \label{product form for Dalpha}
    \mathcal{D}_{\alpha}(q,t) \mathcal{D}_{\beta}(q,t) &= 
    \sum_{\gamma \in 
    \alpha \overline{\overline{\shuffle}} \beta} (q+t)^{\mathrm{c_1}(\gamma)} t^{\mathrm{c_2}(\gamma)} \mathcal{D}_{\gamma^+}(q,t)
    \end{align}  
where $\alpha \overline{\overline{\shuffle}} \beta$ denotes the set of two-way overlapping shuffles of $\alpha$ and $\beta$ (\cref{thm: str consts for B}).
For the undefined notations $\mathrm{c_1}(\gamma)$, $\mathrm{c_2}(\gamma)$,  and $\gamma^+$, refer to~\cref{two-way overlapping shuffle}. 
It is quite interesting to note that 
all coefficients in the above expansion are in $\mathbb{N}[q,t]$, which raises the question of interpreting these functions in representation theory.
As an important consequence of~\cref{thm: str consts for B}, we derive product rules for both $\eta_{\alpha}$ and $\eta_{\alpha}^{(q)}$ by appropriately specializing the variables $q$ and $t$ in~\cref{product form for Dalpha} (\cref{product and coproducts for enriched monomials}).
It should be remarked that product rules for $\eta_{\alpha}$ and $\eta_{\alpha}^{(q)}$ have been previously provided in~\cite[Theorem 3.11]{GV21} and~\cite[Corollary 1]{GV22}(or see~\cite[Theorem 5.1, 5.9, 5.14]{GV23-1}), respectively. However, our product rules take a different form compared to those in these references.

In~\cref{section: quasiHL}, we investigate the basis $\{G_{\alpha}(q) \mid \alpha \in \Comp \}$ of $\QSym_{\mathbb{C}(q)}$ due to Hivert~\cite[Section 6]{H00}, within the context of our framework.
The quasisymmetric Hall--Littlewood function $G_{\alpha}(q)$ interpolates between the fundamental quasisymmetric function and monomial quasisymmetric function. In detail,
\[
G_{\alpha}(0) = L_{\alpha} \quad \text{and}   \quad G_{\alpha}(1) = M_{\alpha}.
\]
We first find the class function $\mathbb{G}_I(\nu)$ such that 
$G_{\comp(I)}(\nu) = \ch_{\nu}(\mathbb{G}_I(\nu))$
for each $I \subseteq [n-1]$ and a positive integer $\nu>1$ (\cref{prop: class function correspond to G}). Then, applying~\cref{thm: categorification of QSym} to these class functions, we derive a product rule for the basis $\{G_{\alpha}(q)\}$ (\cref{thm: product rule of quasi HL} and~\cref{cor: product rule for quasiHL}). 
To be precise, for $I \subseteq [m-1]$ and $J \subseteq [n-1]$, our product rule expresses the coefficients $c(q)_{I,J}^{K}$ in the expansion of $G_{\comp(I)}(q)G_{\comp(J)}(q) = \sum_{K} c(q)_{I,J}^{K} G_{\comp(K)}(q)$ as
\begin{align*}
\sum_{\substack{w \in u \shuffle v[m] \\ \Des(w) \subseteq K }} \left(  \prod_{i \in \sDes(w_{\le m}^0) \setminus I}(1-q^{\wt_I(i)})   \prod_{i \in \sDes(w_{>m}^0) \setminus J}(1-q^{\wt_J(i)}) 
\prod_{i \in K \setminus \Des(w)}(q^{\wt_K(i)} - \sw^w_m(i))
\right).
\end{align*}
Here, $u \in \SG_m$ and $v \in \SG_n$ are selected from arbitrary permutations that satisfy the conditions $\Des(u) = I$ and $\Des(v) = J$.
In a similar way, we derive a coproduct rule for $\{G_{\alpha}(q)\}$ (\cref{thm: coproduct formula of G}). 
Though it was derived using a completely different method, it turns out to be equivalent to the product rule for the dual of ${G_{\alpha}(q)}$ given in~\cite[Theorem 6.15]{H00}.

In~\cref{Section: q basis M for QSym},
we present and investigate a basis $\{M_{\alpha}(q) \mid \alpha \in \Comp \}$ of $\QSym_{\mathbb{C}(q)}$, where $M_{\alpha}(q)$ is a $q$-analogue of the monomial quasisymmetric function $M_{\alpha}$. 
When $q$ is specialized to $0$ and $1$, 
both $M_{\alpha}(q)$ and $G_{\alpha}(q)$ yield the same  quasisymmetric functions. 
Following the same strategy as before, 
we first find the class function $\psi^{\set(\alpha)}(\nu) \in \scf(\mathcal{N}_n(\nu))$ such that 
$
M_{\alpha}( \nu)=\ch_{\nu}(\psi^{\set(\alpha)}(\nu))
$
for each composition $\alpha$ and a positive integer $\nu > 1$.
Then, applying~\cref{thm: categorification of QSym} to these class functions, we derive product and coproduct rules for $\{M_{\alpha}(q)\}$ 
(\cref{thm: structure constant for M}). 

The paper is organized as follows.
In~\cref{Section: Preliminaries}, we review the Hopf algebra $\QSym$ and normal lattice supercharacter theories.
In~\cref{Section: categorification of QSym}, for each integer $\nu>1$, we equip the $\mathbb C$-space $\bigoplus_{n \ge 0} \scf(\mathcal{S}(\mathcal{N}_n(\nu)))$ with a Hopf algebra structure that is isomorphic to $\QSym$. Here, $\mathcal{N}_n(\nu)$ is the normal lattice supercharacter theory introduced in~\cite{AT21-Nsym}. We also provide product and coproduct rules for the basis consisting of superclass identifier functions.
In~\cref{Section: new bases for NSym}, we introduce a new basis $\{ \mathcal{D}_{\alpha}(q,t)\mid \alpha \in \Comp \}$ of $\QSym_{\mathbb{C}(q,t)}$ and present product and coproduct rules for this basis.
In~\cref{section: quasiHL}, we present product and coproduct rules for the basis $\{G_{\alpha}(q) \mid \alpha \in \Comp \}$ of $\QSym_{\mathbb{C}(q)}$ due to Hivert.
Finally, in~\cref{Section: q basis M for QSym},we present  product and coproduct rules for the basis $\{M_{\alpha}(q) \}$ of $\QSym_{\mathbb{C}(q)}$, where $M_{\alpha}(q)$ is a $q$-analogue of the monomial quasisymmetric function.

\section{Preliminaries}\label{Section: Preliminaries}
Given any integers $m$ and $n$, define $[m,n]$ to be the interval $\{t\in \mathbb Z \mid m\le t \le n\}$ whenever $m \le n$ and the empty set $\emptyset$ else.
For simplicity, we set $[n]:=[1,n]$ and therefore
$[n]=\emptyset$ if $n<1$.
Unless otherwise stated, $n$ will denote a nonnegative integer throughout this paper. 

\subsection{The Hopf algebra of quasisymmetric functions}
\label{Hopf algebras in consideration}
A composition is a finite tuple $\alpha = (\alpha_1, \alpha_2,\ldots,\alpha_l)$ of positive integers. Its length is defined to be $l$ and denoted by $\ell(\alpha)$ and its size is defined to be $\alpha_1 + \alpha_2 +\cdots + \alpha_l$ and denoted by $|\alpha|$.
Denote the set of composition of $n$ by $\Comp_n$.
Conventionally, we define $\Comp_0$ to be the set containing only the empty composition, denoted by $\emptyset$. Furthermore, we denote $\Comp = \bigsqcup_{n \ge 0} \Comp_n$ as the set of compositions.
For each positive integer $n$, 
there is a 1-1 correspondence between $\Comp_n$ and subsets of $[n-1]$ given by 
\begin{align*}
     \alpha=(\alpha_1,\alpha_2,\ldots,\alpha_l) &\mapsto  {\rm set}(\alpha):=\{ \alpha_1, \alpha_1 + \alpha_2, \ldots, \alpha_1+\alpha_2+\ldots+\alpha_{l-1} \}\\
     S=\{s_1<s_2< \cdots< s_i\} &\mapsto {\rm comp}(S):= (s_1,s_2-s_1, \ldots,s_i-s_{i-1}, n-s_i). 
\end{align*}

Let $x=(x_1, x_2, \ldots )$ be the infinite totally ordered set of commuting variables, and let $\mathbb{C}[[x_1,x_2,\ldots]]$ be the algebra of formal power series of bounded degree. 
For each composition 
$\alpha=(\alpha_1,\alpha_2,\ldots,\alpha_l)$, define {\it the monomial quasisymmetric function $M_{\alpha}$} to be
$$
\sum_{i_1<i_2<\cdots<i_l} x_{i_1}^{\alpha_1}x_{i_2}^{\alpha_2} \cdots x_{i_{l}}^{\alpha_l}.
$$
The algebra $\QSym$ of quasisymmetric functions over $\mathbb{C}$ is defined by 
$$
\QSym := \bigoplus_{n \ge 0} \QSym_n \subseteq \mathbb{C}[[x_1,x_2,\ldots]],
$$
where $\QSym_n := {\rm span}_{\mathbb{C}} \{ M_{\alpha} \mid \alpha \in \Comp_n \}$.

For $\alpha \in \Comp$, define {\it the fundamental quasisymmetric function $L_{\alpha}$} to be
\[
\sum_{\substack{1 \le i_1 \le i_2 \le \cdots \le i_n \\ i_j < i_{j+1} \text{ for } j \in \set(\alpha) }} x_{i_1}x_{i_2} \cdots x_{i_n}.
\]
For $\alpha, \beta \in \Comp$, we say that $\alpha$ refines $\beta$ if one can obtain $\beta$ from $\alpha$ by combining some of its adjacent parts, equivalently, $\set(\beta) \subseteq \set(\alpha)$. Denote this by $\alpha \preceq \beta$.
Then the expansion of fundamental quasisymmetric function into monomial quasisymmetric function is given by 
\begin{align}\label{eq: mono to funda}
L_{\alpha}=\sum_{\beta \preceq \alpha} M_{\beta}.
\end{align}
This shows that $\{ L_{\alpha} \mid \alpha \in \Comp \}$ is a basis of $\QSym$.

The algebra $\QSym$ has a natural coproduct structure. 
In particular, in the basis of monomial quasisymmetric functions, the coproduct formula can be expressed 
in the following form:
$$
\triangle M_{\alpha}= \sum_{\beta \cdot \gamma = \alpha} M_{\beta} \otimes M_{\alpha},
$$
where $\beta \cdot \gamma$ is the concatenation of compositions $\beta$ and $\gamma$.
For reference, in this case, the antipode is given as follows:
\[
S(M_{\alpha}) = (-1)^{\ell(\alpha)} \sum_{\gamma \succeq \alpha^r} M_{\gamma},
\]
where $\alpha^r= (\alpha_l,\alpha_{l-1},\ldots,\alpha_1)$ is the reverse composition of $\alpha = (\alpha_1,\alpha_2,\ldots,\alpha_l)$.

In the rest of this subsection, we introduce product and coproduct rules for $\{ L_{\alpha} \mid \alpha \in \Comp \}$.
For positive integers $m$ and $n$, let 
$\textsf{Sh}_{m,n}$ be the set of permutations $\sigma$ in the symmetric group $\SG_{m+n}$ satisfying that 
\begin{align*}
&\sigma^{-1}(1) < \sigma^{-1}(2) < \cdots < \sigma^{-1}(m) \quad \text{and}\\ 
&\sigma^{-1}(m+1) < \sigma^{-1}(m+2) < \cdots < \sigma^{-1}(m+n).
\end{align*}
For words $u = u_1 u_2  \ldots u_m$, $v = v_1 v_2  \ldots v_n$
with entries taken from positive integers, and a permutation $\sigma \in \textsf{Sh}_{m,n}$, 
set
\begin{equation}\label{shuffle of words}
u \shuffle_{\sigma} v := w_1 w_2 \ldots w_{m+n},
\end{equation}
where $w_{\sigma^{-1}(i)}=u_i$ for $1\le i \le m$ and $w_{\sigma^{-1}(m+j)}= v_{j}$ 
for $1\le j \le n$.
Then {\it the multiset of shuffles of $u$ and $v$} is defined by  
\[u \shuffle v := \{ u \shuffle_{\sigma} v \mid \sigma \in \textsf{Sh}_{m,n} \}.\]
Throughout this paper, we conventionally identify each permutation $\pi \in \SG_n$ with the word $\pi_1 \pi_2 \ldots \pi_n$, where $\pi_i = \pi(i)$ for $i \in [n]$.
For $\pi \in \SG_n$ and a nonnegative integer $m$, {\it the $m$-shifted permutation of $\pi$}, denoted $\pi[m]$, is the permutation in $\SG_{[m+1, m+n]}$ corresponding to the word $(\pi_1+m) (\pi_2+m) \ldots (\pi_n+m)$.
For instance, $21[2]=43$ and therefore
\[
12\,\shuffle \,{21[2]}= \{12 {43}, 1{4}2{3}, 1{4}{3}2, {4}12{3},{4}1 { 3}2, {43}{12}\}.
\]
For a word $w = w_1 w_2 \ldots w_n$ of length $n$ with entries taken from non-negative integers, {\it the descent set of $w$} is defined by 
\begin{equation}
\Des(w) := \{ i \mid w(i) > w(i+1)\} \subseteq [n-1].   
\end{equation}
Now, we are ready to present product and coproduct rules for $\{ L_{\alpha} \mid \alpha \in \Comp \}$.
Let $I \subseteq [m-1]$ and $J \subseteq [n-1]$. Choose arbitrary permutations $u \in \SG_m$ and $v \in \SG_n$ satisfying $\Des(u) = I$ and $\Des(v) = J$. Then we have 
\begin{equation}\label{eq: L-product and set parameter}
L_{\comp(I)} L_{\comp(J)} = \sum_{w \in u \shuffle v[m]} L_{\comp(\Des(w))}.
\end{equation}
And, it is well known that for any permutation $w \in \SG_n$,
    \begin{equation}\label{eq: L-coproduct}
\triangle L_{\comp(\Des(w))} = \sum_{k=0}^{n} L_{\comp(\Des( \std(w_1 \ldots w_k) ))} \otimes L_{\comp(\Des( \std(w_{k+1} \ldots w_{n}) ))}
\end{equation}
(for instance, see~\cite[Cororally 8.1.14]{GR20}).
Here, $\std(w)$ is the standardization of $w$.
This coproduct rule can also be expressed in terms of compositions.
For two compositions $\alpha = (\alpha_1, \alpha_2, \ldots, \alpha_l)$ and $\beta = (\beta_1, \beta_2, \ldots, \beta_k )$, 
{\it the near-concatenation of $\alpha$ and $\beta$} is defined by the composition  
\begin{equation}\label{eq: def of odot}
\alpha \odot \beta := (\alpha_1, \alpha_2, \ldots, \alpha_l + \beta_1, \beta_2, \ldots, \beta_k).
\end{equation}
In the case where either $\alpha=\emptyset$ or $\beta=\emptyset$, we set $ \alpha \odot \emptyset := \alpha$ and $\emptyset \odot \beta :=\beta$.
Then the identity~\cref{eq: L-coproduct} is equivalent to the following formula.
\begin{equation}\label{eq: coproduct of L 2}
\triangle L_{\gamma} = \sum_{\alpha \cdot \beta = \gamma \text{ or } \alpha \odot \beta = \gamma} L_{\alpha} \otimes L_{\beta}
\end{equation}
for any $\gamma \in \Comp$ (for instance, see~\cite[Proposition 5.2.15]{GR20}).

\subsection{Supercharacter theories of a finite group}\label{subsection: Supercharacter theories of a finite group}
In this subsection, $G$ denotes a finite group.
Given a set partition $\mathcal{K}$ of $G$, let $f(G;\mathcal{K})$ 
be the $\mathbb C$-vector space consisting of functions which are constant on the blocks in $\mathcal{K}$, that is, 
\[
f(G;\mathcal{K}):= \{\psi \mid G \rightarrow \mathbb{C} \mid \psi(g) = \psi(k) \text{ if $g$ and $h$ are in the same block in $\mathcal K$}\}.
\]
In particular, if $\mathcal{K}$ is the partition of $G$ consisting of conjugacy classes, then $f(G;\mathcal{K})$ recovers the $\mathbb C$-vector space of class functions of $G$, denoted $\cf(G)$.
And, we let $\Irr(G)$ be the set of irreducible characters of $G$.

\begin{definition}{\rm (\cite{DI08})}\label{def: supercharacter theory}
{\it A supercharacter theory $\mathcal{S}$ of $G$} is a pair $(\cl(\mathcal{S}) , \ch(\mathcal{S}))$, where $\cl(\mathcal{S})$ is a set partition of $G$ and $\ch(\mathcal{S})$ 
is a set partition of $\Irr(G)$, such that
\begin{enumerate}
    \item[{\bf C1.}] $\{e\} \in \cl(\mathcal{S}) $,
    \item[{\bf C2.}] $|\cl(\mathcal{S}) | = |\ch(\mathcal{S})|,$
    \item[{\bf C3.}] $\text{For each block $C$ in $\ch(\mathcal{S})$}$,
    \[
    \chi^C:=\sum_{\psi \in C} \psi(e) \psi \in f(G; \cl(\mathcal{S})),
    \]
    where $e$ denotes the identity of $G$.
\end{enumerate}
\end{definition}

A block in $\cl(\mathcal{S})$ is referred to as {\it a  superclass of $\cl(\mathcal{S})$}. Furthermore, for each block $C$ in $\ch(\mathcal{S})$, the corresponding $\chi^C$ is referred as {\it a  supercharacter of $\mathcal{S}$}.
And, $f(G;\cl(\mathcal{S}))$ is called {\it the supercharacter function space of $\mathcal{S}$}, denoted by $\scf(\mathcal{S})$.

Perhaps the most familiar examples of supercharacter theories are 
\begin{align*}
 & (\{\text{conjugacy classes of $G$}\}, \{ \{\psi \} \mid \psi \in \Irr(G) \}) \text{ and }\\
 &(\{ \{e\}, G \setminus \{ e \} \}, \{ \{\mathbbm{1} \}, \Irr(G) \setminus \{\mathbbm{1} \}  \}),
\end{align*}
where $\mathbbm{1}$ is the trivial character of $G$.

For each $L \in \cl(\mathcal{S})$, consider the function $\kappa_L \in f(G;\cl(\mathcal{S}))$
defined by 
\[
\kappa_L(g) =
\begin{cases}
      1 &\text{if } g \in L, \\
      0 &\text{otherwise,}
\end{cases}
\]
which is called {\it the superclass identifier function attached to $L$}.
Combining the orthogonality of irreducible characters with the conditions {\bf C2} and {\bf C3}, 
one can easily see that 
$\{\kappa_L \mid L \in \cl(\mathcal{S}) \}$ and $\{\chi^C \mid C \in \ch(\mathcal{S}) \}$
are $\mathbb C$-bases of $\scf(\mathcal{S})$.

Supercharacter theories can be generated in many ways.
Here, we introduce the normal lattice supercharacter theory introduced by Aliniaeifard~\cite{A17}.
Let $\ker(G):=\{N \trianglelefteq G \}$ be the lattice of normal subgroups of $G$ ordered by inclusion. 
For $M,N\in \ker(G)$, the meet and join of $M$ and $N$ are given by
\[
M \vee N = MN, \quad M \wedge N = M \cap N.
\]
Define {\it a sublattice $\mathcal{N}$ of $\ker(G)$} by a subset of $\ker(G)$ such that
\begin{enumerate}
    \item $\{e\}, G \in \mathcal{N}$ and 
    \item $\mathcal{N}$ is closed under meet and join operations.
\end{enumerate}
Obviously every sublattice also forms a lattice under $\vee, \wedge$.
Given $R \in \mathcal{N}$, let
\[
\mathrm{C}(R) := \{ O \in \mathcal{N} \mid O \text{ covers } R\}
\]
and 
\begin{align*}
N_{\circ} &:= \{g \in N \mid g \notin M \text{ for all } M\in \mathcal{N} \text{ with } N \in {\rm C}(M)\} \text{ and }\\
X^{N^{\bullet}} &:= \{\psi \in \Irr(G) \mid N \subseteq \ker(\psi) , \text{ but } O \nsubseteq \ker(\psi) \text{ for all }O\in C(N)\}.
\end{align*}
With this notation, the following theorem is proved in~\cite{A17}. 
\begin{theorem}{\rm (\cite[Theorem 3.4]{A17})}\label{def of normal character theory} 
Given a sublattice $\mathcal{N}$ of $\ker(G)$, 
let
\begin{align*}
  &\cl(\mathcal{S} (\mathcal{N})) := \{N_{\circ} \mid N \in \mathcal{N} \text{ and } N_{\circ} \neq \emptyset \}\\
  &\ch(\mathcal{S} (\mathcal{N})) := \{ X^{N^{\bullet}} \mid N \in \mathcal{N} \text{ and } X^{N^{\bullet}} \neq \emptyset\}.
    \end{align*}
Then $(\cl(\mathcal{S} (\mathcal{N})),\ch(\mathcal{S}(\mathcal{N})))$ 
defines a supercharacter theory $\mathcal{S}(\mathcal{N})$ of $G$.
\end{theorem}

In the literature, the supercharacter theory $\mathcal{S}(\mathcal{N})$ in~\cref{def of normal character theory}
is called {\it a normal lattice supercharacter theory of $G$}.

Supercharacter theories have been used extensively in the problems of categorifying combinatorial Hopf algebras.
In 2013, using André's supercharacter theories of the unipotent upper triangular groups $UT_n(q)$ given in~\cite{A95}, Aguiar et al.~\cite{28people12} succeeded in categorifying the Hopf algebra $\NCSym$ of symmetric functions in noncommuting variables.
Using normal lattice supercharacter theories, Aliniaeifard and Thiem~\cite{AT21, AT21-Nsym} successfully categorify the Hopf algebra $\FQSym$ of free quasisymmetric functions 
and the Hopf algebra $\NSym$ of noncommutative symmetric functions.
In~\cref{Section: categorification of QSym}, we present a categorification of the Hopf algebra $\QSym$ using normal lattice supercharacter theories of certain finite abelian groups.

\section{Categorifications of $\QSym$ using supercharacter theories}
\label{Section: categorification of QSym}

In this section, we present categorifications of the Hopf algebra $\QSym$ using the framework of supercharacter theory. 

We assume that $\nu$ is a positive integer $>1$, which will be fixed throughout this section.

\subsection{Normal lattice supercharacter theories of $\bigoplus C_\nu$}
\label{subsection: normal lattice supercharacter theories of finite abelian groups}
In~\cite{AT21}, Aliniaeifard and Thiem categorified $\FQSym$ via towers of groups and their supercharacter theories, and our approach here is basically based on this paper.

For a cyclic group $C_\nu$ of order $\nu$ and a finite set $S$, let 
$$
Q_S(\nu) := \bigoplus_{s \in S} C_{\nu,s}, 
$$
where $C_{\nu,s} = C_\nu$ for all $s \in S$. 
When $S=\emptyset$, we understand $Q_S(\nu)$ as the trivial group.
For clarity, we use $0$ to represent the additive identity of $C_\nu$ and $\boldsymbol{0}$ to represent the identity of $Q_S(\nu)$.
For each subset $I$ of $S$, we identify $Q_I(\nu)$ with the subgroup of $Q_S(\nu)$ whose 
$j$th component is $C_\nu$ if $j\in I$ and $\{0\}$ else.
Let
\[\mathcal{N}_S(\nu):= \{ Q_I(\nu) \mid I \subseteq S \}.\]
Under this identification, it can be easily seen that  
$\mathcal{N}_S(\nu)$ is a sublattice of $Q_S(\nu)$ and thus 
gives rise to a normal lattice supercharacter theory $\mathcal{S}(\mathcal{N}_S(\nu))$ of $Q_S(\nu)$ 
(see~\cref{def of normal character theory}).
For simplicity, we write 
$Q_n(\nu)$ and $\mathcal{N}_n(\nu)$ for $Q_{[n-1]}(\nu)$ and $\mathcal{N}_{[n-1]}(\nu)$, respectively.

The sublattice $\mathcal{N}_S(\nu)$ form a distributive lattice which implies that $N_{\circ} \neq \emptyset$ and $X^{N^{\bullet}} \neq \emptyset$ by~\cite[Corollary 3.11]{AT20}. Therefore, all blocks of $\cl(\mathcal{S}(\mathcal{N}_S(\nu)))$
and $\ch(\mathcal{S}(\mathcal{N}_S(\nu)))$ are parametrized by subsets of $S$. 
This implies that the dimension of the supercharacter function space $\scf(\mathcal{S}(\mathcal{N}_n(\nu)))$ of $\mathcal{S}(\mathcal{N}_n(\nu))$ is $|\{ I \subseteq S  \}| = 2^{|S|}$, which is the same as the dimension  of $\QSym_{|S|+1}$.

For each $I \subseteq S$, set
\begin{equation}\label{def of two notable bases}
\begin{aligned}
 &\cl_I(\nu):= {Q_I(\nu)}_{\circ}    \quad     \,\text{ and } \quad  \ch_I(\nu):= X^{{Q_I(\nu)}^{\bullet}}\\
&\kappa_I(\nu):=\kappa_{\cl_I(\nu)} \qquad \text{ and } \quad  \chi^I(\nu):=\chi^{\ch_I(\nu)}. 
\end{aligned}
\end{equation}

\begin{remark}\label{ambiguity of notation}
It should be noted that the notation in~\cref{def of two notable bases} depends $S$ as well as $I$.
If necessary, we will clarify $S$ as in~\cref{def: product of scf}.
\end{remark}

From the definition of $Q_I(\nu)_{\circ}$ it follows that 
\[
\cl_I(\nu) = \{(g_i)_{i \in S} \in Q_n(\nu) \mid g_i \text{ is nonzero if $i \in I$ and zero else}\}.
\] 
Let $\mathbbm{1}$ be the trivial character of $C_\nu$ and $\mathbbm{reg}$ be the character of the regular representation of $C_\nu$. 

\begin{lemma}{\rm (cf.~\cite[Section 4.2]{AT21-Nsym})}\label{thm: supercharacter calculation}
Let $I \subseteq S$ and $\nu$ be a positive integer $>1$.
For $\boldsymbol{g} = (g_i)_{i \in S} \in Q_S(\nu)$, 
we have 
$$
\chi^I(\nu) (\boldsymbol{g}) = \prod_{i \in I} \mathbbm{1}(g_i) \prod_{j \in I^{\mathrm c}} (\mathbbm{reg} - \mathbbm{1}) (g_j).
$$
\end{lemma}
\begin{proof}
By~\cite[Corollary 3.4, 3.5]{AT20}, for $\boldsymbol{g} \in \cl_J(\nu)$, we have
\begin{align*}
    \chi^I(\nu) (\boldsymbol{g}) &= 
\begin{cases}
    \dfrac{\nu^{n-1}}{\nu^{|I|}} \left(\dfrac{\nu-1}{\nu}\right)^{|\text{covers of } I|}  \left(\dfrac{1}{1-\nu}\right)^{|J \setminus I|}  & \text{ if } I <_{c} J,\\
    0              & \text{otherwise}
\end{cases}\\
    &=
\begin{cases}
    (\nu-1)^{|I^{\mathrm c}|} \
    \left(\dfrac{-1}{\nu-1}\right)^{|J \setminus I|} & \text{ if } I <_{c} J,\\
    0              & \text{ otherwise},
\end{cases}
\end{align*}
where $<_c$ represents the covering relation of the poset $\{I \mid I \subseteq S \}$ ordered by inclusion.
Now, the assertion follows from the fact that $\mathbbm{1}(g) = 1$ for all $g \in G$ and
\[
(\mathbbm{reg} - \mathbbm{1}) (g) =
\begin{cases}
    \nu-1 & \text{ if } g = 0, \\
    -1  & \text{ otherwise.}
\end{cases}
\]
\end{proof}

\subsection{A Hopf algebra structure of $\bigoplus_{n \ge 0} \scf(\mathcal{N}_n(\nu))$}\label{section: a hopf algebra str of direct sum}
Hereafter, we simply write $\scf(\mathcal{N}_n(\nu))$ for $\scf(\mathcal S(\mathcal{N}_n(\nu)))$ for every nonnegative integer $n$. 
In particular, when $n=0,1$, we let  
\begin{align*}
\scf(\mathcal S(\mathcal{N}_0(\nu)))&=\cf(Q_0(\nu)):=\mathbb{C} \{\mathbbm{1}_0 \} \text{ and }\\
\scf(\mathcal S(\mathcal{N}_1(\nu)))&=\cf(Q_1(\nu)):=\mathbb{C} \{\mathbbm{1}_1 \}.
\end{align*}
In this subsection, we endow $\bigoplus_{n \ge 0} \scf(\mathcal{N}_n(\nu))$ with a Hopf algebra structure. 
To do this, we need to define a product ${\bf m}$ and a coproduct $\blacktriangle$ on $\bigoplus_{n \ge 0} \scf(\mathcal{N}_n(\nu))$.
Let us first collect the necessary definitions and notation.

For $U \subseteq V$,
let us identify $Q_U(\nu) \times Q_{U^{\mathrm c}}(\nu)$ with $Q_{V}(\nu)$ in the natural way. 
Given $\phi \in \cf(Q_U(\nu))$ and $\psi \in \cf(Q_{U^{\mathrm c}}(\nu))$, we define $\phi \otimes_U \psi$ by the class function of $Q_{V}(\nu)$ according to the following expression:
\begin{align}\label{eq: tensor S}
(\phi \otimes_U \psi) (a,b) = \phi(a) \psi(b),
\end{align}
where $(a,b) \in Q_U(\nu) \times Q_{U^{\mathrm c}}(\nu) = Q_{V}(\nu)$.
It is easy to see that if $\phi \in \scf(\mathcal{N}_S(\nu))$ and $\psi \in \scf(\mathcal{N}_{S^{\mathrm c}}(\nu))$ then $(\phi \otimes_S \psi) \in \scf(\mathcal{N}_n(\nu))$.
In the case where $S = [n-2]$, or equivalently $S^{\mathrm c} = \{n-1 \}$, 
we simply write $\otimes_1$ for $\otimes_S$.

By using~\cref{eq: tensor S} repeatedly, we can obtain the isomorphism of vector spaces
\begin{equation*}\label{eq: definition of []}
    \slb \,\, \cdot \,\, \srb: \bigotimes_{s \in S} \cf(C_{\nu,s})\to \cf(Q_S(\nu))
\end{equation*}
defined by $\lb \bigotimes_{s \in S} \phi_s \rb (\boldsymbol{g}) = \prod_{s\in S}\phi_s(g_s)$ for $\phi_s \in \cf(C_{\nu,s})$ and $\boldsymbol{g} = (g_s)_{s \in S} \in Q_S(\nu)$.
\smallskip
{\bf Convention.} 
When $ S=\{s_1<s_2< \cdots< s_{|S|}\}$, we 
write $\bigotimes_{s \in S} \phi_s$ and $\lb \bigotimes_{s \in S} \phi_s \rb$ as $(\phi_{s_1},\phi_{s_2},\ldots,\phi_{s_{|S|}})$
and $\lb \phi_{s_1},\phi_{s_2},\ldots,\phi_{s_{|S|}} \rb$, respectively.
\smallskip

For each $s \in S$, let   
\begin{align*}
    {\chi^I(\nu)}_{s} := 
\begin{cases}
    \mathbbm{1}  &   \text{if } s \in I,\\
    \mathbbm{reg-1}  & \text{if } s \in S\setminus I.
\end{cases}
\end{align*}
It follows from~\cref{thm: supercharacter calculation} that \[\chi^I(\nu)=\lb(\chi^I(\nu)_{s})_{s\in S} \rb.\]
For example, if $I = \{ 1,6 \} \subseteq S=\{1,2,4,6,7\}$, then
\[\chi^I(\nu) = \lb \mathbbm{1}, \mathbbm{reg-1}, \underline{\hspace{0.5cm}},\mathbbm{reg-1}, \underline{\hspace{0.5cm}}, \mathbbm{1}, \mathbbm{reg-1} \rb.\]

Consider a finite set $S=\{s_1<s_2< \cdots < s_{|S|}\}$. 
We define {\it the standardization of $S$} by the bijection
    \begin{align}\label{eqdef: standardization of S}
    \std_S: S \to [|S|], \quad s_i \mapsto i.
    \end{align}
For each subset $I$ of $[|S|]$, we set  
\begin{align}\label{definition of IS}
I_S:={\std_S}^{-1}(I).
\end{align}
Then we consider the group isomorphism 
\begin{align*}
    \iota_S: Q_S(\nu) \to Q_{|S|+1}(\nu), \quad 
    (g_{s_i})_{s_i\in S} \mapsto (g_{s_i})_{1\le i \le |S|},
\end{align*}
which is obtained by the standization of indices. 
This induces a $\mathbb C$-linear isomorphism 
\begin{align*}
    {\iota_S}^\ast  : \cf(Q_{|S|+1}(\nu)) \to \cf(Q_S(\nu)), \quad 
    \phi \mapsto \phi \circ {\iota}_S.
\end{align*}
Since ${\iota_S}^\ast(\chi^I(\nu)) = \chi^{I_S}(\nu)$,
this isomorphism restricts to the following isomorphism of the supercharacter function spaces:
\begin{align*}
    {\iota_S}^\ast : \scf(\mathcal{N}_{|S|+1}(\nu)) \to \scf(\mathcal{N}_S(\nu)), \quad 
    \chi^I(\nu) \mapsto \chi^{I_S}(\nu).
\end{align*}

Let $A$ be a subset of $[n]$. 
Write it as the disjoint union 
$$A=\bigcup_{1\le a_1<a_2<\cdots < a_r\le n}[a_i,b_i],$$
where $[a_i,b_i]$'s are maximal among the subintervals of $([n],\le)$ in $A$.
Let 
\begin{equation}\label{definition of Int(A) and e(A)}
\begin{aligned}
\mathsf{Int}(A) &:=\{[a_i,b_i] \mid 1\le i \le r\}\\
\e(A) &:= \{ b_i \mid [a_i,b_i] \in  \Int(A) \} \setminus \{ n \}\\
\overline{\e}(A) &:= \e(A) \sqcup \e(A^{\mathrm c}).
\end{aligned}
\end{equation}

For instance, when $A = \{ 1, 3, 4, 6, 8 \} \subseteq [10]$, one has that 
$\mathrm{Int}(A) =\{[1, 1], [3, 4], [6, 6], [8, 8]\}$,
$\e(A) = \{1, 4, 6, 8 \}$, $\e(A^{\mathrm c}) = \{ 2, 5, 7 \}$, and therefore 
$\overline{\e}(A) = \{ 1, 2, 4, 5, 6, 7, 8 \}$.
Besides, we need the following notations. 
\begin{itemize}
\item 
Let $G$ be a group and $H$ a subgroup of $G$. 
For $\phi \in \cf(G)$, denote by $\phi\downarrow^{G}_{H}$
the restriction of $\phi$ from $G$ to $H$.

\item
For positive integers $a$ and $b$, 
denote by $\binom{[a]}{b}$ the collection of subsets of $[a]$ with size $b$.

\item
For $I \subseteq S$, set 
$\dot\chi^I(\nu):=\dfrac{\chi^{I}(\nu)}{\chi^{I}(\nu)(\boldsymbol{0})}$.

\item
Let $m,n\ge 0$ and $A \in \binom{[m+n]}{n}$.
For $\phi \in \cf(Q_m(\nu))$ and $\psi \in \cf(Q_n(\nu))$,
set
\begin{equation}\label{def of s(phi psi)}
{\bf s}_A(\phi, \psi) := 
\left({{\iota}_{A^{\mathrm c}}}^\ast\left(\phi \otimes_1 \dfrac{\mathbbm{reg - 1}}{\nu-1}\right)\right) \otimes_{A^{\mathrm c}} \left({{\iota}_{A}}^\ast\left(\psi \otimes_1 \dfrac{\mathbbm{reg - 1}}{\nu-1} \right)\right).
\end{equation}
Here, $\otimes_{A^{\mathrm c}}$ denotes the tensor product defined in~\cref{eq: tensor S}.
\end{itemize}

\begin{definition}{\rm (product)}\label{def: product of scf}
Let $m$ and $n$ be nonnegative integers and $\nu$ be a positive integer $>1$.
\begin{enumerate}[label = {\rm (\alph*)}]
\item
For $A \in \binom{[m+n]}{n}$, let  
\[{\bf m}_A:\cf(Q_m(\nu)) \times \cf(Q_n(\nu)) \to \cf(Q_{m+n}(\nu))\]
be the $\mathbb C$-bilinear map given by 
\begin{equation*}
{\bf m}_A(\phi, \psi) 
=
\begin{cases}
\phi \psi \footnotemark
&\text{ if } m = 0 \text{ or } n =0,\\
\dot\chi^{\e(A)}(\nu) \otimes_{\overline{\e}(A)} 
\left( {\bf s}_A(\phi, \psi)  \big\downarrow^{Q_{m+n+1}(\nu)}_{Q_{[m+n-1] \setminus \overline{\e}(A)}(\nu)} \right) &\text{ otherwise}
\end{cases}
\end{equation*}
for $\phi \in \cf(Q_m(\nu))$ and $\psi \in \cf(Q_n(\nu))$.
Here, we are viewing $\dot\chi^{\e(A)}(\nu)$ as an element of 
$\cf(Q_{\overline{\e}(A)}(\nu))$ and $\otimes_{\overline{\e}(A)}$ denotes the tensor product defined in~\cref{eq: tensor S}.
\footnotetext{In the expression $\phi \psi$ above, we employ the notation to represent the scalar product. Specifically, we interpret $\phi \in \mathbb{C}{\mathbbm{1}_0} \cong \mathbb{C}$ as an element of $\mathbb{C}$ when $m = 0$, and similarly, we regard $\psi$ as an element of $\mathbb{C}$ when $n = 0$. To illustrate, suppose $\phi = a \mathbbm{1}_0$, then $\phi \psi = a \psi$.}

\item
We define 
\[{\bf m}:\cf(Q_m(\nu)) \times \cf(Q_n(\nu)) \to \cf(Q_{m+n}(\nu))\] 
by
$$
{\bf m}:=
\sum_{A \in \binom{[m+n]}{n}} {\bf m}_A.
$$
\end{enumerate}
\end{definition}

\begin{example} \label{example for multiplication}
Let $m=4, n=3$.
And let $I = \{2,3\} \subseteq [3]$, $J = \{2\} \subseteq [2]$, and $A = \{1,3,4 \}\in \binom{[7]}{3}$.
Then 
\begin{align*}
    \dot\chi^I(\nu) = 
    \lb \dfrac{\mathbbm{reg-1}}{\nu-1}, \mathbbm{1}, \mathbbm{1} \rb \quad \text{and} \quad
    \dot\chi^J(\nu) = \lb \dfrac{\mathbbm{reg-1}}{\nu-1}, \mathbbm{1} \rb,
\end{align*}
which implies
\begin{align*}
    &\dot\chi^I(\nu) \otimes_1 \dfrac{\mathbbm{reg-1}}{\nu-1} = 
    \lb \dfrac{\mathbbm{reg-1}}{\nu-1}, \mathbbm{1}, \mathbbm{1}, \dfrac{\mathbbm{reg-1}}{\nu-1} \rb \text { and }\\
    &\dot\chi^J(\nu) \otimes_1 \dfrac{\mathbbm{reg-1}}{\nu-1} = \lb \dfrac{\mathbbm{reg-1}}{\nu-1}, \mathbbm{1}, \dfrac{\mathbbm{reg-1}}{\nu-1} \rb.
\end{align*}
It follows that 
\begin{align*}
     {\bf s}_A(\dot\chi^I(\nu), \dot\chi^J(\nu))
     &\stackrel{\text{by def.}} {=}{{\iota}_{A^{\mathrm c}}}^\ast \left( \dot\chi^I(\nu) \otimes_1 \dfrac{\mathbbm{reg-1}}{\nu-1} \right) \otimes_{A^{\mathrm c}} \left( {{\iota}_{A}}^\ast \left( \dot\chi^J(\nu) \otimes_1 \dfrac{\mathbbm{reg-1}}{\nu-1} \right) \right) \\
     &{\,\,\,\,}= 
    \lb \dfrac{\mathbbm{reg - 1}}{\nu-1},  \dfrac{\mathbbm{reg - 1}}{\nu-1}  ,  \mathbbm{1}  , \dfrac{\mathbbm{reg - 1}}{\nu-1}, \mathbbm{1}, \mathbbm{1}, \dfrac{\mathbbm{reg - 1}}{\nu-1} \rb.
    \end{align*}
On the other hand, since $\overline{\e}(A) = \{1,2,4\}$, we have 
\begin{align}\label{restriction of $S_A$}
    {\bf s}_A(\dot\chi^I(\nu), \dot\chi^J(\nu))  \big\downarrow^{Q_{m+n+1}(\nu)}_{Q_{[m+n-1] \setminus \overline{\e}(A)}(\nu)} 
    = 
    \lb \, \underline{\hspace{0.5cm}},  \underline{\hspace{0.5cm}}  ,  \mathbbm{1}  ,  \underline{\hspace{0.5cm}},  \mathbbm{1},  \mathbbm{1} \rb 
    \in \cf(Q_{[m+n-1] \setminus \overline{\e}(A)}(\nu)).
\end{align}
Combining~\cref{restriction of $S_A$} with the equality 
\[\dot\chi^{\e(A)}(\nu) = \lb \mathbbm{1}, \dfrac{\mathbbm{reg - 1}}{\nu-1}, \underline{\hspace{0.5cm}},   \mathbbm{1}, \underline{\hspace{0.5cm}}, \underline{\hspace{0.5cm}} \rb \in \cf(Q_{\overline{\e}(A)}(\nu)),\] 
we derive that 
\begin{align*}
     {\bf m}_A(\dot\chi^I(\nu), \dot\chi^J(\nu))
     & \stackrel{\text{by def.}} {=}\dot\chi^{\e(A)}(\nu) \otimes_{\overline{\e}(A)} ({\bf s}_A(\dot\chi^I(\nu), \dot\chi^J(\nu))  \big\downarrow^{Q_{m+n+1}(\nu)}_{Q_{[m+n-1] \setminus \overline{\e}(A)}(\nu)})\\
    &{\,\,\,\,}= \lb \mathbbm{1}, \dfrac{\mathbbm{reg - 1}}{\nu-1}, \mathbbm{1},   \mathbbm{1}, \mathbbm{1}, \mathbbm{1} \rb.
\end{align*}
\end{example}

\begin{definition}{\rm (coproduct)}\label{def: coproduct of scf}
Let $n$ be a nonnegative integer and $\nu$ a positive integer $>1$.

\begin{enumerate}[label = {\rm (\alph*)}]

\item
For $0\le k \le n$, let 
\[\blacktriangle_k:\cf(Q_{n}(\nu))\to \cf(Q_k(\nu)) \otimes \cf(Q_{n-k}(\nu))\]
be the $\mathbb C$-linear map given by
\begin{equation}\label{eq: def of kth coproduct}
\blacktriangle_k(\phi)
=\begin{cases}
\mathbbm{1}_0 \otimes \phi & \text{ if } k=0, \\
\mathsf{T} (\phi \big\downarrow^{Q_n(\nu)}_{Q_{[k-1] \sqcup [k+1,n-1] }(\nu)})
 \quad & \text{ if } 1\le k< n, \\
\phi \otimes \mathbbm{1}_0 & \text{ if } k=n, 
\end{cases}
\end{equation}
Here, $\phi \in \cf(Q_n(\nu))$ and $\mathsf{T}$ represents the inverse of the $\mathbb{C}$-linear isomorphism
\begin{align*}
\cf(Q_k(\nu)) \otimes \cf(Q_{n-k}(\nu)) &\to \cf(Q_{[k-1] \sqcup [k+1,n-1]}(\nu)) \\
\phi \otimes \psi &\mapsto 
\phi \otimes_{[k-1]}
\left( \iota^*_{[k+1,n-1]}(\psi) \right).
\end{align*}

\item
We define 
\[\blacktriangle:\cf(Q_{n}(\nu))\to \bigoplus_{k\in [n-1]}\cf(Q_k(\nu)) \otimes \cf(Q_{n-k}(\nu))\]
by 
$$
\blacktriangle := 
\sum_{k=0}^n \blacktriangle_k.
$$
\end{enumerate}
\end{definition}

It is important to emphasize that if  
$\phi \big\downarrow^{Q_n(\nu)}_{Q{[k-1] \sqcup [k+1,n-1] }(\nu)}$  is written in a simple tensor form, such as
\[
\phi \big\downarrow^{Q_n(\nu)}_{Q_{[k-1] \sqcup [k+1,n-1] }(\nu)} = \phi_1 \otimes_{[k-1]} \phi_2 \quad \text{(in the notation~\cref{eq: tensor S})},
\]
then the expression for $\blacktriangle_k(\phi)$ becomes very straightforward. Specifically, we have
\begin{equation}\label{when simple tensor}
\blacktriangle_k(\phi) = \phi_1 \otimes (\iota^\ast_{[k+1,n-1]})^{-1} (\phi_2),
\end{equation}
where $(\iota^\ast_{[k+1,n-1]})^{-1}$ denotes the inverse image of $\phi_2$ under the map $\iota^\ast_{[k+1,n-1]}$.

Extending ${\bf m}$ bi-additively
and $\blacktriangle$ additively, 
we obtain a $\mathbb C$-bilinear map
\begin{align*}
{\bf m}:\bigoplus_{n \ge 0} \cf(Q_n(\nu))\otimes
\bigoplus_{n \ge 0}\cf(Q_n(\nu)) \to \bigoplus_{n \ge 0}\cf(Q_n(\nu))
\end{align*}
and a $\mathbb C$-linear map
\begin{align*}
\blacktriangle:\bigoplus_{n \ge 0}\cf(Q_n(\nu)) \to \bigoplus_{n \ge 0}\cf(Q_n(\nu)) \otimes \bigoplus_{n \ge 0} \cf(Q_n(\nu)).
\end{align*}

We will show that ${\bf m}$ and $\blacktriangle$ restrict to the supercharacter function spaces.
Recall that for $S=\{s_1<s_2< \cdots < s_{|S|} \}$ and $I\subseteq [|S|]$,
the symbol $I_S$ represents $\{ s_i \mid i \in I \}$.

\begin{definition}\label{def: I preshuffle J}
Let $m$ and $n$ be nonnegative integers
and let $A \in \binom{[m+n]}{n}$, $I \subseteq [m-1]$, and $J \subseteq [n-1]$.   

\begin{enumerate}[label = {\rm (\alph*)}]
\item
Let 
\begin{align*}
    (I \#_A J)' := {I'}_{A^{\mathrm c}} \sqcup {J'}_{A} \quad (\subseteq [m+n]).
\end{align*}
Here, $I', J'$ are equal to $I, J$ as sets, respectively, but they differ in the aspect that $I' \subseteq [m]$ and $J' \subseteq [n]$.

\item
{\it The $A$-preshuffle $I \#_A J$ of $I$ and $J$} is defined by the subset of $[m+n-1]$ 
with the same elements as $(I \#_A J)'$.

\item 
Let
$$I \shuffle_A J := \e(A) \sqcup ((I \#_A J) \setminus \overline{\e}(A)).$$
\end{enumerate}
\end{definition}

For $S =\{s_1, s_2, \ldots, s_{|S|}\} \subseteq [k+1, n-1]$, let $S-k$ be the translation of $S$ by $-k$, i.e.,
\[
S-k=\{s_1 -k, s_2 -k, \ldots, s_{|S|} -k \} \quad (\subseteq [n-k+1]).
\]
With this notation together with~\cref{def: I preshuffle J}, we can derive the following property.
\begin{lemma}\label{thm: categorification of QSym 0}
For any $\nu>1$ integer, we have the following.

\begin{enumerate}[label = {\rm (\alph*)}]
\item 
Let $m$ and $n$ be nonnegative integers
and $A \in \binom{[m+n]}{n}$. 
For $I\subseteq [m-1]$ and $J\subseteq [n-1]$,  
we have 
\[
{\bf m}_A \left( \dot\chi^I(\nu), \dot\chi^J(\nu) \right) = \dot\chi^{I \shuffle_A J}(\nu).\]

\item 
Let $n\ge 2$ and $1\le k\le n-1$. For $I\subseteq [n-1]$, we have 
\[
\blacktriangle_k(\dot\chi^I(\nu)) =
\dot\chi^{I \cap [k-1]}(\nu) \otimes 
\dot\chi^{ (I \cap [k+1,n-1]) - k}(\nu).\]
\end{enumerate}
\end{lemma}

\begin{proof}
(a) 
Let
\begin{align*}
\dot\chi^I(\nu)_i := 
\begin{cases}
    \mathbbm{1}  & \text{if } i \in I,\\
    \dfrac{\mathbbm{reg-1}}{\nu-1}             & \text{if } i \in [n-1] \setminus I.
\end{cases}
\end{align*}
Since $\chi^I(\nu)(\boldsymbol{0}) = (\nu-1)^{|I^{\mathrm c}|}$, it follows from~\cref{thm: supercharacter calculation} that   
\[\lb\left(\dot\chi^I(\nu)_i\right)_{i\in [n-1]}\rb= \dot\chi^I(\nu).\]
Let 
\begin{align*}
   {\bf s}_A(\phi, \psi)_i := 
\begin{cases}
    \mathbbm{1}  & \text{if } i \in (I \#_A J)',\\
    \dfrac{\mathbbm{reg-1}}{\nu-1}              & \text{otherwise}
\end{cases}
\end{align*}
for all $i \in [m+n]$.
Combining~\cref{def: I preshuffle J} with~\cref{def of s(phi psi)}, 
we derive that 
\[\lb \left( {\bf s}_A(\phi, \psi)_i \right)_{i \in [m+n]} \rb = {\bf s}_A(\phi, \psi).\]
Next, for each $i \in [m+n-1] \setminus \overline{\e}(A)$, we let 
\begin{align*}
    \left({{\bf s}_A(\phi, \psi) \big\downarrow^{Q_{m+n+1}}_{Q_{[m+n-1] \setminus \overline{\e}(A)}}} \right)_i := 
\begin{cases}
    \mathbbm{1}  & \text{if } i \in (I \#_A J) \setminus \overline{\e}(A), \\
    \dfrac{\mathbbm{reg-1}}{\nu-1}             & \text{otherwise}.
\end{cases}
\end{align*}
In view of
$$
\frac{(\mathbbm{reg - 1})(0)}{\nu-1} = \frac{\nu-1}{\nu-1} = 1 \quad \text{and} \quad \mathbbm{1}(0) = 1,
$$
we have 
\[\lb ((({\bf s}_A(\phi, \psi)) \big\downarrow^{Q_{m+n+1}}_{Q_{[m+n-1] \setminus \overline{\e}(A)}})_i)_{i \in [m+n-1] \setminus \overline{\e}(A)} \rb = {\bf s}_A(\phi, \psi) \big\downarrow^{Q_{m+n+1}}_{Q_{[m+n-1] \setminus \overline{\e}(A)}}.\]
Finally, let 
\begin{align*}
    {\bf m}_A(\dot\chi^I(\nu), \dot\chi^J(\nu))_i :=
\begin{cases}
    \mathbbm{1}  & \text{if } i \in \e(A) \sqcup ((I \#_A J) \setminus \overline{\e}(A)), \\
    \dfrac{\mathbbm{reg-1}}{\nu-1}              & \text{otherwise}
\end{cases}
\end{align*}
for $i \in [m+n-1]$. 
It holds that 
\begin{align*}
    \lb \left({\bf m}_A(\dot\chi^I(\nu), \dot\chi^J(\nu))_i\right)_{i \in [m+n-1]} \rb = 
    \dot\chi^{\e(A)}(\nu) \otimes_{\overline{\e}(A)} 
\left( {\bf s}_A(\phi, \psi)  \big\downarrow^{Q_{m+n+1}(\nu)}_{Q_{[m+n-1] \setminus \overline{\e}(A)}(\nu)} \right).
\end{align*}
Now the desired result follows from~\cref{def: product of scf} (a) and~\cref{def: I preshuffle J} (c).

(b) 
It is immediate to show that
\[
\dot\chi^I(\nu)\big\downarrow^{Q_n(\nu)}_{Q_{[k-1] \sqcup [k+1,n-1] }(\nu)} =
\dot\chi^{I \cap [k-1]}(\nu) \otimes_{[k-1]} 
\dot\chi^{ (I \cap [k+1,n-1])}(\nu).
\]
Hence the desired result follows from the identity
$$(\iota^\ast_{[k+1,n-1]})^{-1} (\dot\chi^{ (I \cap [k+1,n-1])}(\nu)) = \dot\chi^{ (I \cap [k+1,n-1]) - k}(\nu).$$
\end{proof}

\begin{example} \label{2 example for multiplication}
Let us revisit~\cref{example for multiplication}. 
We have
\[
(I')_{A^{\mathrm c}} = \{5,6 \} \subseteq A^{\mathrm c} \quad \text{and} \quad J'_A = \{3 \} \subseteq A,\]
thus
\[
(I \#_A J)' = \{3,5,6 \} \subseteq [7] 
\quad \text{ and }  \quad 
I \#_A J = \{3,5,6 \} \subseteq [6].
\]
Since $\e(A) = \{1,4 \}, \e(A^{\mathrm c}) = \{2 \}$, and $\overline{\e}(A) = \{1,2,4 \}$, we have that 
$
I \shuffle_A J = \{1,3,4,5,6 \}.
$
Consequently,
\[
\dot\chi^{I \shuffle_A J}(\nu) = \lb \mathbbm{1}, \dfrac{\mathbbm{reg - 1}}{\nu-1}, \mathbbm{1},   \mathbbm{1}, \mathbbm{1}, \mathbbm{1} \rb
\]
which is equal to ${\bf m}_A(\dot\chi^I(\nu), \dot\chi^J(\nu))$ as illustrated in~\cref{example for multiplication}.
\end{example}

\begin{example}
Let $n=5, I = \{1,3,4 \} \subseteq [4]$.
Then 
$$
\dot\chi^I(\nu) = \lb \mathbbm{1}, \dfrac{\mathbbm{reg - 1}}{\nu-1}, \mathbbm{1}, \mathbbm{1} \rb,
$$
thus 
\begin{align*}
  \blacktriangle_0(\dot\chi^I(\nu) ) &= \mathbbm{1}_0 \otimes \lb \mathbbm{1}, \dfrac{\mathbbm{reg - 1}}{\nu-1}, \mathbbm{1},  \mathbbm{1} \rb, & \blacktriangle_3(\dot\chi^I(\nu) ) &= \lb \mathbbm{1}, \dfrac{\mathbbm{reg - 1}}{\nu-1} \rb \otimes \lb \mathbbm{1} \rb,\\
  \blacktriangle_1(\dot\chi^I(\nu) ) &= \mathbbm{1}_1 \otimes \lb \dfrac{\mathbbm{reg - 1}}{\nu-1}, \mathbbm{1}, \mathbbm{1} \rb, &\blacktriangle_4(\dot\chi^I(\nu) ) &= \lb \mathbbm{1}, \dfrac{\mathbbm{reg - 1}}{\nu-1}, \mathbbm{1} \rb \otimes \mathbbm{1}_1,\\
  \blacktriangle_2(\dot\chi^I(\nu) ) &= \lb \mathbbm{1} \rb \otimes \lb \mathbbm{1}, \mathbbm{1} \rb,
   &\blacktriangle_5(\dot\chi^I(\nu) ) &= \lb \mathbbm{1}, \dfrac{\mathbbm{reg - 1}}{\nu-1}, \mathbbm{1}, \mathbbm{1} \rb \otimes \mathbbm{1}_0.
\end{align*}
\end{example}

\cref{thm: categorification of QSym 0} implies that the maps ${\bf m}$ and $\blacktriangle$ restrict to the supercharacter function spaces. Thus we have  
\begin{equation}\label{def of product and coproduct}
\begin{aligned}
{\bf m}:\bigoplus_{n \ge 0} \scf(\mathcal{N}_n(\nu))\otimes
\bigoplus_{n \ge 0}\scf(\mathcal{N}_n(\nu)) \to \bigoplus_{n \ge 0}\scf(\mathcal{N}_n(\nu)),\\
\blacktriangle:\bigoplus_{n \ge 0}\scf(\mathcal{N}_n(\nu)) \to \bigoplus_{n \ge 0}\scf(\mathcal{N}_n(\nu)) \otimes \bigoplus_{n \ge 0} \scf(\mathcal{N}_n(\nu)).
\end{aligned}
\end{equation}

For each nonnegative integer $n$, define 
\begin{align*}
    \ch^n_\nu : \scf(\mathcal{N}_n(\nu)) \to \QSym_n
\end{align*}
by the $\mathbb C$-vector space isomorphism given by 
\[\dot\chi^I(\nu)\mapsto L_{\comp(I)} \quad (I \subset [n-1]).\] 
This induces a $\mathbb C$-vector space isomorphism:
\begin{align*}
    \ch_\nu :=\bigoplus_{n \ge 0}\ch^n_\nu: \bigoplus_{n \ge 0} \scf(\mathcal{N}_n(\nu)) \to \QSym
\end{align*}

\begin{lemma}\label{lem: categorification of QSym}
For nonnegative $m$ and $n$ and $\nu>1$, the following hold.

\begin{enumerate}[label = {\rm (\alph*)}]
\item
For $I\subseteq [m-1]$ and $J\subseteq [n-1]$, we have
\begin{align*}
    \ch_\nu ({\bf m}(\dot\chi^I(\nu),\dot\chi^J(\nu)))=L_{\comp(I)} L_{\comp(J)}. 
\end{align*}

\item 
For $I\subseteq [n-1]$, we have
\begin{align*}
    \ch_\nu \otimes \ch_\nu (\blacktriangle(\dot\chi^I(\nu)))=
    \triangle L_{\comp(I)},
\end{align*}
where $\triangle$ denotes the coproduct of $\QSym$.
\end{enumerate}
\end{lemma}

\begin{proof}
(a) By~\cref{thm: categorification of QSym 0} (a), the assertion can be verified by showing that 
\begin{equation}\label{eq: L-product and subset parameter}
L_{\comp(I)} L_{\comp(J)} = \sum_{A \in \binom{[m+n]}{n}} L_{\comp(I \shuffle_A J)}.
\end{equation}
Let $u \in \SG_m$ be a permutation with $\Des(u) = I$ and 
$v \in \SG_n$ be a permutation with $\Des(v) = J$. 
Due to~\cref{eq: L-product and set parameter}, the identity~\cref{eq: L-product and subset parameter} is equivalent to 
\begin{equation}\label{eq: L-product and set parameter2}
\sum_{w \in u \shuffle v[m]} L_{\comp(\Des(w))} = \sum_{A \in \binom{[m+n]}{n}} L_{\comp(I \shuffle_A J)}.
\end{equation}
Consider the bijection
$$
\Theta:\textsf{Sh}_{m,n} \to \binom{[m+n]}{n}, \quad \sigma \mapsto \{ \sigma^{-1}(m+1), \sigma^{-1}(m+2), \ldots, \sigma^{-1}(m+n) \}.
$$
Fix $A \in \binom{[m+n]}{n}$.
Letting   
$\shuffle_A:=\shuffle_{\Theta^{-1}(A)}$,
the multiset $w \shuffle w'$ of shuffles of 
a word $w$ of length $m$ and a word $w'$ of length $n$ can be rewritten as  
\begin{align}\label{eq: shuffle of words and subset A}
\left\{ w \shuffle_A w' \mid A \in \binom{[m+n]}{n} \right\}.
\end{align}
Under the bijection $\Theta$, the word $u \shuffle_A v[m]$ can be obtained from $u$ and $v[m]$ in the following steps: 

{\it Step 1.} Place the entries of $v[m]$ in order in the positions occupied by $A$.

{\it Step 2.} Place the entries of $u$ in order in the positions occupied by $A^{\mathrm c}$.

\noindent
This shows that 
\begin{align}\label{eq: eqeq1}
\Des(u \shuffle_A v[m]) \setminus \overline{\e}(A) = \Des(u \shuffle_A v) \setminus \overline{\e}(A).
\end{align}
For the definition of $\overline{\e}(A)$, see~\cref{definition of Int(A) and e(A)}.
On the other hand, since every entry in $v[m]$ is larger than the greatest entry in $u$
and every entry in $u$ is smaller than the smallest entry in $v[m]$, 
it follows that 
\begin{align}\label{eq: eqeq2}
 \e(A) \subseteq \Des(u \shuffle_A v[m]) \quad \text{ and } \quad 
 \e(A^{\mathrm c}) \subseteq \Des(u \shuffle_A v[m])^{\mathrm c}.
\end{align}
Combining~\cref{eq: eqeq1} and~\cref{eq: eqeq2}, we have
\begin{align}\label{eq: eqeqeqeqeq}
\Des(u \shuffle_A v[m]) = \e(A) \sqcup (\Des(u \shuffle_A v) \setminus \overline{\e}(A)).
\end{align}
In view of~\cref{def: I preshuffle J}, one sees that 
$
\Des(u \shuffle_A v) = I \#_A J.
$
Plugging this identity to~\cref{eq: eqeqeqeqeq}, we finallly derive  that
\begin{align*}
\Des(u \shuffle_A v[m]) &{\,\,\,\,}= \e(A) \sqcup ((I \#_A J) \setminus \overline{\e}(A))
\stackrel{\text{by def.}} {=} I \shuffle_A J.
\end{align*}
This proves the equality~\cref{eq: L-product and set parameter2}.

(b) By~\cref{thm: categorification of QSym 0} (b), the assertion can be verified by showing that
\begin{align*}
\triangle L_{\comp(I)} 
= L_{\emptyset} \otimes L_{\comp(I)} +  
\sum_{k=1}^{n-1} L_{\comp(I \cap [k-1])} \otimes L_{\comp( (I \cap [k+1,n-1]) - k)} 
+  L_{\comp(I)} \otimes L_{\emptyset}.
\end{align*} 
Let $w \in \SG_n$ be a permutation with $\Des(w) = I$. 
For $k = 1, \ldots, n-1$, it is easy to see that 
\begin{align*}
    \Des( \std(w_1 \ldots w_k) ) &= I \cap [k-1] \,\,\,(\subseteq [k-1]) \quad \text{ and } \\
    \Des( \std(w_{k+1} \ldots w_{n}) ) &= I \cap [k+1, n-1] \,\,- k \,\,\,(\subseteq [n-k-1]).
\end{align*}
Now, the desired result can be obtained by applying these equalities to~\cref{eq: L-coproduct}. 
\end{proof}

\begin{example}\label{ex: shuffle of I, J}
Let us revisit 
~\cref{example for multiplication}.
Take $u = 1432$, $v = 132$.
In the~\cref{2 example for multiplication}, we see that $\e(A) = \{1,4\}$, $\e(A^{\mathrm c}) = \{ 2 \}$, and $I \#_A J =\{3, 5,6 \}$. As a consequence, we have 
$$
u \shuffle_A v[m] = \{1,4\} \sqcup (\{3,5,6 \} \setminus \{1,2,4\})
= \{1,3,4,5,6 \}
$$
which is equal to $I \shuffle_A J$.
\end{example}

By~\cref{lem: categorification of QSym}, we have the following commuting diagrams:
\[\begin{tikzcd}
	\bigoplus_{n \ge 0}\scf(\mathcal{N}_n(\nu))\otimes \bigoplus_{n \ge 0}\scf(\mathcal{N}_n(\nu)) & \bigoplus_{n \ge 0}\scf(\mathcal{N}_n(\nu)) \\
	\QSym \otimes \QSym & \QSym \\
	\bigoplus_{n \ge 0}\scf(\mathcal{N}_n(\nu))\otimes \bigoplus_{n \ge 0}\scf(\mathcal{N}_n(\nu)) & \bigoplus_{n \ge 0}\scf(\mathcal{N}_n(\nu)) \\
	\QSym \otimes \QSym & \QSym
	\arrow[from=1-1, to=1-2, "{\bf m}"]
	\arrow[from=1-2, to=2-2, "\ch_\nu"]
	\arrow[from=2-1, to=2-2, "\cdot"]
	\arrow[from=1-1, to=2-1, "\ch_\nu \otimes \ch_\nu"]
	\arrow[from=3-2, to=3-1, "\blacktriangle"]
	\arrow[from=4-2, to=4-1, "\triangle"]
	\arrow[from=3-2, to=4-2, "\ch_\nu"]
	\arrow[from=3-1, to=4-1, "\ch_\nu \otimes \ch_\nu"]
\end{tikzcd}\]

Now we are ready to state the main result of this section.

\begin{theorem}\label{thm: categorification of QSym}
For any integer $\nu>1$, we have the following.

\begin{enumerate}[label = {\rm (\alph*)}]
\item
$(\bigoplus_{n \ge 0} \scf(\mathcal{N}_n(\nu)), {\bf m}, \blacktriangle)$
has a Hopf algebra structure.

\item
The map 
\begin{align*}
    \ch_\nu : \bigoplus_{n \ge 0} \scf(\mathcal{N}_n(\nu)) \to \QSym
\end{align*}
is an isomorphism of Hopf algebras.
\end{enumerate}
\end{theorem}

\begin{proof}
By~\cref{lem: categorification of QSym}, $\ch_\nu$ is a bialgebra isomorphism.
As $\QSym$ is a connected graded Hopf algebra, this proves our assertions.
\end{proof}

\begin{remark}
(a) \cref{thm: categorification of QSym} remains still valid 
even if $C_\nu$ is replaced by any finite group $G$ of order $\nu$, 
In this case,  
$Q_S(\nu) = \bigoplus_{s \in S} G_{s}$, where $G_{s} =G$ for all $s \in S$. 

(b) In~\cite{AT21-Nsym}, the vector space $\bigoplus_{n \ge 0} \scf(\mathcal{N}_n(\nu))$, equipped with a suitable product and coproduct, is identified with the Hopf algebra $\NSym$ of noncommutative symmetric functions.
Since $\QSym$ is the graded dual Hopf algebra of $\NSym$, it is possible, in principle, to obtain a categorification of $\QSym$ on the same supercharacter function space.
This can be achieved by considering Frobenius duals of the product and coproduct that were employed in the categorization of $\NSym$ and vice versa.
In this regard, our identification (\cref{thm: categorification of QSym}) is the dual counterpart of an example explored at the end of~\cite[Section 4.2]{AT21-Nsym}.
\end{remark}

\subsection{The superclass identifiers of $\bigoplus_{n\ge 0}\scf(\mathcal{N}_n(\nu))$}
\label{sec:The superclass identifiers}
The purpose of this subsection is to 
study the superclass identifiers of $\bigoplus_{n\ge 0}\scf(\mathcal{N}_n(\nu))$.
To be precise, for $I \subseteq [n-1]$ and $J \subseteq [m-1]$, 
we expand 
${\bf m}(\kappa_I(\nu), \kappa_J(\nu) )$ in the basis $\{\kappa_I(\nu) \mid I \subseteq [m+n-1]\}$ of $\scf(\mathcal{N}_{m+n}(\nu))$  
and $\blacktriangle(\kappa_I(\nu))$ in the basis 
of $\bigoplus_{k\in [n-1]}\scf(\mathcal{N}_k(\nu)) \otimes \scf(\mathcal{N}_{n-k}(\nu))$ consisting of tensor products of the superclass identifiers.
This results obtained here will play a crucial role in~\cref{Section: new bases for NSym}.

\begin{lemma}\label{lem: vector notation of superclass identifier}
For $I \subseteq [n-1]$ and integer $\nu>1$, let 
\[
\kappa_I(\nu)_i :=
\begin{cases}
\mathbbm{1} - \frac{1}{\nu} \mathbbm{reg} &\text{ if } i \in I,\\
\frac{1}{\nu} \mathbbm{reg} &\text{ if } i \in [n-1]\setminus I
\end{cases}
\]
for $i \in [n-1]$. Then $\lb ({\kappa_I(\nu)}_i)_{i \in [n-1]} \rb = \kappa_I(\nu)$.
\end{lemma}

\begin{proof}
The assertion follows from 
$
\frac{1}{\nu}\,\mathbbm{reg}(g)  =
\begin{cases}
1 &\text{ if } g = 0,\\
0 &\text{ otherwise }.
\end{cases}
$
\end{proof}

Let us first introduce the expansion of  
${\bf m}(\kappa_I(\nu), \kappa_J(\nu) )$ in the basis 
consisting of the superclass identifiers.

\begin{proposition}\label{thm: product formula of superclass identifier}
Let $I \subseteq [m-1], J \subseteq [n-1]$, and $\nu$ be a positive integer $>1$. Then 
\[
{\bf m}(\kappa_I(\nu), \kappa_J(\nu) ) = \sum_{K \subseteq [m+n-1]} d_K\,\kappa_K(\nu),
\]
where 
\[d_K=\sum_{\substack{A \in \binom{[m+n]}{n}:\\ (I \#_A J) \cap \overline{\e}(A) = \emptyset \\ I \#_A J \subseteq K \subseteq (I \#_A J ) \cup \overline{\e}(A) }} \left(\dfrac{1}{1-\nu}\right)^{|K \cap \e(A^{\mathrm c})|} .\]
\end{proposition}

\begin{proof}
Observe the following equalities:
\begin{align}\label{eq: 1 decomposition}
    \mathbbm{1} = \left(\mathbbm{1} - \frac{1}{\nu} \mathbbm{reg}\right) + \frac{1}{\nu} \mathbbm{reg},
\end{align}
\begin{align}\label{eq: reg-1/reg-1 decomposition}
    \frac{\mathbbm{reg} - \mathbbm{1}}{\nu-1} = \left(\frac{-1}{\nu-1}\right)  \left(\mathbbm{1} - \frac{1}{\nu} \mathbbm{reg}\right) + \frac{1}{\nu} \mathbbm{reg}.
\end{align}
For $A \subseteq [m+n]$, 
set 
\[
z_A:= 
\begin{cases}
\max A & \text{ if } m+n \notin A,\\
\max \,\,[m+n]\setminus A & \text{ if } m+n \in A.
\end{cases}
\]
If we let
\[
\phi_i : = 
\begin{cases}
\mathbbm{1} - \frac{1}{\nu} \mathbbm{reg} &\text{ if } i \in (I \#_A J)', \\
\frac{\mathbbm{reg} - \mathbbm{1}}{\nu-1} &\text{ if } i = z_A \text{ or } i= m+n,\\
\frac{1}{\nu} \mathbbm{reg} &\text{ otherwise,}
\end{cases}
\]
then 
\[
\lb (\phi_i)_{i \in [m+n]} \rb =
{\bf s}_A(\kappa_I(\nu), \kappa_J(\nu)).
\]
Applying~\cref{eq: reg-1/reg-1 decomposition} 
to the $\phi_{z_A}$ in this vector notation and then 
using~\cref{lem: vector notation of superclass identifier}, 
one can simplify 
\begin{equation*}
({\bf s}_A(\kappa_I(\nu), \kappa_J(\nu))) \big\downarrow^{Q_{m+n+1}(\nu)}_{Q_{[m+n-1] \setminus \overline{\e}(A)}(\nu)}
\end{equation*}
in the following simple form:
\begin{equation}\label{tensor product for multiplication of identifier}
\left( \left(\frac{-1}{\nu-1}\right)   \kappa_{(I \#_A J) \sqcup \{z_A\}}(\nu) + \kappa_{I \#_A J}(\nu) \right) \Big\downarrow^{Q_{m+n+1}(\nu)}_{Q_{[m+n-1] \setminus \overline{\e}(A)}(\nu)}
\end{equation}
On the other hand, one can easily see that for any subsets $S,T \subseteq [n-1]$, 
\begin{align}\label{eq: restriction of superclass identifier}
\kappa_S(\nu) \big\downarrow^{Q_n(\nu)}_{Q_T(\nu)} = 
\begin{cases}
\kappa_{S}(\nu) &\text{ if } S \subseteq T,\\
0 &\text{ otherwise }.
\end{cases}
\end{align}
Since $z_A \in \overline{\e}(A)$, from~\cref{eq: restriction of superclass identifier} it follows that 
\[
\eqref{tensor product for multiplication of identifier}=
\begin{cases}
\kappa_{I \#_A J}(\nu) &\text{ if } I \#_A J \subseteq [m+n-1] \setminus \overline{\e}(A),\\
0 &\text{ otherwise }.
\end{cases}
\]
Finally, using~\cref{eq: 1 decomposition} and~\cref{eq: reg-1/reg-1 decomposition}, we derive that 
\begin{align*}
&{\bf m}_A(\kappa_I(\nu), \kappa_J(\nu) )\\ 
&=
\begin{cases}
\dot\chi^{\e(A)}(\nu) \otimes_{\overline{\e}(A)} \kappa_{I \#_A J}(\nu) &\text{ if } I \#_A J \subseteq [m+n-1] \setminus \overline{\e}(A),\\
0 &\text{ otherwise }
\end{cases}\\
&=
\begin{cases}
\sum_{I \#_A J \subseteq K \subseteq (I \#_A J) \sqcup \overline{\e}(A) } \left(\dfrac{1}{1-\nu}\right)^{|K \cap \e(A^{\mathrm c})|}  \kappa_{K}(\nu) &\text{ if } I \#_A J \subseteq [m+n-1] \setminus \overline{\e}(A),\\
0 &\text{ otherwise }.
\end{cases}
\end{align*}
This completes the proof.
\end{proof}

\begin{example}
Let $m = 2, n = 3$ and $I = \{1\}, J = \{2\}$.
The following table shows all the statistics required to calculate 
${\bf m}(\kappa_I(\nu), \kappa_J(\nu))$.

\newcolumntype{C}[1]{>{\centering\arraybackslash}p{#1}}

\makeatletter
\def\hlinewd#1{%
\noalign{\ifnum0=`}\fi\hrule \@height #1 \futurelet
\reserved@a\@xhline}
\makeatother

\newcolumntype{?}{!{\vrule width 1pt}}


\begin{table}[ht]
\small{
\begin{tabular}
{|C{1.5cm}?C{1.5cm}|C{1.5cm}|C{1.5cm}|C{1.5cm}|C{2.4cm}|C{2.4cm}|}
\hline
& & & & & & \\[-1em]
 \textbf{$A$}  & \textbf{$\e(A)$}       & \textbf{$\e(A^{\mathrm c})$}        &  \textbf{$\overline{\e}(A)$}  &  \textbf{$I \#_A J$}  &  \textbf{$(I \#_A J) \cap \overline{\e}(A)$}  &  \textbf{$(I \#_A J) \cup \overline{\e}(A)$}  \\[-1em]
& & & & & &  \\ \hlinewd{1pt}
 \{1,2,3\}   & \{3\}                    & $\emptyset$ & \{3\}           &  \{2,4\}              & $\emptyset$                                           & \{2,3,4\}                      \\ \hline
\{1,2,4\}    & \{2,4\}                 & \{3\}                    & \{2,3,4\}       & \{2,3\}             & $\{2,3\}$                                           & unnecessary                                \\ \hline
\{1,2,5\}    & \{2\}                    & \{4\}                    & \{2,4\}         & \{2,3\}             & $\{2\}$                                           & unnecessary                                \\ \hline
\{1,3,4\}    & \{1,4\}                  & \{2\}                    & \{1,2,4\}       & \{2,3\}             & $\{2\}$                                           & unnecessary                                \\ \hline
\{1,3,5\}    & \{1,3\}                  & \{4\}                    & \{1,3,4\}       & \{2,3\}             & $\{3\}$                                           &  unnecessary                               \\ \hline
\{1,4,5\}    & \{1\}                    & \{3\}                    & \{1,3\}         & \{2,4\}             & $\emptyset$                                           & \{1,2,3,4\}                     \\ \hline
\{2,3,4\}    & \{4\}                    & \{1\}                    & \{1,4\}         & \{1,3\}             & $\{1\}$                                           & unnecessary                                \\ \hline
\{2,3,5\}    & \{3\}                    & \{1,4\}                  & \{1,3,4\}       & \{1,3\}             & $\{1,3\}$                                           &  unnecessary                               \\ \hline
\{2,4,5\}    & \{2\}                    & \{1,3\}                  & \{1,2,3\}       & \{1,4\}             & $\{1\}$                                           & unnecessary                                \\ \hline
\{3,4,5\}    & $\emptyset$ & \{2\}                    & \{2\}           & \{1,4\}             & $\emptyset$                                           & \{1,2,4\}                       \\ \hline
\end{tabular}}
\end{table}

\noindent
From this, we have
\begin{align*}
{\bf m}(\kappa_I(\nu), \kappa_J(\nu) ) &= \left(\dfrac{1}{1-\nu}\right)^{0}  \kappa_{\{1,4 \}}(\nu) + \left\{\left(\dfrac{1}{1-\nu}\right)^{0}+\left(\dfrac{1}{1-\nu}\right)^{0}\right\}  \kappa_{\{2,4 \}}(\nu)\\
&+ \left\{\left(\dfrac{1}{1-\nu}\right)^{0}+\left(\dfrac{1}{1-\nu}\right)^{1}\right\}  \kappa_{\{1,2,4 \}}(\nu)\\
&+ \left\{\left(\dfrac{1}{1-\nu}\right)^{0}+\left(\dfrac{1}{1-\nu}\right)^{1}\right\}  \kappa_{\{2,3,4 \}}(\nu)
+\left(\dfrac{1}{1-\nu}\right)^{1}  \kappa_{\{1,2,3,4 \}}(\nu)\\
&= \kappa_{\{1,4 \}}(\nu) + 2 \kappa_{\{1,4 \}}(\nu) + \left(\dfrac{2-\nu}{1-\nu}\right)  \kappa_{\{1,2,4 \}}(\nu) + \left(\dfrac{2-\nu}{1-\nu}\right)  \kappa_{\{1,3,4 \}}(\nu)\\
&+ \left(\dfrac{1}{1-\nu}\right)  \kappa_{\{1,2,3,4 \}}(\nu).
\end{align*}

\end{example}

The next result concerns the expansion of $\blacktriangle(\kappa_I(\nu))$ in the basis $\{\kappa_I(\nu) \}$.
\begin{proposition}\label{thm: coproduct formula of superclass identifier}
Let $n$ be a nonnegative integer and $\nu$ be a positive integer $>1$.
For $\gamma \in \Comp_n$, we have
\[
\blacktriangle(\kappa_{\set(\gamma)}(\nu)) = \sum_{\alpha \odot \beta = \gamma} \kappa_{\set(\alpha)}(\nu) \otimes \kappa_{\set(\beta)}(\nu).
\]
\end{proposition}

\begin{proof}
Let us fix $\gamma \in \Comp_n$.
Let 
$P := \{(\alpha, \beta) \mid \alpha \cdot \beta = \gamma \}$ and 
$Q := \{(\alpha, \beta) \mid \alpha \odot \beta = \gamma \}$.
Then
$P \cap Q = \{ (\emptyset, \gamma),(\gamma, \emptyset)  \}$.
By~\cref{thm: categorification of QSym} and~\cref{eq: coproduct of L 2}, we have 
\begin{equation}\label{coproduct of chi(nu)}
\begin{aligned}
\blacktriangle(\dot\chi^{\set(\gamma)}(\nu)) 
&= \sum_{\alpha \cdot \beta = \gamma \text{ or } \alpha \odot \beta = \gamma} \dot\chi^{\set(\alpha)}(\nu) \otimes \dot\chi^{\set(\beta)}(\nu)\\
&= \sum_{(\alpha, \beta) \in P \cup Q} \dot\chi^{\set(\alpha)}(\nu) \otimes \dot\chi^{\set(\beta)}(\nu).
\end{aligned}
\end{equation}
For simplicity, set $I := \set(\gamma)$.
Combining~\cref{thm: categorification of QSym 0} (b) with~\cref{coproduct of chi(nu)}, it can be easily seen that $P \cup Q$ is equal to  
\[
\{( \comp(I \cap [k-1]), \comp((I \cap [k+1, n-1]) - k) ) \mid 1 \le k \le n-1 \} \sqcup \{ (\emptyset, \gamma),(\gamma, \emptyset)  \}.
\]
For $1 \le k \le n-1$, it holds that  
\begin{equation}\label{eq: concatenation and near concatenation in set}
    \comp(I)
    =\begin{cases}
    \comp(I \cap [k-1]) \odot \comp((I \cap [k+1, n-1]) - k) & \text{if }k \notin I,\\
    \comp(I \cap [k-1]) \cdot \comp((I \cap [k+1, n-1]) - k) & \text{if }k \in I,
    \end{cases}
\end{equation}
and which implies 
\[
Q = \{( \comp(I \cap [k-1]), \comp((I \cap [k+1, n-1]) - k) ) \mid k \notin I \} \sqcup \{ (\emptyset, \gamma),(\gamma, \emptyset)  \}.
\]
On the other hand, by~\cref{eq: restriction of superclass identifier}, 
we have that 
\begin{align}\label{eq: kappa coproduct}
\blacktriangle_k(\kappa_I(\nu)) 
&= 
\begin{cases}
\kappa_{I \cap [k-1]}(\nu) \otimes (\iota^\ast_{[k+1,n-1]})^{-1} (\kappa_{I \cap [k+1, n-1]}(\nu)) &\text{ if } k \notin I,\\
0 &\text{ if } k \in I
\end{cases}\\
&=
\begin{cases}
\kappa_{I \cap [k-1]}(\nu) \otimes \kappa_{I \cap [k+1, n-1] -k}(\nu) &\text{ if } k \notin I,\\
0 &\text{ if } k \in I
\end{cases}
\end{align}
for $1 \le k \le n-1$.
If $k=0$ or $n$, then
$\blacktriangle_0(\phi)
=\mathbbm{1}_0 \otimes \phi \text{ and } \blacktriangle_n(\phi)
=\phi \otimes \mathbbm{1}_0,
$
so 
\begin{align*}
&\blacktriangle_0(\kappa_{\set(\gamma)}(\nu))
=\kappa_{\set(\emptyset)}(\nu) \otimes \kappa_{\set(\gamma)}(\nu) \text{ and }\\ &\blacktriangle_n(\kappa_{\set(\gamma)}(\nu))
=\kappa_{\set(\gamma)}(\nu) \otimes \kappa_{\set(\emptyset)}(\nu).
\end{align*}
Putting these together, we conclude that 
\[
\blacktriangle(\kappa_I(\nu)) = \sum_{k=0}^{n}\blacktriangle_k(\kappa_I(\nu)) = \sum_{(\alpha, \beta) \in Q} \kappa_{\set(\alpha)}(\nu) \otimes \kappa_{\set(\beta)}(\nu).
\]
\end{proof}

\begin{example}
Let $\gamma = (1,3,2) = \comp(\{1,4\})$. 
All the possible ways to write $\gamma$ as $\alpha \odot \beta$ are 
\[\emptyset \odot (1,3,2) = (1,1) \odot (2,2) = (1,2) \odot (1,2) = (1,3,1) \odot (1) = (1,3,2) \odot \emptyset.\]
Therefore 
$\blacktriangle(\kappa_{\{1,4\}})$ is equal to 
 \begin{align*}
  \mathbbm{1}_0 \otimes \kappa_{\{1,4\}}(\nu) +
    \kappa_{\{1\}}(\nu) \otimes \kappa_{\{2\}}(\nu) 
    +\kappa_{\{1\}}(\nu) \otimes \kappa_{\{2\}}(\nu) 
  +\kappa_{\{1,4\}}(\nu) \otimes \mathbbm{1}_1 
    +\kappa_{\{1,4\}}(\nu) \otimes \mathbbm{1}_0,
\end{align*}
where 
\begin{align*}
    &\mathbbm{1}_0 \otimes \kappa_{\{1,4\}}(\nu) \in  \scf(\mathcal{N}_0(\nu)) \otimes \scf(\mathcal{N}_6(\nu)), \quad 
    \kappa_{\{1\}}(\nu) \otimes \kappa_{\{2\}}(\nu) \in  \scf(\mathcal{N}_2(\nu)) \otimes \scf(\mathcal{N}_4(\nu)),\\
    &\kappa_{\{1\}}(\nu) \otimes \kappa_{\{2\}}(\nu) \in \scf(\mathcal{N}_3(\nu)) \otimes \scf(\mathcal{N}_3(\nu)), \quad 
    \kappa_{\{1,4\}}(\nu) \otimes \mathbbm{1}_1 \in  \scf(\mathcal{N}_5(\nu)) \otimes \scf(\mathcal{N}_1(\nu)),\\
    &\kappa_{\{1,4\}}(\nu) \otimes \mathbbm{1}_0 \in  \scf(\mathcal{N}_6(\nu)) \otimes \scf(\mathcal{N}_0(\nu)).
\end{align*}

\end{example}

\section{A new basis $\{\mathcal{D}_{\alpha}(q,t)\mid \alpha \in \Comp\}$ of $\QSym_{\mathbb{C}(q,t)}$ }\label{Section: new bases for NSym}
Let $\QSym_{\mathbb{C}(q,t)}$ be the Hopf algebra of quasisymmetric functions defined over $\mathbb{C}(q, t)$, where $q$ and $t$ are commuting variables.
In this section, we introduce a basis $\{\mathcal{D}_{\alpha}(q,t)\mid \alpha \in \Comp\}$ of $\QSym_{\mathbb{C}(q,t)}$. This basis is closely connected to the superclass identifier $\kappa_I(\nu)$ under the map $\ch_{\nu}$.
Using this relationship, we derive product and coproduct rules for this basis.
As before, we assume that $n$ is any nonnegative integer.

\subsection{The definition and notable properties of $\mathcal{D}_{\alpha}(q,t)$}\label{Subsection: qt basis B for NSym}
We begin by defining $\mathcal{D}_{\alpha}(q,t)$ and introducing some of its specializations.
\begin{definition}\label{defn of B(q t)}
For $\alpha = (\alpha_1,\alpha_2,\cdots,\alpha_l) \in \Comp$, define
\[
\mathcal{D}_{\alpha}(q,t):= \sum_{\beta \succeq \alpha} q^{\ell(\beta) - n} (-t)^{\ell(\alpha)-\ell(\beta)} M_{\beta} \in \QSym_{\mathbb{C}(q,t)}.
\]
Equivalently, 
\[
\mathcal{D}_{\alpha}(q,t) = \sum_{i_1\le i_2 \le \cdots \le i_l} q^{|\{i_1,i_2,\cdots,i_l\}|-n} (-t)^{l-|\{i_1,i_2,\cdots,i_l\}|} x_{i_1}^{\alpha_1}x_{i_2}^{\alpha_2}\cdots x_{i_l}^{\alpha_l}.
\]
\end{definition}

Setting 
\[\alpha^{\mathrm c}:=\comp([n-1]\setminus \set(\alpha))\]
for all $\alpha \in \Comp_n$,
we will see that 
$\mathcal{D}_{\alpha^{\mathrm c}}(-\nu, \nu-1) = \ch_\nu ( \kappa_{\set(\alpha)}(\nu) / (\nu-1)^{|\set(\alpha)|})$ 
in~\cref{cor: D and kappa}. 
Let us fix an linear extension of the partial order $\preceq$ of $\Comp_n$.
Since the transition matrix from $\{\mathcal{D}_{\alpha}(q,t)\}$ to $\{M_{\alpha}\}$ in this order is triangular with non-zero diagonals, $\{ \mathcal{D}_{\alpha}(q,t) \mid \alpha \in \Comp \}$ is a basis of $\QSym_{\mathbb{C}(q,t)}$.

Next, we show that several notable bases of $\QSym$ can be obtained from $\mathcal{D}_{\alpha}(q,t)$ by specializations of $q, t$. For $\alpha \in \Comp$, let $ \Lambda_{\alpha}^* $ denote the dual of {\it elementary noncommutative symmetric function} $ \Lambda_{\alpha} $ of the Hopf algebra of noncommutative symmetric functions, $E_{\alpha}$ denote {\it the essential quasisymmetric function} introduced in~\cite{H15}. And, let $\eta_{\alpha}$ and $\eta_{\alpha}^{(q)}$ denote {\it the enriched monomial quasisymmetric function} and {\it the enriched $q$-monomial quasisymmetric function}, respectively, as introduced in~\cite{GV21} and~\cite{GV22,GV23-1}. All of $\{ \Lambda_{\alpha} \}, \{E_{\alpha} \},$ and $ \{\eta_{\alpha} \}$ are bases of $\QSym$ and $\{ \eta_{\alpha}^{(q)} \}$ is a basis of $\QSym_{\mathbb{C}(q)}$.

\begin{proposition}\label{prop: specializations of B(q t)}
For each $\alpha \in \Comp$, we have
\begin{enumerate}[label = {\rm (\alph*)}]
    \item $\mathcal{D}_{\alpha}(1,0) = M_{\alpha}$,
    \item $\mathcal{D}_{\alpha}(-1,1) = \Lambda^*_{\alpha}$,
    \item $\mathcal{D}_{\alpha}(1,-1) = E_{\alpha}$,
    \item $2^n \mathcal{D}_{\alpha}(2,-1) = \eta_{\alpha}$,
    \item $(q+1)^n \mathcal{D}_{\alpha}(q+1,-1) = \eta^{(q)}_{\alpha}$.
\end{enumerate}
\end{proposition}
\begin{proof}
The assertions follow from~\cref{defn of B(q t)} and the following  expansions: 
\begin{align*} 
\Lambda_{\alpha}^* &= (-1)^{n-\ell(\alpha)}\sum_{\beta \succeq \alpha} M_{\beta}&&(\text{\cite[Section 4.1]{Gelfand95}}),\\
E_{\alpha}&=\sum_{\beta \succeq \alpha} M_{\beta} &&(\text{\cite[Section 3]{H15}}),\\
\eta_{\alpha} &= \sum_{\beta \succeq \alpha} 2^{\ell(\beta)} M_{\beta} &&(\text{\cite[Definition 6]{GV21}}),\\
\eta_{\alpha}^{(q)} &= \sum_{\beta \succeq \alpha} (q+1)^{\ell(\beta)} M_{\beta} &&(\text{\cite[Proposition 3.6]{GV23-1}}). 
\end{align*}
\end{proof}

In the rest of this subsection, we study the relationship between $\mathcal{D}_{\alpha}(q,t)$ and $\kappa_I(\nu)$.
We first investigate the relation between 
the superclass identifier $\kappa_I(\nu)$ and both the fundamental quasisymmetric functions and the monomial quasisymmetric functions.
For each $\alpha \in \Comp$ and integer $\nu>1$, let
\[
\mathcal{K}_{\alpha}(\nu) := \ch_\nu\left(\frac{\kappa_{\set(\alpha)}(\nu)}{(\nu-1)^{|\set(\alpha)|}}\right).
\]
Since $\{\kappa_I(\nu) \mid I \subseteq [n-1]\}$ is a basis of
$\scf(\mathcal{N}_n(\nu))$ and $\ch_\nu$ is an isomorphism, it follows that  
$\{\mathcal{K}_{\alpha}(\nu) \mid \alpha\in \Comp_n\}$ is a basis of  $\QSym_n$.
The following lemma shows the transition matrices between   
$\{L_{\alpha} \mid \alpha\in \Comp_n\}$ and 
$\{\mathcal{K}_{\alpha}(\nu) \mid \alpha\in \Comp_n\}$.

\begin{lemma}\label{thm: transition between Pi and L} 
Let $I,J \subseteq [n-1]$ and $\nu$ be a positive integer $>1$. Then  
\begin{enumerate}[label = {\rm (\alph*)}]
\item
$
L_{\comp(I)} = \displaystyle \sum_{J} (-1)^{|J \setminus I|}  (\nu-1)^{|I \cap J|}  \mathcal{K}_{\comp(J)}(\nu) 
$ and 
\item
$
\mathcal{K}_{\comp(J)}(\nu) = \displaystyle\sum_{I} \dfrac{1}{\nu^{n-1}}  (-1)^{|J \setminus I|}  (\nu-1)^{|(I \cup J)^{\mathrm c}|}  L_{\comp(I)}.
$
\end{enumerate}
\end{lemma}

\begin{proof}
For $I,J \subseteq [n-1]$, let $\chi^I_J(\nu)$ be the coefficient of $\kappa_J(\nu)$ in $\chi^I(\nu)$, that is, 
\begin{align}\label{eq: supercharacter to superclass identifier}
\chi^I(\nu) = \sum_{J} \chi^I_J(\nu) \kappa_J(\nu).
\end{align}
~\cref{thm: supercharacter calculation} implies that  
$\chi^I_J(\nu) = (-1)^{|J \setminus I|}  (\nu-1)^{|(I \cup J)^{\mathrm c}|}.$  
Let $\langle \cdot, \cdot \rangle$ be the Hall-inner product on $\cf(Q_n(\nu))$, that is,
\[
\langle \phi, \psi \rangle = \frac{1}{|Q_n(\nu)|} \sum_{\boldsymbol{g} \in Q_n(\nu)} \phi(\boldsymbol{g}) \psi(\boldsymbol{g}^{-1}) \quad (\phi, \psi \in \cf(Q_n(\nu))).
\]
By~\cref{thm: supercharacter calculation} and~\cref{lem: vector notation of superclass identifier}, we see that  
\begin{align*}
\langle \chi^I(\nu), \chi^I(\nu)\rangle=(\nu-1)^{|I^{\mathrm c}|} \quad \text{ and }
\quad \langle \kappa_I(\nu), \kappa_I(\nu)\rangle =\dfrac {\nu^{n-1}}{(\nu-1)^{|I|}}.
\end{align*}
Letting  
$$
{\overline {\chi}}^I(\nu):=\sqrt{\dfrac{1}{(\nu-1)^{|I^{\mathrm c}|}}} \,\, \chi^I(\nu) \quad \text{and} \quad
{\overline \kappa}^I(\nu):=\sqrt{\dfrac{(\nu-1)^{|I|}}{\nu^{n-1}}} \,\, \kappa_I(\nu),
$$
one can see that $\{{\overline {\chi}}^I(\nu) \mid I \subseteq [n-1]\}$
and $\{{\overline \kappa}^I(\nu) \mid I \subseteq [n-1]\}$ are 
orthonormal bases of $\scf(\mathcal{N}_n(\nu))$.
Using this notation, 
we can rewrite~\cref{eq: supercharacter to superclass identifier} as 
\begin{align}\label{eq: supercharacter to superclass identifier1}
{\overline {\chi}}^I(\nu) = \sum_{J} {\overline {\chi}}^I_J(\nu) {\overline \kappa}_J(\nu)
\end{align}
with
$${\overline {\chi}}^I_J(\nu):=\frac{\chi^I_J(\nu)}{\sqrt{(\nu-1)^{|I^{\mathrm c}|}}} \ \sqrt{\frac{(\nu-1)^{|J|}}{\nu^{n-1}}}.$$
Since the matrix $\left({\overline \chi}^I_J(\nu)\right)_{I,J}$ is unitary, 
we have 
\begin{align}\label{eq: supercharacter to superclass identifier2}
{\overline \kappa}^J(\nu) = \sum_{I} {\overline {\chi}}^I_J(\nu) {\overline {\chi}}_I(\nu).
\end{align}
Now, our assertions can be obtained by taking $\ch_\nu$ on both sides~\cref{eq: supercharacter to superclass identifier1} and 
~\cref{eq: supercharacter to superclass identifier2}.
\end{proof}

The following lemma shows the transition matrices between $\{M_{\alpha}\}$ and $\{K_{\alpha}(\nu)\}$.

\begin{lemma}\label{lem: transition between Pi and M}
Let $I,J \subseteq [n-1]$ and $\nu$ be a positive integer $>1$. Then
\begin{enumerate}[label = {\rm (\alph*)}]
\item
$M_{\comp(I)} = \displaystyle\sum_{J: I \cup J = [n-1]} (-\nu)^{|J \setminus I|}  (\nu-1)^{|I \cap J|}  \mathcal{K}_{\comp(J)}(\nu)$ and 
\item
$
\mathcal{K}_{\comp(J)}(\nu) = \left(\dfrac{1}{1-\nu}\right)^{|J|} \displaystyle\sum_{I: I \cap J = \emptyset} \left(\frac{\nu-1}{\nu}\right)^{(n-1) - |I|}  M_{\comp(I)}.
$
\end{enumerate}
\end{lemma}

\begin{proof}
(a) In view of~\cref{thm: transition between Pi and L} (a), we have that   
\begin{align*}
    M_{\comp(I)} &= \sum_{K: I \subseteq K} (-1)^{|K \setminus I|}  L_{\comp(K)}\\
    &= \sum_{K: I \subseteq K} (-1)^{|K \setminus I|}  \left( \sum_{J} (-1)^{|J \setminus K|}  (\nu-1)^{|K \cap J|}  \mathcal{K}_{\comp(J)}(\nu) \right)\\
    &= \sum_{J} \left( \sum_{K: I \subseteq K} (-1)^{|K \setminus I|}   (-1)^{|J \setminus K|}  (\nu-1)^{|K \cap J|} \right)  \mathcal{K}_{\comp(J)}(\nu).
\end{align*}
For simplicity, we use the following abbreviations: 
\[
V:= ([n-1] \setminus (I \cup J)) \cap K \quad \text{and} \quad
W:= (J \setminus I) \cap K.
\]
Then $K = I \sqcup V \sqcup W$, and therefore our assertion follows from the calculation below: 
\begin{align*}
    &\sum_{K: I \subseteq K} (-1)^{|K \setminus I|}   (-1)^{|J \setminus K|}  (\nu-1)^{|K \cap J|} \\
    &=  \sum_{\substack{V \subseteq [n-1] \setminus (I \cup J) \\W \subseteq J \setminus I }} (-1)^{|V|+|W|}   (-1)^{|J \setminus I|- |W|}  (\nu-1)^{|I \cap J|+|W|}\\
    &= (-1)^{|J \setminus I|}  (\nu-1)^{|I \cap J|}  \sum_{\substack{V \subseteq [n-1] \setminus (I \cup J) \\W \subseteq J \setminus I }} (-1)^{|V|}  (\nu-1)^{|W|}\\
    &= (-1)^{|J \setminus I|}  (\nu-1)^{|I \cap J|}  0^{|[n-1] \setminus (I \cup J)|} \ \nu^{|J \setminus I|}.
\end{align*}
The last equality can be derived by applying 
the binomial expansion formula.

(b) Combining~\cref{thm: transition between Pi and L} (b) with 
$L_{\comp(K)} = \sum_{I: K \subseteq I} M_{\comp(I)}$, we derive that 
\begin{align*}
    \mathcal{K}_{\comp(J)}(\nu)
    &= \sum_{K} \frac{1}{\nu^{n-1}}  (-1)^{|J \setminus K|}  (\nu-1)^{|(K \cup J)^{\mathrm c}|}  \left(
    \sum_{I: K \subseteq I} M_{\comp(I)}
    \right)\\
    &= \frac{1}{\nu^{n-1}}  \sum_{I} 
    \left(
    \sum_{K: K \subseteq I}
    (-1)^{|J \setminus K|}  (\nu-1)^{|(K \cup J)^{\mathrm c}|} 
    \right) 
    M_{\comp(I)}.
\end{align*}
We use the following abbreviations: 
\[
X:= (I \cap J) \cap K \quad \text{and} \quad Y:=(I \setminus J) \cap K.
\]
Then it holds that $K = V \sqcup Y$ whenever $K \subseteq I$. 
Now, our assertion follows from the calculation below:
\begin{align*}
    &\sum_{K: K \subseteq I}
    (-1)^{|J \setminus K|}  (\nu-1)^{|(K \cup J)^{\mathrm c}|} \\
    &=
    \sum_{\substack{X \subseteq I \cap J \\Y \subseteq I \setminus J}}
    (-1)^{|J| - |X|}  (\nu-1)^{(n-1)-|J|-|Y|}\\
    &= (-1)^{|J|}  0^{|I \cap J|}  (\nu-1)^{(n-1) - |J|}  \left(1+\frac{1}{\nu-1}\right)^{|I \setminus J|}.
\end{align*}
\end{proof}

Now, we can establish the following connection between $\mathcal{D}_{\alpha}(q,t)$ and $\kappa_I(\nu)$ based on~\cref{lem: transition between Pi and M} (b).

\begin{proposition}\label{cor: D and kappa}
Let $\alpha \in \Comp$ and a positive integer $\nu>1$.
Then we have
\[
\mathcal{D}_{\alpha^{\mathrm c}}(-\nu, \nu-1) = \mathcal{K}_{\alpha}(\nu).
\]
\end{proposition}

\subsection{Product and coproduct rules for $\{\mathcal{D}_{\alpha}(q,t)\}$}\label{subsection: Structure constants of N for the basis B} 
Recall that we previously introduced the shuffles of two words (refer to~\cref{shuffle of words}). Continuing along the same lines, we will now introduce the shuffles of two compositions.
Let $\alpha = (\alpha_1, \alpha_2, \ldots, \alpha_l) \in \Comp_m$ and $\beta = (\beta_1, \beta_2, \ldots, \beta_k) \in \Comp_n$.
{\it A shuffle of $\alpha$ and $\beta$} is a composition obtained by permuting the parts of the concatenation $\alpha \cdot \beta=(c_1,c_2,\ldots,c_{l+k})$ of $\alpha$ and $\beta$ with a permutation contained in  $\textsf{Sh}_{l,k}$. For $\sigma \in \textsf{Sh}_{l,k}$, set
$\alpha \shuffle_{\sigma} \beta := (c_{\sigma(1)}, c_{\sigma(2)}, \cdots, c_{\sigma(l+k)})$,
and set 
\[
\alpha \shuffle \beta:=\{\alpha \shuffle_{\sigma} \beta: \sigma \in \textsf{Sh}_{l,k} \}.
\]
Here, $\alpha \shuffle \beta$ is a multiset that considers duplication, not a regular set.
As in~\cref{eq: shuffle of words and subset A}, we also use the notation 
$\alpha \shuffle_D \beta$ instead of $\alpha \shuffle_{\sigma} \beta$, where 
\[
D = \{\sigma^{-1}(l+1), \sigma^{-1}(l+2), \ldots, \sigma^{-1}(l+k) \},
\]
thus yielding 
\[
\alpha \shuffle \beta = \left\{ \alpha \shuffle_D \beta \mid D \in \binom{[l+k]}{k} \right\}.
\]
On the other hand, an {\it overlapping shuffle of $\alpha$ and $\beta$} is a composition obtained from a shuffle of $\alpha$ and $\beta$ by allowing the replacement of any pair of consecutive components of the form $(\alpha_i, \beta_j)$ by $\alpha_i + \beta_j$ 
(see~\cref{example of overlapping shuffles}).
Denote by $\alpha \overline{\shuffle} \beta$ the set of all 
overlapping shuffles obtained from $\alpha$ and $\beta$.
These objects are very useful in describing a product rule for  monomial quasisymmetric functions. For instance, it holds that 
\begin{align*}
M_{\alpha} M_{\beta} = \sum_{\gamma \in \alpha \overline{\shuffle} \beta} M_{\gamma}
\end{align*}
(for instance, see~\cite[Proposition 5.1.3]{GR20}).

Having these objects in mind, we introduce new combinatorial objects 
that will play an essential role in describing our product rule for $\{\mathcal{D}_{\alpha}(q,t)\}$.

\begin{definition}\label{two-way overlapping shuffle}
Let $\alpha$ and $\beta$ be compositions.
\begin{enumerate}[label = {\rm (\alph*)}]
\item
{\it A two-way overlapping shuffle of $\alpha$ and $\beta$} is an expression obtained from a shuffle of $\alpha$ and $\beta$ by allowing the replacement of commas `,' within any consecutive component pairs of the form $(\alpha_{i}, \beta_{j})$ with $+_1$ and within any consecutive component pairs of the form $(\beta_{j}, \alpha_{i})$ with $+_2$.
Denote by $\alpha \overline{\overline{\shuffle}} \beta$ the set of two-way overlapping shuffle of $\alpha$ and $\beta$.

\item
For $\gamma \in \alpha \overline{\overline{\shuffle}} \beta$, 
we define $\gamma^+$ to be the composition obtained from $\gamma$ by replacing $+_1$ and $+_2$ by the usual addition $+$.

\item
For $\gamma \in \alpha \overline{\overline{\shuffle}} \beta$, 
we define $\mathrm{c_1}(\gamma)$ (resp. $\mathrm{c_2}(\gamma)$) to be the number of occurrences of $+_1$ (resp. $+_2$) in $\gamma$.
\end{enumerate}
\end{definition}

\begin{example}\label{example of overlapping shuffles}
Let $\alpha = (2,1)$ and $\beta = (2)$. 
Then $\alpha \shuffle \beta =\{ ({\color{red}{2}},{\color{red}{1}}, {\color{blue}{2}} ),( {\color{red}{2}}, {\color{blue}{2}}, {\color{red}{1}} ),( {\color{blue}{2}}, {\color{red}{2}}, {\color{red}{1}})\}$.
For clarity, the entries coming from $\alpha$ are written in red, while those coming from $\beta$ are written in blue.
Therefore 
\begin{align*}
\alpha \overline{\shuffle} \beta = 
\{&({\color{red}{2}},{\color{red}{1}},{\color{blue}{2}}), ({\color{red}{2}},{\color{red}{1}}+{\color{blue}{2}}),\\
&({\color{red}{2}},{\color{blue}{2}},{\color{red}{1}}), ({\color{red}{2}}+{\color{blue}{2}},{\color{red}{1}}),\\
&({\color{blue}{2}},{\color{red}{2}},{\color{red}{1}})\}\\
=\{&(2,1,2), (2,2,1), (2,2,1), (4,1), (2,3) \}   
\end{align*}
In contrast, according to~\cref{two-way overlapping shuffle}, 
\begin{align*}
\alpha \overline{\overline{\shuffle}} \beta = \{&({\color{red}{2}},{\color{red}{1}},{\color{blue}{2}}), ({\color{red}{2}}, {\color{red}{1}}+_1{\color{blue}{2}}),\\ &({\color{red}{2}},{\color{blue}{2}},{\color{red}{1}}), ({\color{red}{2}}+_1{\color{blue}{2}},{\color{red}{1}}), ({\color{red}{2}},{\color{blue}{2}}+_2{\color{red}{1}}), ({\color{red}{2}}+_1{\color{blue}{2}}+_2{\color{red}{1}}),\\ &({\color{blue}{2}},{\color{red}{2}},{\color{red}{1}}), ({\color{blue}{2}}+_2{\color{red}{2}},{\color{red}{1}})\}
\end{align*}
and therefore
$$\{\gamma^+:\gamma \in \alpha \overline{\overline{\shuffle}} \beta \}=\{(2,1,2), (2,3), (2,2,1), (4,1), (2,3), (5), (2,2,1), (4,1) \}.$$
And, in the case where $\gamma = ({\color{red}{2}} +_1 {\color{blue}{2}} +_2{\color{red}{1}})$, we have $\mathrm{c_1}(\gamma) = 1$ and $\mathrm{c_2}(\gamma) = 1$.
\end{example}

With the above notations, the main result of this subsection can be stated in the following form.

\begin{theorem}\label{thm: str consts for B}
The following formulas hold.
\begin{enumerate}[label = {\rm (\alph*)}] 
    \item For $\alpha ,\beta \in \Comp$,
    \begin{align*}
    \mathcal{D}_{\alpha}(q,t) \mathcal{D}_{\beta}(q,t) = 
    \sum_{\gamma \in 
    \alpha \overline{\overline{\shuffle}} \beta}  (q+t)^{\mathrm{c_1}(\gamma)} t^{\mathrm{c_2}(\gamma)}\mathcal{D}_{\gamma^+}(q,t).
    \end{align*}
    \item For $\gamma \in \Comp$,
    \begin{align*}
    \triangle \mathcal{D}_{\gamma}(q,t) = \sum_{\alpha \cdot \beta = \gamma} \mathcal{D}_{\alpha}(q,t) \otimes \mathcal{D}_{\beta}(q,t).
    \end{align*}
\end{enumerate}
\end{theorem}
The proof for this theorem will be presented in~\cref{The proof of Theorem for str consts for B}.

\begin{example} Let $\alpha = (2,1)$ and $\beta = (2)$.~\cref{thm: str consts for B} (a) says that $\mathcal{D}_{\alpha}(q,t) \mathcal{D}_{\beta}(q,t)$ is expanded as 
\begin{align*}
&\mathcal{D}_{(2,1,2)}(q,t) +  (q+t) \mathcal{D}_{(2,3)}(q,t)\\
&+ \mathcal{D}_{(2,2,1)}(q,t) + (q+t) \mathcal{D}_{(4,1)}(q,t) + t \mathcal{D}_{(2,3)}(q,t) +(q+t) t \mathcal{D}_{(5)}(q,t)\\
&+ \mathcal{D}_{(2,2,1)}(q,t) + t \mathcal{D}_{(4,1)}(q,t).
\end{align*}
\end{example}

In~\cref{prop: specializations of B(q t)}, we showed that 
the enriched monomial quasisymmetric function $\eta_{\alpha}$ and 
the enriched $q$-monomial quasisymmetric function $\eta^{(q)}_{\alpha}$
can be obtained by suitable specializations of $q$ and $t$. To be precise, 
$ \eta_{\alpha}=2^n \mathcal{D}_{\alpha}(2,-1) $
and $ \eta^{(q)}_{\alpha}=(q+1)^n \mathcal{D}_{\alpha}(q+1,-1)$.
It should be pointed out that a product rule for $\eta_{\alpha}$ was presented in~\cite[Theorem 3.11]{GV21} and 
a product rule for $\eta_{\alpha}^{(q)}$ was presented in~\cite[Theorem 5.1, 5.9, 5.14]{GV23-1} (or~\cite[Corollary 1]{GV22}). 
However,~\cref{thm: str consts for B} also enables us to derive another 
product formulas for these bases by specializing $q$ and $t$.

\begin{corollary}\label{product and coproducts for enriched monomials}
For $\alpha, \beta \in \Comp$, we have 
\begin{align*}
\eta_{\alpha} \eta_{\beta} &= 
\sum_{\gamma \in 
\alpha \overline{\overline{\shuffle}} \beta} (-1)^{\mathrm{c_2}(\gamma)} \eta_{\gamma^+}, \text{ and }\\
\eta_{\alpha}^{(q)} \eta_{\beta}^{(q)} &= 
\sum_{\gamma \in 
\alpha \overline{\overline{\shuffle}} \beta} (-1)^{\mathrm{c_2}(\gamma)} q^{\mathrm{c_1}(\gamma)} \eta_{\gamma^+}^{(q)}\,\,.
\end{align*}
\end{corollary}

We close this subsection by remarking that 
the well known product rules for $\{ M_{\alpha} \}$, $\{ \Lambda_{\alpha}^* \}$ and $\{ E_{\alpha} \}$ can be recovered from~\cref{thm: str consts for B} (a).
For $\alpha, \beta \in \Comp$, let ${\mathcal X}_1$(resp. ${\mathcal X}_2$) be the multiset containing two-way overlapping shuffles of $\alpha$ and $\beta$ without any $+_2$(resp. $+_1$) occurrences.
One can observe that
\[
\{\gamma^+ \mid \gamma \in  {\mathcal X}_1\} = \{\gamma^+ \mid \gamma \in  {\mathcal X}_2\} = \alpha \overline{\shuffle} \beta. 
\]
Hence, specializing $(q,t)$ into $(1,0)$, $(-1,1)$, and $(1,-1)$ 
in~\cref{thm: str consts for B} (a) yields the following formulas:
\begin{align*}
M_{\alpha} M_{\beta} &= \sum_{\gamma \in \alpha \overline{\shuffle} \beta} M_{\gamma} &{}& (q = 1, t = 0), \\
\Lambda_{\alpha}^* \Lambda_{\beta}^* &= \sum_{\gamma \in \alpha \overline{\shuffle} \beta} \Lambda_{\gamma}^* &{}& (q = -1, t = 1), \\
E_{\alpha} E_{\beta} &= \sum_{\gamma \in \alpha \overline{\shuffle} \beta} E_{\gamma} &{}& (q = 1, t = -1).
\end{align*}

\subsection{The proof of~\cref{thm: str consts for B}}
\label{The proof of Theorem for str consts for B}
Let us collect the necessary notions and lemmas.
Let $S=\{s_1<s_2< \cdots \}$ be a nonempty finite set and $K$ be a subset of $S$.
We define ${\rm{w}}(S;K)$ to be the (T,F)-word of length $|S|$ whose $i$th entry is given by 
\begin{align*}
{\rm{w}}(S;K)_i = 
\begin{cases}
    {\rm{T}} &\text{ if } s_i \in K, \\
    {\rm{F}} &\text{ if } s_i \notin K.
\end{cases}
\end{align*}
From the definition, ${\rm{w}}(S;K)$ determines $\std_S(K)$:
$$\std_S(K)=\{i\in [|S|] \mid {\rm{w}}(S;K)_i = {\rm{T}}\}.$$
For the definition of $\std_S(\cdot)$, see~\cref{eqdef: standardization of S}.

For instance, if $K= \{2,3,5\} \subseteq S = \{1,2,3,5,7 \}$, then ${\rm{w}}(S,K) = {\rm{FTTTF}}$ and $\std_S(K) = \{ 2,3,4\} \in \binom{[5]}{3}$.

In what follows, we fix 
$\alpha = (\alpha_1, \alpha_2, \ldots, \alpha_l) \in \Comp_m$, $\beta = (\beta_1, \beta_2, \ldots, \beta_k) \in \Comp_n$,
$I = \set(\alpha) \subseteq [m-1]$, and $J = \set(\beta) \subseteq [n-1]$.
For $A \in \binom{[m+n]}{n}$, 
we set
\[
\widetilde{J_A}:=
\begin{cases}
J_A \sqcup \{z_A \} &\text{ if } m+n \notin A,\\
J_A \sqcup \{m+n \} &\text{ if } m+n \in A
\end{cases}
\]
and 
\[
\widetilde{I_{A^{\mathrm c}}}:=
\begin{cases}
I_{A^{\mathrm c}} \sqcup \{m+n \} &\text{ if } m+n \notin A,\\
I_{A^{\mathrm c}} \sqcup \{z_A \} &\text{ if } m+n \in A
\end{cases}
\]
where $z_A$ is the number defined by
\[ 
\begin{cases}
\max A & \text{ if } m+n \notin A,\\
\max \,\,[m+n]\setminus A & \text{ if } m+n \in A
\end{cases}
\]
(see in the proof of~\cref{thm: product formula of superclass identifier}).
It holds that 
$$I \#_A J \sqcup \{z_A \} \sqcup \{m+n \} = \widetilde{I_{A^{\mathrm c}}} \sqcup \widetilde{J_A}.$$
Let
\begin{align*}
\mathcal{A}_{\alpha, \beta} := \left\{A \in \binom{[m+n]}{n} : \overline{\e}(A) \setminus \{z_A \} \subseteq I \#_A J \right\}
\end{align*}
and consider the map $\Psi: \mathcal{A}_{\alpha,\beta} \to \binom{[l+k]}{k}$ given by 
\[
\Psi(A) = \std_{ I \#_A J \sqcup \{z_A \} \sqcup \{m+n \}}(\widetilde{J_A}).
\]

\begin{lemma}\label{lem: Psi is injective}
$\Psi$ is injective.
\end{lemma}

\begin{proof}
Let $A,B \in \mathcal{A}_{\alpha,\beta}$ and $\Psi(A) = \Psi(B)$.
Suppose that $A \neq B$. 
We may assume that there exists $x \in [m+n]$ such that $x \in A$ and $x \neq B$ and for all $z<x$, $z \in A$ if and only if $z \in B$.

First, we assume that $x = 1$. 
Since $1 \in A $ and 
$$\overline{\e}(A) \setminus \{z_A \} \subseteq I \#_A J \quad 
(\text{equivalently, } \overline{\e}(A) \sqcup \{z_A \} \sqcup \{ m+n \} \subseteq \widetilde{I_{A^{\mathrm c}}} \sqcup \widetilde{J_A}),$$ 
one sees that the word ${\rm{w}}(I \#_A J \sqcup \{z_A \} \sqcup \{m+n \}, \widetilde{J_A})$ starts with T.
Similarly, since $1 \neq B$ and $\overline{\e}(B) \setminus \{z_B \} \subseteq I \#_B J$, the word ${\rm{w}}(I \#_B J \sqcup \{z_B \} \sqcup \{m+n \}, \widetilde{J_B})$  starts with F. This implies $\Psi(A) \neq \Psi(B)$.

Next, we assume that $x>1$. Then there exists $s \le x$ such that $[s,x-1] \in {\rm Int}(A^{\mathrm c})$.
And, there exists $y >x$ such that $[s,y] \in {\rm Int}(B^{\mathrm c})$.
Note that $[s,x-1] \cap \widetilde{I_{A^{\mathrm c}}} = [s,x-1] \cap \widetilde{I_{B^{\mathrm c}}}$ by the definition of of shifted set $I_{A^{\mathrm c}}$, $I_{B^{\mathrm c}}$ and the minimality of $x$.
However, by the property $\overline{\e}(B) \setminus \{z_B \} \subseteq I \#_B J$, we have $y \in \widetilde{I_{B^{\mathrm c}}}$.
Therefore, $|[s,y] \cap \widetilde{I_{B^{\mathrm c}}}|>|[s,x-1] \cap \widetilde{I_{A^{\mathrm c}}}|$.
This implies that 
$${\rm{w}}(I \#_A J \sqcup \{z_A \} \sqcup \{m+n \}, \widetilde{J_A}) \neq {\rm{w}}(I \#_B J \sqcup \{z_B \} \sqcup \{m+n \},\widetilde{J_B}),$$ 
hence $\Psi(A) \neq \Psi(B)$.
\end{proof}

We show the surjectivity of $\Psi$ by constructing its inverse.
For $D\in \binom{[l+k]}{k}$ and $1\le i \le k$, set   
\[
\Sigma(\alpha, \beta, D, i) := \sum_{j=1}^{d_i-1} (\alpha \shuffle_D \beta)_j.
\]
Here, the superscript $d_i$ denotes the $i$th element in the increasing order in $D$, i.e.,
$D = \{d_1 < d_2 < \cdots < d_k\}$. 
Since $d_i$ is the position of $\beta_i$ in the composition $\alpha \shuffle_D \beta$, $\Sigma(\alpha, \beta, D, i)$ represents the partial sum of $\alpha \shuffle_D \beta$ up to, but excluding, the appearance of $\beta_i$. 
Consider the map 
$$\Phi: \binom{[l+k]}{k} \to \mathcal{A}_{\alpha,\beta}, \quad D \mapsto \Phi(D),$$ 
where 
\begin{align*}
\Phi(D) = \bigcup_{1\le i\le k}
[ \Sigma(\alpha, \beta, D, i) +1, \Sigma(\alpha, \beta, D, i) +\beta_i   
].
\end{align*}

\begin{lemma}\label{lem: Psi is bijective}
$\Phi$ is the inverse of $\Psi$. 
\end{lemma}

\begin{proof}
Due to the equality $J = \{\beta_1, \beta_1+\beta_2, \ldots, \beta_1+\cdots+\beta_{k-1} \}$, we see that  
$$J_{\Phi(D)}=\{ \Sigma(\alpha, \beta, D, 1) +\beta_1, \Sigma(\alpha, \beta, D, 2) +\beta_2, \ldots, \Sigma(\alpha, \beta, D, k-1) +\beta_{k-1} \}$$ 
and therefore
\begin{align}\label{def of tilde J}
\widetilde{J_{\Phi(D)}}=\{ \Sigma(\alpha, \beta, D, 1) +\beta_1, \Sigma(\alpha, \beta, D, 2) +\beta_2, \ldots, \Sigma(\alpha, \beta, D, k) +\beta_k \}.
\end{align}
On the other hand, since 
$$I \#_{\Phi(D)} J \sqcup \{z_{\Phi(D)} \}= \set(\alpha \shuffle_D \beta)$$ and $\set(\alpha \shuffle_D \beta)$ is given by 
\begin{align*}
\{ (\alpha \shuffle_D \beta)_1, (\alpha \shuffle_D \beta)_1+ (\alpha \shuffle_D \beta)_2, \ldots , (\alpha \shuffle_D \beta)_1+\cdots+(\alpha \shuffle_D \beta)_{l+k-1} \},
\end{align*}
we have that  
\begin{equation}
\begin{aligned}\label{def of IJZ}
&I \#_{\Phi(D)} J \sqcup \{z_{\Phi(D)} \} \sqcup \{m+n \}\\
&=\{ (\alpha \shuffle_D \beta)_1, (\alpha \shuffle_D \beta)_1+ (\alpha \shuffle_D \beta)_2, \ldots , (\alpha \shuffle_D \beta)_1+\cdots+(\alpha \shuffle_D \beta)_{l+k} \}. 
\end{aligned}
\end{equation}
Combining~\cref{def of tilde J} and~\cref{def of IJZ} yields that 
\begin{equation} \label{showing surj}
\Psi(\Phi(D)) =  \std_{I \#_{\Phi(D)} J \sqcup \{z_{\Phi(D)} \} \sqcup \{m+n \}}\left(\widetilde{J_{\Phi(D)}}\right) = D.
\end{equation}
Therefore the assertion can be obtained by combining~\cref{showing surj} and~\cref{lem: Psi is injective}. 
\end{proof}

The subsequent lemma, which plays an important role in the proof of~\cref{thm: str consts for B}, can be derived from~\cref{lem: Psi is bijective}.

\begin{lemma}\label{lem: main lemma for overlapping shuffle}
Let $\alpha, \beta \in \Comp$ and $I = \set(\alpha), J = \set(\beta)$. 
\begin{enumerate}[label = {\rm (\alph*)}]
    \item As multisets,
    $$\{\comp((I \#_A J) \sqcup \{z_A \}) \mid A \in \mathcal{A}_{\alpha,\beta}\} = \alpha \shuffle \beta.$$
    \item As multisets,
    $$\biguplus_{A \in \mathcal{A_{\alpha,\beta}}} \{\comp(K) \mid (I \#_A J) \setminus \overline{\e}(A) \subseteq K \subseteq (I \#_A J) \sqcup \{ z_A \} \} = 
    \{\gamma^+ \mid \gamma \in \alpha \overline{\overline{\shuffle}} \beta \}$$
    where the symbol $\biguplus$ denotes the multiset union operation.
\end{enumerate}
\end{lemma}

\begin{proof}
(a) For each $A \in \mathcal{A}_{\alpha,\beta}$, it holds that 

\begin{equation} \label{shuffle equality}
\comp((I \#_A J) \sqcup \{ z_A \}) = \alpha \shuffle_{\Psi(A)} \beta.
\end{equation}
Hence the assertion is immediate from~\cref{lem: Psi is bijective}.

(b) Let us fix $A\in \mathcal{A}_{\alpha,\beta}$.
Write it as the disjoint union 
$$A=\bigcup_{1\le a_1<a_2<\cdots < a_r\le m+n}[a_i,b_i],$$
where $[a_i,b_i]$'s are maximal among the subintervals of $([m+n],\le)$ in $A$.
By~\cref{definition of Int(A) and e(A)}, we have 
\begin{equation}\label{definition of Int(A) and e(A) 1}
\begin{aligned}
\e(A) = \{ b_i \mid 1\le i \le r \} \setminus \{ m+n \}.
\end{aligned}
\end{equation}
For later use, let 
$\e(A) = \{ b_1 < b_2 < \cdots < b_t\}$ and 
$\e(A^{\mathrm c}) = \{ c_1 < c_2 < \cdots < c_s\}$, where $t=|\e(A)|$ and $s=|\e(A^{\mathrm c})|$.

Let $\alpha \overline{\overline{\shuffle}}_{\Psi(A)} \beta$ be 
the submultiset of $\alpha \overline{\overline{\shuffle}} \beta$ consisting of all the expressions obtained from the composition $\alpha \shuffle_{\Psi(A)} \beta$,
and let $\mathscr A$ (resp. $\mathscr B$) be 
the set of consecutive component pairs of the form $(\beta_j, \alpha_i)$ (resp. $(\alpha_i, \beta_j)$) that appear in $\alpha \shuffle_{\Psi(A)} \beta$.
Then the identity~\cref{shuffle equality} implies that 
$\mathscr A$ and $\mathscr B$ can be written as 
\begin{align*}
&\mathscr A=\{(\beta_{j_1}, \alpha_1), (\beta_{j_2}, \alpha_2), \ldots (\beta_{j_t}, \alpha_t)\} \text{ and } \\
&\mathscr B=\{(\alpha_{1}, \beta_{k_1}), (\alpha_{2}, \beta_{k_2}), \ldots, (\alpha_{s}, \beta_{k_s})\}. 
\end{align*}
Now, consider the bijections
\begin{align*}
&\mathfrak f:\e(A) \to \mathscr A, \quad b_i \mapsto (\beta_{j_i}, \alpha_i),\\
&\mathfrak g:\e(A^{\mathrm c}) \to \mathscr B, \quad c_i \mapsto (\alpha_{i}, \beta_{k_i}).
\end{align*}
It can be easily seen that $\comp(((I \#_A J) \sqcup \{ z_A \}) \setminus (S_1 \sqcup S_2))$ is equal to the composition $\gamma^+$, where $\gamma$ is the expression obtained from  $\alpha \shuffle_{\Psi(A)} \beta$ by substituting the comma within each pair   $(\beta_{j_i},\alpha_i)\in \mathfrak g(S_1)$ with $+_1$
and by substituting the comma within each pair $(\alpha_{a_i},\beta_{k_i})\in \mathfrak f(S_2)$ with $+_2$.
As a consequence, we have
\[
\{\comp(((I \#_A J) \sqcup \{ z_A \}) \setminus (S_1 \sqcup S_2)) \mid S_1 \subseteq \e(A^{\mathrm c}), S_2 \subseteq \e(A) \} = \{ \gamma^+ \mid \gamma \in \alpha \overline{\overline{\shuffle}}_{\Psi(A)} \beta \}.
\]
Next, we will show that 
\begin{align*}
&\{\comp(((I \#_A J) \sqcup \{ z_A \}) \setminus (S_1 \sqcup S_2)) \mid S_1 \subseteq \e(A^{\mathrm c}), S_2 \subseteq \e(A)\}\\
&=\{\comp(K) \mid (I \#_A J) \setminus \overline{\e}(A) \subseteq K \subseteq (I \#_A J) \sqcup \{ z_A \} \}.
\end{align*}
To do this, we observe that if 
$(I \#_A J) \setminus \overline{\e}(A) \subseteq K \subseteq (I \#_A J) \sqcup \{ z_A \}$,
then   
$K = ((I \#_A J) \sqcup \{ z_A \}) \setminus (S_1 \sqcup S_2)$, 
where $S_1= \e(A^{\mathrm c}) \setminus K$ and 
$S_2 = \e(A)  \setminus K$. 
This observation induces a bijection from 
\[
\{(S_1, S_2) \mid S_1 \subseteq \e(A^{\mathrm c}), S_2 \subseteq \e(A) \}
\]
to
\[
\{K \mid (I \#_A J) \setminus \overline{\e}(A) \subseteq K \subseteq (I \#_A J) \sqcup \{ z_A \} \}.
\]
by sending $(S_1, S_2)$ to $((I \#_A J) \sqcup \{ z_A \}) \setminus (S_1 \sqcup S_2)$,
and therefore  
\[
\{\comp(K) \mid (I \#_A J) \setminus \overline{\e}(A) \subseteq K \subseteq (I \#_A J) \sqcup \{ z_A \} \} = \{ \gamma^+ \mid \gamma \in \alpha \overline{\overline{\shuffle}}_{\Psi(A)} \beta \}.
\]
Now, the assertion follows from (a).
\end{proof}

\begin{example}
Let $m = 5$, $n = 3$, $\alpha = (3,2)$, and $\beta = (2,1)$.
Then $l = k = 2$, $I = \{3 \} \subseteq [4]$, $J = \{ 2\} \subseteq [2]$, and   
$$\mathcal{A}_{\alpha,\beta} = \{ \{ 1,2,3 \}, \{ 1,2,6 \}, \{ 1,2,8 \}, \{ 4,5,6 \}, \{ 4,5,8 \}, \{ 6,7,8 \} \}.$$
Since  
\begin{align*}
&\Psi(\{ 1,2,3 \}) = \{1,2 \}, \quad \Psi(\{ 1,2,6 \}) = \{1,3 \}, \\
&\Psi(\{ 1,2,8 \}) = \{1,4 \}, \quad \Psi(\{ 4,5,6 \}) = \{2,3 \}, \\
&\Psi(\{ 4,5,8 \}) = \{2,4 \}, \quad \Psi(\{ 6,7,8 \}) = \{3,4 \},
\end{align*}
we have $\Psi(\mathcal{A}_{\alpha, \beta}) = \binom{[4]}{2}$.
And, a simple computation shows that the multiset 
\[
\{\comp((I \#_A J) \sqcup \{z_A \}) \mid A \in \mathcal{A}_{\alpha,\beta}\}
\]
consists of 
\begin{align*}
&\comp(\{ 2,3,6 \}) = (2,1,3,2), \quad \comp(\{ 2,5,6 \}) = (2,3,1,2), \\
&\comp(\{ 2,5,7 \}) = (2,3,2,1), \quad \comp(\{ 3,5,6 \}) = (3,2,1,2), \\
&\comp(\{ 3,5,7 \}) = (3,2,2,1), \quad \comp(\{ 3,5,7 \}) = (3,2,2,1).
\end{align*}
Consequently we conclude that $\{\comp((I \#_A J) \sqcup \{z_A \}) \mid A \in \mathcal{A}_{\alpha,\beta}\} = \alpha \shuffle \beta$.
\end{example}

\begin{example}
Let $m=3$, $n=2$, $\alpha = (2,1)$, and $\beta = (2)$. Then $I = \{ 2\} \subseteq [2]$, $J = \emptyset \subseteq [1]$, and 
\[
\mathcal{A}_{\alpha, \beta} = \{ \{1,2 \}, \{3,4 \}, \{4,5 \} \}.
\]
We calculate the following:
\begin{itemize}
    \item For $A = \{1,2 \}$, we have $
(I \#_A J) \setminus \overline{\e}(A) = \{4\}$ and $(I \#_A J) \sqcup \{ z_A \} = \{2,4 \}$.
    \item For $A = \{3,4 \}$, we have $(I \#_A J) \setminus \overline{\e}(A) = \emptyset$ and $(I \#_A J) \sqcup \{ z_A \} = \{2,4 \}$.
    \item For $A = \{4,5 \}$, we have  $(I \#_A J) \setminus \overline{\e}(A) = \{2\}$ and $(I \#_A J) \sqcup \{ z_A \} = \{2,3 \}$.
\end{itemize}
This implies that  
\begin{align*}
&\biguplus_{A \in \mathcal{A_{\alpha,\beta}}} \{\comp(K) \mid (I \#_A J) \setminus \overline{\e}(A) \subseteq K \subseteq (I \#_A J) \sqcup \{ z_A \} \}\\
&=\{ (4,1), (2,2,1), (5), (2,3), (4,1), (2,2,1), (2,1), (2,1,2) \}\\
&=\{\gamma^+ \mid \gamma \in \alpha \overline{\overline{\shuffle}} \beta \}.
\end{align*}
The second equality follows from~\cref{example of overlapping shuffles}.
\end{example}

Another significant lemma pertains to product and coproduct rules for the basis 
$\{\mathcal{K}_{\alpha}(\nu)\}$.

\begin{lemma}\label{lem: str consts for Pi}
Let $m$, $n$, and $k$ be nonnegative integers and $\nu$ be a positive integer $>1$.

\begin{enumerate}[label = {\rm (\alph*)}]
    \item Let $I \subseteq [m-1], J \subseteq [n-1]$. Then 
    \[
    \mathcal{K}_{\comp(I)}(\nu) \mathcal{K}_{\comp(J)}(\nu) =
    \sum_{ K } C^K_{I,J}(\nu)  \mathcal{K}_{\comp(K)}(\nu) ,
    \] 
    where
    \[
    C^K_{I,J}(\nu) =  \sum_{\substack{A \in \binom{[m+n]}{n}:\\ (I \#_A J) \cap \overline{\e}(A) = \emptyset \\ I \#_A J \subseteq K \subseteq (I \#_A J ) \cup \overline{\e}(A) }} (-1)^{|K \cap \e(A^{\mathrm c})|}  (\nu-1)^{|K \cap \e(A)|}.
    \]
    \item For $\gamma \in \Comp$,\\
\[
\triangle K_{\gamma}(\nu) =  \sum_{\alpha \odot \beta = \gamma}   K_{\alpha}(\nu) \otimes K_{\beta}(\nu),
\]
where $\odot$ is the near-concatenation.
\end{enumerate}
\end{lemma}

\begin{proof}
(a) 
Note that $\ch_\nu(\kappa_I(\nu)) = (\nu-1)^{|I|}  \mathcal{K}_{\comp(I)}(\nu)$.
Therefore, from~\cref{thm: product formula of superclass identifier} together with the equality
\[
\left(\frac{1}{1-\nu}\right)^{|K \cap \e(A^{\mathrm c})|}  (\nu-1)^{|K|} = (-1)^{|K \cap \e(A^{\mathrm c})|}  (\nu-1)^{|K \setminus \e(A^{\mathrm c})|}
\]
it follows that 
\[
C^K_{I,J}(\nu) = \frac{1}{(\nu-1)^{|I|+|J|}} \sum_{\substack{A \in \binom{[m+n]}{n}:\\ (I \#_A J) \cap \overline{\e}(A) = \emptyset \\ I \#_A J \subseteq K \subseteq (I \#_A J ) \cup \overline{\e}(A) }} (-1)^{|K \cap \e(A^{\mathrm c})|}  (\nu-1)^{|K \setminus \e(A^{\mathrm c})|}.
\]
Whenever $(I \#_A J) \cap \overline{\e}(A) = \emptyset$
and $I \#_A J \subseteq K \subseteq (I \#_A J ) \cup \overline{\e}(A)$,
it holds that   
\[
|K \setminus \e(A^{\mathrm c})|-|I|-|J| = |K \cap \e(A)|,
\]
which verifies the assertion.

(b) The assertion follows from~\cref{thm: coproduct formula of superclass identifier}.
\end{proof}

Now, we are ready to prove~\cref{thm: str consts for B}.

\begin{proof}[Proof of~\cref{thm: str consts for B}]
(a) To begin with, we note that $\{\mathcal{D}_{\gamma}(-q,q-1) \mid \gamma \in \Comp\}$ is a basis of $\QSym_{\mathbb C(q)}$. 
For $I= \set(\alpha)$ and $J = \set(\beta)$, we define $C^K_{I,J}(q)$'s by the structure constants of $\QSym_{\mathbb C(q)}$ for this basis, that is,
$$
\mathcal{D}_{\comp(I)}(-q,q-1)   \mathcal{D}_{\comp(J)}(-q,q-1) = \sum_{K} C^K_{I,J}(q)  \mathcal{D}_{\comp(K)}(-q,q-1).
$$ 
Upon specializing $q$ to any positive integer $\nu>1$, the combination of~\cref{cor: D and kappa} with~\cref{lem: str consts for Pi} (a) leads to the following expression:
$$
C^K_{I,J}(\nu) =   \sum_{A}(-1)^{|\e(A^{\mathrm c}) \setminus K|} (\nu-1)^{|\e(A) \setminus K|}.
$$
where $A$ ranges over the set 
\begin{align*}
\mathcal B:=\left\{A \in \binom{[m+n]}{n} \bigg|
{{\overline{\e}(A) \setminus \{ z_A\} \subseteq (I \#_A J)} 
\atop {(I \#_A J ) \setminus \overline{\e}(A) \subseteq K \subseteq (I \#_A J) \sqcup \{z_A\} 
}} \right\}.
\end{align*}
Since $\nu$ ranges over the infinite set $\{2, 3, \ldots \}$, the above identity implies that  
$$
C^K_{I,J}(q)=\sum_{A \in \mathcal B} (-1)^{|\e(A^{\mathrm c}) \setminus K|} (q-1)^{|\e(A) \setminus K|}
$$
as polynomials in $q$.
On the other hand, letting 
$Q =q/q+t$
in~\cref{defn of B(q t)}, 
one can derive the identity 
\begin{align}\label{eq: D as 1 variable}
\mathcal{D}_{\comp(I)}(q,t) = (-q-t)^{-|I^{\mathrm c}|} \mathcal{D}_{\comp(I)}(-Q,Q-1).
\end{align}
Let 
\[
\mathcal{D}_{\Comp(I)}(q,t) \mathcal{D}_{\Comp(J)}(q,t) = \sum_{K} C^K_{I,J}(q,t) \mathcal{D}_{\Comp(K)}(q,t).
\] 
Then $C^K_{I,J}(q,t)$ can be written as 
\begin{align*}
&\sum_{A \in \mathcal B}  (-q-t)^{-|I^{\mathrm c}|-|J^{\mathrm c}|+|K^{\mathrm c}|} (-1)^{|\e(A^{\mathrm c}) \setminus K|} \left(\frac{-t}{q+t}\right)^{|\e(A) \setminus K|}.
\end{align*}
For $A\in \mathcal B$, one can see that 
\begin{align*}
& (-q-t)^{-|I^{\mathrm c}|-|J^{\mathrm c}|+|K^{\mathrm c}|} (-1)^{|\e(A^{\mathrm c}) \setminus K|} \left(\frac{-t}{q+t}\right)^{|\e(A) \setminus K|}\\
&=  (-q-t)^{|I|+|J|-|K|+1} (-1)^{|\e(A^{\mathrm c}) \setminus K|} \left(\frac{-t}{q+t} \right)^{|\e(A) \setminus K|}\\
&= (q+t)^{|I|+|J|-|K| +1 - |\e(A) \setminus K|} (-1)^{|I|+|J|-|K| +1 + |\e(A) \setminus K| + |\e(A^{\mathrm c}) \setminus K|} t^{|\e(A) \setminus K|}\\
&= (q+t)^{|\e(A^{\mathrm c}) \setminus K|}  t^{|\e(A) \setminus K|}.
\end{align*}
For the last equality, we used the fact that 
if the triple $(I,J,K)$ satisfies
\begin{align*}    
&\overline{\e}(A) \setminus \{ z_A\} \subseteq (I \#_A J) \text{ and } \\
&(I \#_A J ) \setminus \overline{\e}(A) \subseteq K \subseteq (I \#_A J) \sqcup \{z_A\},
\end{align*}
then 
\begin{align*}
|I| +|J| -|K| +1 &= |I \#_A J \sqcup \{z_A \}| - | K |\\
&= |\e(A) \setminus K| + |\e(A^{\mathrm c}) \setminus K|.
\end{align*}

Therefore 
\begin{align*}
C^K_{I,J}(q,t)=\sum_{A \in \mathcal B} (q+t)^{|\e(A^{\mathrm c}) \setminus K|}  t^{|\e(A) \setminus K|}.
\end{align*}

Therefore, $\mathcal{D}_{\alpha}(q,t) \mathcal{D}_{\beta}(q,t)$ is given by 
\begin{align*}
\sum_{A \in \mathcal{A}_{\alpha,\beta}} \left(\sum_{K: (I \#_A J ) \setminus \overline{\e}(A) \subseteq K \subseteq (I \#_A J) \sqcup \{z_A\}} (q+t)^{|\e(A^{\mathrm c}) \setminus K|}  t^{|\e(A) \setminus K|} \mathcal{D}_{\Comp(K)}(q,t)\right).
\end{align*}
The assertion follows by applying~\cref{lem: main lemma for overlapping shuffle} (b).

(b) It follows from
~\cref{eq: concatenation and near concatenation in set} that  
$(\alpha \odot \beta)^{\mathrm c} = \alpha^{\mathrm c} \cdot \beta^{\mathrm c}$.
Hence, by combining~\cref{cor: D and kappa} with~\cref{lem: str consts for Pi} (b), we see that
\[
\triangle \mathcal{D}_{\gamma}(\nu,\nu-1) = \sum_{\alpha \cdot \beta = \gamma} \mathcal{D}_{\alpha}(\nu,\nu-1) \otimes \mathcal{D}_{\beta}(\nu,\nu-1).
\]
Now, the desired result can be obtained  
using the same approach as in (a).
\end{proof}

Combining~\cref{cor: D and kappa} and~\cref{eq: D as 1 variable} with~\cref{thm: transition between Pi and L} (a), (b), or~\cref{lem: transition between Pi and M} (a), we deduce the following transitions between $\mathcal{D}_{\alpha}$ and $M_{\alpha}$ or $L_{\alpha}$:
\begin{align*}
L_{\comp(I)} &= \sum_{J} t^{|I \setminus J|}(q+t)^{|(I \cup J)^{\mathrm c}|} \mathcal{D}_{\comp(J)}(q,t)\\
\mathcal{D}_{\comp(J)} &= \sum_{I} (-1)^{|I \setminus J| + |J \setminus I|} t^{|J \setminus I|} q^{-n+1} (q+t)^{|I \cap J|} L_{\comp(I)}\\
M_{\alpha} &= \sum_{\beta \succeq \alpha} q^{n-\ell(\alpha)} t^{\ell(\alpha)-\ell(\beta)} \mathcal{D}_{\beta}(q,t)
\end{align*}
By the substitutions $q \mapsto q+1$ and $t \mapsto -1$, we recover the relations between $\eta^{(q)}_{\alpha}$ and $L_{\alpha}$ or $M_{\alpha}$ (\cite[Proposition 3.12, 3.11, 3.6]{GV23-1}).

\section{Quasisymmetric Hall--Littlewood functions}\label{section: quasiHL}
The Hall--Littlewood polynomials were initially introduced indirectly by Hall through the Hall algebra and were subsequently formally defined by Littlewood~\cite{Littlewood61}.
These polynomials are symmetric functions depending on a parameter $q$ and a partition $\lambda$, holding significant importance in the realms of representation theory and combinatorics.
To be precise, for each partition $\lambda=(\lambda_1, \lambda_2, \ldots, \lambda_n)$ of length at most $n$, the Hall--Littlewood polynomial
$Q_{\lambda}(x_1,\ldots,x_n;q)$ is defined by 
\[
Q_{\lambda}(x_1,\ldots,x_n;q) = \dfrac{(1-q)^{\ell(\lambda)}}{[n-\ell(\lambda)]_q !} \sum_{\sigma \in \SG_n} \sigma\left( x^{\lambda} \prod_{1\le i<j \le n} \dfrac{x_i-qx_j}{x_i-x_j} \right),
\]
where $x^\lambda=x_1^{\lambda_1}\cdots x_n^{\lambda_n}$ and $\sigma(x_1^{m_1}\cdots x_n^{m_n})=x_1^{m_{\sigma(1)}}\cdots x_n^{m_{\sigma(n)}}$ for all $(m_1,\ldots, m_n)\in \mathbb Z_{\ge 0}^n$ (for details, see~\cite{M98}).

It was shown in~\cite{DKL95} that when we replace $q$ with its reciprocal $q^{-1}$, 
we can represent this polynomial using the full $q$-symmetrizing operator $\square_\omega$, where $\omega$ is the longest permutation in $\SG_n$, in the following manner:
\[
Q_{\lambda}(x_1,\ldots,x_n;q^{-1}) = q^{-\binom{N}{2}} \dfrac{(1-q^{-1})^{\ell(\lambda)}}{[n-\ell(\lambda)]_{q^{-1}} !} \square_{\omega} (x^{\lambda})
\]
In 2000, Hivert introduced a quasisymmetric analogue $\boxdot_{\omega}$ of $\square_{\omega}$ and defined {\it the quasisymmetric Hall--Littlewood polynomial $G_{\alpha}(x_1,\ldots,x_n;q)$ indexed by a composition $\alpha$ of length $k \le n$} as follows:
\[
G_{\alpha}(x_1,\ldots,x_n;q) =\frac{1}{[k]!_q [n-k]!_q} \boxdot_{\omega}(x_1^{\alpha_1}\cdots x_k^{\alpha_k} )
\]
(see~\cite[Section 6]{H00} for the precise definition). 
Then, employing the stability property~\cite[Proposition 6.5]{H00}, he established the existence of the limit $G_{\alpha}(q) := \lim_{n \to \infty} G_{\alpha}(x_1,\ldots,x_n;q)$, referred to as {\it the quasisymmetric Hall--Littlewood function}.
The Hall--Littlewood function $G_{\alpha}(q)$ interpolates between the fundamental quasisymmetric function and monomial quasisymmetric function. In detail,
\[
G_{\alpha}(0) = L_{\alpha} \quad \text{and}   \quad G_{\alpha}(1) = M_{\alpha},
\]
which is reminiscent of the following interpolation of the Hall--Littlewood $P$-function $P_{\lambda}(q)$
\[
P_{\lambda}(0) = s_{\lambda} \quad \text{and} \quad P_{\lambda}(1) = m_{\lambda},
\]
where $s_{\lambda}$ is the Schur function and $m_{\lambda}$ is the monomial symmetric function.

As Hall--Littlewood functions, quasisymmetric Hall--Littlewood functions have also drawn the attention of many mathematicians, as seen in~\cite{LSW13, Huang14, HLMW11-HL, BBS14-HL, NNE21-HL}. 
For example, in~\cite{LSW13}, transition matrices between ${ G_{\alpha}(q) }$ and other bases of symmetric or quasisymmetric functions were computed, and in~\cite[Corollary 8.4, Remark 8.5]{Huang14}, a representation-theoretical interpretation of $G_{\alpha}(q)$ was provided, particularly when $\alpha$ is a hook-shaped partition. 
Furthermore, in~\cite{HLMW11-HL, BBS14-HL, NNE21-HL}, various types of quasisymmetric Hall--Littlewood functions, distinct from $G_{\alpha}(q)$, were introduced.

In this section, we provide both a product rule and a coproduct rule for $\{G_{\alpha}(q)\}$ using~\cref{thm: categorification of QSym}. 
Our product rule is new, whereas our coproduct rule 
turns out to equivalent to the product rule for the dual of $\{G_{\alpha}(q)\}$ given in~\cite[Theorem 6.15]{H00} although it was derived using an entirely different method.
For simplicity, we frequently use $G_I(q)$ and $L_I$ to refer to $G_{\comp(I)}(q)$ and $L_{\comp(I)}$, respectively, 
for every subset $I$ of $[n-1]$.

\subsection{Class functions corresponding to quasisymmetric Hall--Littlewood functions}
We begin with introducing the necessary notations.
Let $I = \{ i_1 < i_2 < \cdots < i_l \}$ and $J = \{ j_1 < j_2 < \cdots < j_r \}$ be subsets of $[n-1]$ such that $I \subseteq J$.
Following~\cite[Section 2, 6]{H00}\cite[Section 5.1]{LSW13}, we introduce the notations 
\footnote
{In~\cite{H00}, these statistics were originally defined in terms of compositions. However, we have translated them into their current forms for the sake of notation simplicity.}
\begin{equation*}\label{Hivert's notation}
\begin{aligned}
\Bre(I,J)_t &:= 
\begin{cases}
0 &\text{ if } t=0,\\
|\{j \in J: j \le i_{t} \}| & \text{ if } 1\le t \le l, 
\end{cases}\\
\Bre(I,J) &:= \{\Bre(I,J)_t : 1 \le t \le l \} \subseteq [r],\\
s(I,J) &:= \sum_{1 \le t \le l} t (\Bre(I,J)_t - \Bre(I,J)_{t-1} - 1)\\
g(I,J) &:= \sum_{k \in [r] \setminus \Bre(I,J)} k.
\end{aligned}
\end{equation*}

\begin{lemma} {\rm (\cite[Theorem 6.6, 6.13]{H00}, \cite[Theorem 26]{LSW13})} \label{lem: transition between G(q) and F}
For a positive integer $n$ and $I, J \subseteq [n-1]$, 
we have 
\begin{align}
\label{eq: G into L} G_I(q) &= \sum_{I \subseteq J } (-1)^{|J \setminus I|} q^{s(I,J)} L_J, \text{ and }\\
\label{eq: L into G} L_J &= \sum_{I \subseteq J} q^{g(I,J)} G_I(q).
\end{align}
\end{lemma}

To each $I \subseteq [n-1]$, we assign the function 
$$\wt_I: [n-1] \to \mathbb{N}, \quad i \mapsto |[1,i-1] \cap I| + 1,$$   
called {\it the weight associated to $I$}. 
When $n=4$, the following table shows $\wt_I$'s for all $I \subseteq [3]$.
\begin{table}[h]
\small{
\begin{tabular}{|c|c|c|c|c|c|c|c|c|}
\hline
$I$                      & $\emptyset$ & $\{1 \}$  & $\{2 \}$  & $\{3 \}$  & $\{1,2 \}$ & $\{1,3 \}$ & $\{2,3 \}$ & $\{1,2,3 \}$ \\ \hline
$(\wt_I(i))_{i \in [3]}$ & $(1,1,1)$   & $(1,2,2)$ & $(1,1,2)$ & $(1,1,1)$ & $(1,2,3)$  & $(1,2,2)$  & $(1,1,2)$  & $(1,2,3)$    \\ \hline
\end{tabular}}
\caption{$\wt_I$ for $I \subseteq [3]$}
\label{table: wt_I}
\end{table}

\begin{lemma}\label{lem: s,g decomposition into wt}
For a positive integer $n$ and $I,J \subseteq [n-1]$ with $I \subseteq J$, we have 
\begin{align}
\label{eq: s into wt}    s(I,J) &= \sum_{i \in J \setminus I} \wt_I(i), \text{ and }\\
\label{eq: g into wt}     g(I,J) &= \sum_{i \in J \setminus I} \wt_J(i).
\end{align}
\end{lemma}

\begin{proof}
Let $I = \{ i_1 < i_2 < \cdots < i_l \}$ and $J = \{ j_1 < i_2 < \cdots < j_r \}$.
Then      
\[\Bre(I,J)_t - \Bre(I,J)_{t-1} -1 = |\{ j: i_{t-1}<j<i_t\}|= \{j \in J \setminus I : \wt_I(i) = t\}.
\]
The first equality can be obtained by iterating summations. 
The second equality follows from the fact that $\wt_J(j_k) = k$ for $1\le k\le r$. 
\end{proof}

For each positive integer $\nu>1$, we set
\begin{align} \label{ definition of hat r}
\overline{\mathbbm{reg}} := \dfrac{\mathbbm{reg - 1}}{\nu-1}    \in \cf(C_{\nu}),
\end{align}
where $\mathbbm{1}$ and $\mathbbm{reg}$ are the trivial character and the character of the regular representation of $C_{\nu}$, respectively. 
Define 
\begin{equation}\label{eq: def of G_I}
\mathbb{G}_I(\nu) = \slb (\mathbb{G}_I(\nu)_i)_{i \in [n-1]} \srb: \bigoplus C_\nu \to \mathbb C
\end{equation} 
by the class function of $\bigoplus C_\nu$ given by
\[
    \mathbb{G}_I(\nu)_i = \begin{cases} \mathbbm{1} &\text{ if } i \in I, \\
\overline{\mathbbm{reg}} - \nu^{wt_I(i)} \mathbbm{1} &\text{ if } i \in I^{\mathrm c}.        
    \end{cases}
\]

\begin{proposition}\label{prop: class function correspond to G}
Let $I \subseteq [n-1]$ and $\nu$ be a positive integer $>1$. Then
    $$\ch_\nu(\mathbb{G}_I(\nu)) = G_I(\nu).$$
\end{proposition}

\begin{proof}
Recall that in~\cref{thm: categorification of QSym}, we showed the equality $\ch_{\nu}(\slb {\dot\chi^I(\nu)}_{i} \srb)= L_{I}$, 
where  
$${\dot\chi^I(\nu)}_{i} = 
\begin{cases}
    \mathbbm{1}  &   \text{if } i \in I,\\
    \overline{\mathbbm{reg}}  & \text{if } i \in S\setminus I.
\end{cases}
$$
The assertion can be deduced by combining this identity with~\cref{eq: G into L} and~\cref{eq: s into wt}.
\end{proof}

\subsection{Product and coproduct formulas for 
$\{G_I(q)\}$}
As before, assume that $m,n$ are nonnegative integers. 
Consider a word $v = v_1v_2 \dots v_{m+n}$ of length $m+n$ with entries taken from the set of non-negative integers, allowing for repetitions.
Let $P_v:=\{1\le i\le m+n \mid v_i \ne 0\}$. 
It is evident that $\Des(v) \subseteq P_v$, where 
$\Des(v) = \{ 1\le i \le m+n-1 \mid v_i > v_{i+1}\}$.
Define {\it the standardized descent set of $v$} by  
\[
\sDes(v):=\std_{P_v}(\Des(v)) \setminus \{ m \}.
\]
For the definition of $\std_{P_v}(\cdot)$, see~\cref{eqdef: standardization of S}.

Hereafter, we fix a permutation $w \in \SG_{m+n}$. Recall that in this paper, we are identifying $w$ with the word $w(1)w(2)\cdots w(m+n)$. 
Under this identification, 
define $w_{\le m} \in \SG_m$ be the subword of $w$ consisting of the entries in $[m]$, and $w_{>m} \in \SG_n$ the subword of $w$ consisting of the entries in $[m+n]\setminus [m]$. 
And, we define $w_{\leq m}^0$ and $w_{>m}^0$ by the words of length $m+n$
whose $i$th entries are given by  
\begin{align*}
w_{\leq m}^0(i) = 
\begin{cases}
w(i) &\text{ if } w(i) \le m,\\
0 &\text{ if } w(i) > m,
\end{cases}\\
w_{>m}^0(i) = 
\begin{cases}
0 &\text{ if } w(i) \le m,\\
w(i) &\text{ if } w(i) > m.
\end{cases}
\end{align*}
Then one can easily see that $\Des(w_{\le m}) \subseteq \sDes(w_{\le m}^0)$ and $\Des(w_{> m}) \subseteq \sDes(w_{> m}^0)$. 

For simplicity, we set 
$$
\mathsf{pos}_{(w;m)}(i) :=
\begin{cases}
(w_{\le m})^{-1}(w(i)) &\text{ if } w(i) \le m,\\
(w_{>m})^{-1}(w(i)) &\text{ if } w(i) > m.
\end{cases}
$$
With this notation, {\it the standardized $q$-weight associated to $(w;m)$} is defined to be  
the function $\sw^w_m : [m+n-1] \to \mathbb{C}(q)$ given by
\begin{align*}
i \mapsto \begin{cases}
q^{\wt_{\Des(w_{\le m})}(\mathsf{pos}_{(w;m)}(i))} &\text{ if } w(i) \le m, w(i+1) \le m,\\
q^{\wt_{\Des(w_{>m})}(\mathsf{pos}_{(w;m)}(i))} &\text{ if } w(i) > m, w(i+1) > m,\\
0 &\text{ otherwise.}
\end{cases}
\end{align*}

\begin{example}
Let $m=3, n=4$, and $w = 7614532$. Then we have  
\begin{align*}
& w_{\le m} = 132, \quad w_{>m} = 7645, \\ 
& w_{\le m}^0 = 0010032, \quad w_{>m}^0 = 7604500, \\
& \sDes(w_{\le m}^0) = \{1, 2 \} \subseteq [2], \quad 
\sDes(w_{>m}^0) = \{ 1,2 \} \subseteq [3].
\end{align*}
Furthermore, we see that $\Des(w_{\le m}) = \{ 2 \} \subseteq [2]$ and $\Des(w_{>m}) = 
\{ 1\} \subseteq [3]$, and therefore 
\begin{align*}
&\wt_{\Des(w_{\le m})}(1) = 1, \wt_{\Des(w_{\le m})}(2)=1, \\
&\wt_{Des(w_{>m})}(1)=1, \wt_{Des(w_{>m})}(2)=2, \wt_{Des(w_{>m})}(3)=2.
\end{align*}
Consequently, we have
\begin{align*}
&\sw^w_m(1) = q^{\wt_J(1)} = q, \quad \,\,\sw^w_m(2) = 0, \quad \sw^w_m(3) = 0,\\
&\sw^w_m(4) = q^{\wt_J(3)} = q^2, \quad \sw^w_m(5) = 0, \quad \sw^w_m(6) = q^{\wt_I(2)}=q. 
\end{align*}
\end{example}

Now, we are ready to state our product rule for  $\{ G_I(q) \}$.

\begin{theorem}\label{thm: product rule of quasi HL}
Let $I \subseteq [m-1], J \subseteq [n-1]$. 
Choose arbitrary permutations $u \in \SG_m$ and $v \in \SG_n$  
satisfying that $\Des(u) = I$ and $\Des(v) = J$.
Then we have 
\begin{align*}
&G_I(q)G_J(q) =
&\sum_{w \in u \shuffle v[m]} c_{w;m} \sum_{K: \Des(w) \subseteq K} \left( \prod_{i \in K \setminus \Des(w)}(q^{\wt_K(i)} - \sw^w_m(i)) \right) G_K(q),
\end{align*}
where 
\[
c_{w;m} =  \left( \prod_{i \in \sDes(w_{\le m}^0) \setminus I}(1-q^{\wt_I(i)}) \right) \left( \prod_{i \in \sDes(w_{>m}^0) \setminus J}(1-q^{\wt_J(i)}) \right).
\]
\end{theorem}

\begin{corollary}\label{cor: product rule for quasiHL} Under the same setting as in~\cref{thm: product rule of quasi HL}, 
let 
\[
G_I(q)G_J(q) = \sum_{K} c(q)_{I,J}^{K} G_K(q), \text{ where }c(q)_{I,J}^{K} \in \mathbb C(q).
\]
Then the coefficient $c(q)_{I,J}^{K}$ is given by
\begin{align*}
\sum_{\substack{w \in u \shuffle v[m] \\ \Des(w) \subseteq K }} \left(  \prod_{i \in \sDes(w_{\le m}^0) \setminus I}(1-q^{\wt_I(i)})   \prod_{i \in \sDes(w_{>m}^0) \setminus J}(1-q^{\wt_J(i)}) 
\prod_{i \in K \setminus \Des(w)}(q^{\wt_K(i)} - \sw^w_m(i))
\right).
\end{align*}
\end{corollary}

\begin{example} 
Let $m = n =2$, $I = \emptyset \subseteq [1]$, and $J = \{1 \} \subseteq [1]$. 
Then we have $\wt_I(1) = \wt_J(1) =1$.
Let us choose permutations $u = 12, v = 21 \in \SG_2$, which satisfy the condition $\Des(u) = I$ and $\Des(v) = J$.
The shifted shuffle set $u \shuffle v[2]$ is given by $\{1243, 1423,1432,4123,4132,4312 \}$.
The following table shows the statistics required to calculate $G_I(q) G_J(q)$.

\newcolumntype{C}[1]{>{\centering\arraybackslash}p{#1}}

\makeatletter
\def\hlinewd#1{%
\noalign{\ifnum0=`}\fi\hrule \@height #1 \futurelet
\reserved@a\@xhline}
\makeatother

\newcolumntype{?}{!{\vrule width 1pt}}

\begin{table}[H]
\resizebox{\columnwidth}{!}{%
\small{
\begin{tabular}
{|C{0.8cm}?C{0.8cm}|C{0.8cm}|C{2.5cm}|C{2.5cm}|C{1.2cm}|C{1.2cm}|C{2cm}|}
\hline
& & & & & & & \\[-1em]
$w$  & $w^0_{\le m}$ & $w^0_{>m}$ & $\sDes(w^0_{\le m}) \setminus I$ & $\sDes(w^0_{>m}) \setminus J$ & $c_{w;m}$ & $\Des(w)$   & $(\sw^w_m(i))_{i \in [3]}$ \\ [-1em]
& & & & & &  \\ \hlinewd{1pt}
1243 & 1200          & 0043       & $\{3\}$                          & $\emptyset$                   & $1$      & $\emptyset$ & $(q,0,q)$                  \\ \hline
1423 & 1020          & 0403       & $\{1\}$                          & $\emptyset$                   & $1-q$    & $\{2\}$     & $(0,0,0)$                  \\ \hline
1432 & 1002          & 0430       & $\{1\}$                          & $\emptyset$                   & $1-q$    & $\{2,3\}$   & $(0,q,0)$                  \\ \hline
4123 & 0120          & 4003       & $\emptyset$                      & $\emptyset$                   & $1$      & $\{1\}$     & $(0,q,0)$                  \\ \hline
4132 & 0102          & 4030       & $\{1\}$                          & $\emptyset$                   & $1-q$    & $\{1,3\}$   & $(0,0,0)$                  \\ \hline
4312 & 0012          & 4300       & $\emptyset$                      & $\emptyset$                   & $1$      & $\{1,2\}$   & $(q,0,q)$                  \\ \hline
\end{tabular}}
}
\end{table}
Applying this information along with~\cref{table: wt_I} to  
~\cref{thm: product rule of quasi HL} yields the equality
\begin{align*}
G_I(q) G_J(q) = &G_{\{ 3\} }(q) +(q-q) G_{\{ 1, 3\} }(q) + q G_{\{ 2, 3\} }(q) \\
&+(1-q)(G_{\{2\}}(q) + q G_{\{1,2 \}}(q) + q^2 G_{\{2,3 \}}(q) + q^4 G_{\{1,2,3 \}}(q))\\
&+(1-q) (G_{\{2,3 \}}(q) + q G_{\{1,2,3 \}}(q))\\
&+G_{\{1 \}}(q) + (q^2-q) G_{\{1,2 \}}(q) + q^2 G_{\{1,3 \}}(q)+ (q^2-q)q^3 G_{\{1,2,3 \}}(q)\\
&+(1-q) (G_{\{1,3 \}}(q) + q^2 G_{\{1, 2,3 \}}(q))\\
&+G_{\{1,2 \}}(q) + (q^3-q) G_{ \{1,2,3 \}}(q).
\end{align*}
Therefore the expansion of  
$G_I(q) G_J(q)$ in the basis $\{G_I(q)\}$ is given by
\[
G_{ \{1 \}}(q) + (1-q) G_{ \{2 \} }(q) + G_{ \{3 \} }(q) +  G_{ \{1,2 \}  }(q) + (1-q+q^2) G_{ \{1,3 \}  }(q) + (1+q^2-q^3) G_{ \{2,3 \}  }(q).
\]
\end{example}

For nonnegative integers $u, v$,  set 
$c_q(u,v) := (1-q^u)(1-q^{u-1})\cdots(1-q^{u-v+1})$ with the convention $c_q(u,0) = 1$ and $c_q(u,v) = 0$ if $v>u$.
With this notation, we provide the following coproduct formula for $\{ G_{\gamma}(q) \}$.

\begin{theorem} {\rm (cf.~\cite[Theorem 6.15]{H00})}\label{thm: coproduct formula of G}
For $\gamma \in \Comp$, we have
\[
\triangle G_{\gamma}(q) = \sum_{\alpha \cdot \beta = \gamma \text{ or } \alpha \odot \beta = \gamma} \left( \sum_{\beta' \preceq \beta}  q^{g(\beta, \beta')} d(\alpha; \beta, \beta';\gamma) G_{\alpha}(q) \otimes G_{\beta'}(q) \right),
\]   
where 
$$d(\alpha; \beta, \beta';\gamma) =
\begin{cases}
    c_q(\ell(\alpha), \ell(\beta') - \ell(\beta)) &\text{ if } \alpha \cdot \beta = \gamma, \\
    c_q(\ell(\alpha), \ell(\beta') - \ell(\beta)+1) &\text{ otherwise. }
\end{cases}$$
\end{theorem}

\begin{example}
Let $\gamma = (1,2,1)$. 
A simple computation using~\cref{thm: coproduct formula of G} shows that  $\triangle G_{\gamma}(q)$ is given by
\begin{align*}
&G_{\emptyset}(q) \otimes G_{(1,2,1)}(q) + G_{(1)}(q) \otimes G_{(2,1)}(q) + q(1-q) G_{(1)}(q) \otimes G_{(1,1,1)}(q)\\ 
&+G_{(1,2)}(q) \otimes G_{(1)}(q) + G_{(1,2,1)}(q) \otimes G_{\emptyset}(q) +(1-q^2) G_{(1,1)}(q) \otimes G_{(1,1)}(q).
\end{align*}
\end{example}

\subsection{The proofs of~\cref{thm: product rule of quasi HL} and~\cref{thm: coproduct formula of G}}
We introduce two lemmas necessary for the proofs.

\begin{lemma}\label{lem: m_A G_I and G_J }
Let $I \subseteq [m-1]$, $J \subseteq [n-1]$, and $\nu$ be a positive integer $>1$. 
For each $A \in \binom{[m+n]}{n}$, we have 
\begin{align*}
{\bf m}_A(\mathbb{G}_I(\nu),\mathbb{G}_J(\nu))
=c_{I,J,A}\,\,\slb ( \mathbb{P}_A(I,J)_i)_{i \in [m+n-1]}\srb,
\end{align*}
where 
$$c_{I,J,A}=\left(\prod_{i \in \std_{A^{\mathrm c}}(\e(A^{\mathrm c})) \setminus \{m\} \setminus I}(1-\nu^{\wt_I(i)}) \right) \left( \prod_{i \in \std_A(\e(A)) \setminus \{n\} \setminus J}(1-\nu^{\wt_J(i)})  \right)$$ 
and   
\begin{align*}
\mathbb{P}_A(I,J)_i = 
\begin{cases}
\mathbbm{1} &\text{ if } i \in I \shuffle_A J    ,\\
\overline{\mathbbm{reg}} &\text{ if } i \in \e(A^{\mathrm c})     ,\\
\overline{\mathbbm{reg}}-\nu^{\wt_I(\std_{A^{\mathrm c}}(i))}\mathbbm{1} &\text{ if }  i \in (I^{\mathrm c})_{A^{\mathrm c}} \setminus \e(A^{\mathrm c})      ,\\
\overline{\mathbbm{reg}}-\nu^{\wt_J(\std_A(i))}\mathbbm{1} &\text{ if }  i \in (J^{\mathrm c})_{A} \setminus \e(A)
\end{cases}
\end{align*}
$($for the definitions of $\mathbb{G}_I(\nu)$ and 
$\overline{\mathbbm{reg}}$, see~\cref{eq: def of G_I} and~\cref{ definition of hat r}, respectively$)$.
\end{lemma}

\begin{proof}
By the definition of $s_A$, the $i$th component of $s_A(\mathbb{G}_I(\nu),\mathbb{G}_I(\nu))$ is given by 
\begin{align*}  s_A(\mathbb{G}_I(\nu),\mathbb{G}_I(\nu))_i
= 
\begin{cases}
\mathbbm{1} &\text{ if } i \in (I)_{A^{\mathrm c}}    ,\\
\overline{\mathbbm{reg}}-\nu^{\wt_I(\std_{A^{\mathrm c}}(i))}\mathbbm{1} &\text{ if }  i \in (I^{\mathrm c})_{A^{\mathrm c}}      ,\\
\mathbbm{1} &\text{ if } i \in J_A       ,\\
\overline{\mathbbm{reg}}-\nu^{\wt_J(\std_A(i))}\mathbbm{1} &\text{ if }  i \in (J^{\mathrm c})_A,\\
\overline{\mathbbm{reg}} &\text{ if }  i =z_A,
\end{cases}
\end{align*}
where  
$$z_A= 
\begin{cases}
\max A & \text{ if } m+n \notin A,\\
\max \,\,[m+n]\setminus A & \text{ if } m+n \in A.
\end{cases}
$$
It follows that 
\begin{equation}\label{product of sa(GI GJ)}
\begin{aligned}
&{\bf s}_A(\mathbb{G}_I(\nu),\mathbb{G}_I(\nu))  \big\downarrow^{Q_{m+n+1}(\nu)}_{Q_{[m+n-1] \setminus \overline{\e}(A)}(\nu)}\\
&=\prod_{i \in \overline{\e}(A) \cap (I^{\mathrm c})_{A^{\mathrm c}}}(1-\nu^{\wt_I(\std_{A^{\mathrm c}}(i))})  
\prod_{i \in \overline{\e}(A) \cap (J^{\mathrm c})_A}(1-\nu^{\wt_J(\std_{A}(i))}) \,\, 
\slb \mathbb{X}_A(I,J)_i \srb\\
&=\prod_{i \in \std_{A^{\mathrm c}}(\e(A^{\mathrm c})) \setminus \{m\} \setminus I}(1-\nu^{\wt_I(i)})
\prod_{i \in \std_A(\e(A)) \setminus \{n\} \setminus J}(1-\nu^{\wt_J(j)}) \,\, 
\slb \mathbb{X}_A(I,J)_i \srb,
\end{aligned}
\end{equation}
where  
\begin{align*}
\mathbb{X}_A(I,J)_i = 
\begin{cases}
\mathbbm{1} &\text{ if } i \in (I)_{A^{\mathrm c}} \setminus \e(A^{\mathrm c})    ,\\
\overline{\mathbbm{reg}}-\nu^{\wt_I(\std_{A^{\mathrm c}}(i))}\mathbbm{1} &\text{ if }  i \in (I^{\mathrm c})_{A^{\mathrm c}} \setminus \e(A^{\mathrm c})      ,\\
\mathbbm{1} &\text{ if } i \in J_A \setminus \e(A)      ,\\
\overline{\mathbbm{reg}}-\nu^{\wt_J(\std_A(i))}\mathbbm{1} &\text{ if }  i \in (J^{\mathrm c})_A \setminus \e(A).
\end{cases}
\end{align*}
The second equality follows from the identities
\begin{align*}
&\std_{A^{\mathrm c}}(\overline{\e}(A) \cap (I^{\mathrm c})_{A^{\mathrm c}}) = 
\std_{A^{\mathrm c}}(\e(A^{\mathrm c}) \cap (I^{\mathrm c})_{A^{\mathrm c}}) = \std_{A^{\mathrm c}}(\e(A^{\mathrm c})) \setminus \{m\} \setminus I, \text{ and}\\
&\std_{A}(\overline{\e}(A) \cap (J^{\mathrm c})_{A}) = 
\std_{A}(\e(A) \cap (J^{\mathrm c})_{A}) = \std_{A}(\e(A)) \setminus \{n\} \setminus J.
\end{align*}
Plugging~\cref{product of sa(GI GJ)} to the right hand side of the equality
\begin{align*}
{\bf m}_A(\mathbb{G}_I(\nu),\mathbb{G}_I(\nu))
=\dot\chi^{\e(A)}(\nu) \otimes_{\overline{\e}(A)} 
\left( {\bf s}_A(\mathbb{G}_I(\nu),\mathbb{G}_I(\nu)) \big\downarrow^{Q_{m+n+1}(\nu)}_{Q_{[m+n-1] \setminus \overline{\e}(A)}(\nu)} \right)
\end{align*}
yields that 
\begin{align*}
&{\bf m}_A(\mathbb{G}_I(\nu),\mathbb{G}_I(\nu))\\
&= \prod_{i \in \std_{A^{\mathrm c}}(\e(A^{\mathrm c})) \setminus \{m\} \setminus I}(1-\nu^{\wt_I(i)})
\prod_{i \in \std_A(\e(A)) \setminus \{n\} \setminus J}(1-\nu^{\wt_J(j)}) \,\, \slb \mathbb{P}_A(I,J)_i \srb,
\end{align*}
where 
\begin{align*}
\mathbb{P}_A(I,J)_i &= 
\begin{cases}
\mathbbm{1} &\text{ if } i \in (I)_{A^{\mathrm c}} \setminus \e(A^{\mathrm c})    ,\\
\overline{\mathbbm{reg}} &\text{ if } i \in (I)_{A^{\mathrm c}} \cap \e(A^{\mathrm c})    ,\\
\overline{\mathbbm{reg}}-\nu^{\wt_I(\std_{A^{\mathrm c}}(i))}\mathbbm{1} &\text{ if }  i \in (I^{\mathrm c})_{A^{\mathrm c}} \setminus \e(A^{\mathrm c})      ,\\
\overline{\mathbbm{reg}} &\text{ if }  i \in (I^{\mathrm c})_{A^{\mathrm c}} \cap \e(A^{\mathrm c})      ,\\
\mathbbm{1} &\text{ if } i \in J_A \setminus \e(A)      ,\\
\mathbbm{1} &\text{ if } i \in J_A \cap \e(A)      ,\\
\overline{\mathbbm{reg}}-\nu^{\wt_J(\std_A(i))}\mathbbm{1} &\text{ if }  i \in (J^{\mathrm c})_A \setminus \e(A),\\
\mathbbm{1} &\text{ if }  i \in (J^{\mathrm c})_A \cap \e(A).
\end{cases}
\end{align*}
Now, the assertion follows from the notation  
$I \shuffle_A J := \e(A) \sqcup ((I \#_A J) \setminus \overline{\e}(A))$ together with the identity $I \#_A J = I_{A^{\mathrm c}} \sqcup J_A.$
\end{proof}

\begin{lemma}\label{lem: Expansion of class function into G}
Let $I \subseteq [n-1]$. Given a function 
$f : [n-1]\setminus I \to \mathbb{C}$, let $\phi^{I,f} = \slb (\phi^{I,f}_i)_{i \in [n-1]} \srb $ be the class function defined by
\[
\phi^{I,f}_i =
\begin{cases}
    \mathbbm{1} &\text{ if } i \in I,\\
    \overline{\mathbbm{reg}} + f(i) \mathbbm{1} &\text{ if } i \in I^{\mathrm c}. 
\end{cases}
\]
Then the expansion of $\phi^{I,f}$ in the basis  $\{\mathbb{G}_I(\nu)\}$ is given by 
\begin{align}\label{eq: phi^I,f expansion of G}
    \phi^{I,f} = \sum_{J: I \subseteq J}  \left( \prod_{j \in J \setminus I}  \left( \nu^{\wt_{J}(j)} + f(j) \right) \right)  \mathbb{G}_J(\nu). 
\end{align}
\end{lemma}

\begin{proof}
Note that 
\begin{align*}
    \ch(\phi^{I,f}) &= \sum_{K: I \subseteq K} \left(\prod_{k \in K \setminus I} f(k) \right) L_K.
\end{align*}
By~\cref{eq: L into G}, we can replace $L_k$ on the right hand side by $\sum_{J: K \subseteq J} \nu^{g(K,J)} G_J(\nu)$. 
This gives rise to the equality  
\begin{align} \label{expansion of ch phi to GI}
\ch(\phi^{I,f}) 
= \sum_{J: I \subseteq J}\left( \sum_{K : I \subseteq K \subseteq J}  \left( \prod_{k \in K \setminus I} f(k) \right) \nu^{g(K,J)} \right)  G_J(\nu).
\end{align}
Due to~\cref{eq: g into wt}, the coefficient of $\mathbb{G}_J(\nu)$ of the right hand side of~\cref{expansion of ch phi to GI} can be rewritten as 
\begin{align*}
    \sum_{K : I \subseteq K \subseteq J} \left(\prod_{k \in K \setminus I} f(k) \right) \left( \prod_{k \in J \setminus K} \nu^{\wt_K(k)} \right),
\end{align*}
which is equal to 
$$ \prod_{j \in J \setminus I} \left( \nu^{\wt_K(j)} + f(j) \right),$$
as required. 
Now, the assertion follows from~\cref{prop: class function correspond to G}, which says that 
$\ch$ is an isomorphism sending $\mathbb{G}_J(\nu)$ to $G_J(\nu)$.
\end{proof}

Now, we are ready to prove~\cref{thm: product rule of quasi HL} and~\cref{thm: coproduct formula of G}.

\begin{proof}
[proof of~\cref{thm: product rule of quasi HL}]
Let $\nu$ be a positive integer $>1$.
For $A \in \binom{[m+n]}{n}$, let $w$ be the unique permutation such that $\Des(w) = I \shuffle_A J$.
Applying~\cref{lem: Expansion of class function into G} to $\slb  \mathbb{P}_A(I,J)_i\srb$ of~\cref{lem: m_A G_I and G_J }, we derive that 
\begin{align*}
\slb  \mathbb{P}_A(I,J)_i\srb = \sum_{K: I \shuffle_A J \subseteq K} \left( \prod_{k \in K \setminus I \shuffle_A J} \left( \nu^{\wt_K(k)} - g(k) \right) \right) \mathbb{G}_K(\nu),
\end{align*}
where 
\[
g(k) =
\begin{cases}
0 &\text{ if } k \in \e(A^{\mathrm c})     ,\\
\nu^{\wt_I(\std_{A^{\mathrm c}}(i))} &\text{ if }  k \in (I^{\mathrm c})_{A^{\mathrm c}} \setminus \e(A^{\mathrm c})      ,\\
\nu^{\wt_J(\std_A(i))} &\text{ if }  k \in (J^{\mathrm c})_{A} \setminus \e(A).
\end{cases}
\]
Note that $g(k) = \sw^w_m(k)|_{q=\nu}$ for each $k \in [m+n-1]$.
And, one can easily see that 
\begin{align*}
    \std_{A^{\mathrm c}}(\e(A^{\mathrm c})) \setminus \{m\} = \sDes(w_{\le m}^0) \quad \text{ and } \quad \std_A(\e(A)) \setminus \{n\} = \sDes(w_{> m}^0).
\end{align*}
Therefore our product rule holds for $q = \nu$. 
Now, the assertion is immediate since $\nu$ ranges over the infinite set $\{2, 3, \ldots\}$.

\end{proof}

\begin{proof}[proof of~\cref{thm: coproduct formula of G}]
For $\gamma \in \Comp_n$, set $I = \set(\gamma)$. For $1 \le k \le n-1$,
we have
\begin{align*}
    \mathbb{G}_I(\nu)\big\downarrow^{Q_n}_{Q_{[k-1]\sqcup[k+1,n-1]}}  &= (1-\nu^{\wt_I(k)})^{\delta_{k,I^{\mathrm c}}} \, \, \mathbb{G}_{I \cap [k-1]}(\nu) \otimes_{[k-1]} \slb \mathbb{X}_i \srb \\
    &= (1-\nu^{|I \cap [k-1]|+1})^{\delta_{k,I^{\mathrm c}}} \, \, \mathbb{G}_{I \cap [k-1]}(\nu) \otimes_{[k-1]} \slb \mathbb{X}_i \srb,
\end{align*}
where  
$\delta_{k,I^{\mathrm c}} =1$ if 
$k \in I^{\mathrm c}$ and $0$ otherwise,
and 
$$ 
\mathbb{X}_i  = 
\begin{cases}
    \mathbbm{1} &\text{ if } i \in I \cap [k+1,n-1],\\
    \overline{\mathbbm{reg}} - \nu^{\wt_I(i)}\mathbbm{1} &\text{ if } i \in I^{\mathrm c} \cap [k+1,n-1].
\end{cases}
$$
This says that 
\[
\blacktriangle_k(\mathbb{G}_I(\nu)) = (1-\nu^{|I \cap [k-1]|+1})^{\delta_{k,I^{\mathrm c}}} \, \, \mathbb{G}_{I \cap [k-1]}(\nu) \otimes (\iota^\ast_{[k+1,n-1]})^{-1} (\slb \mathbb{X}_i \srb).
\]
Set
$$\slb \mathbb{Y}_i \srb:=(\iota^\ast_{[k+1,n-1]})^{-1} (\slb \mathbb{X}_i \srb).$$ 
Then
\[
\mathbb{Y}_i = 
\begin{cases}
    \mathbbm{1} &\text{ if } i \in I \cap [k+1,n-1] - k,\\
    \overline{\mathbbm{reg}} - \nu^{\wt_I(i+k)}\mathbbm{1} &\text{ if } i \in I^{\mathrm c} \cap [k+1,n-1] - k.
\end{cases}
\]
Applying~\cref{lem: Expansion of class function into G} to this, we can rewrite $\slb \mathbb{Y}_i \srb$ as 
\[ 
\sum_{J: I \cap [k+1,n-1] -k \subseteq J} \left( \prod_{j \in J \setminus (I \cap [k+1,n-1] -k)} (\nu^{\wt_{J}(j)} - \nu^{\wt_I(j+k)}) \right) \mathbb{G}_J(\nu).
\]
Let 
$$J \setminus (I \cap [k+1,n-1] -k) = \{ t_1< t_2< \cdots\}.$$
We claim that $(\wt_I(t_i + k) - \wt_J(t_i))_{i=1,2,\ldots}$ is an arithmetic sequence with a common difference $-1$, i.e,
\begin{align}\label{eq: wt difference is 1}
    \wt_I(t_{i+1}+k) -\wt_J(t_{i+1}) = \wt_I(t_i + k) - \wt_J(t_i) -1.
\end{align}
Let 
$s := |\{i \in I \cap [k+1,n-1] -k : t_i < i < t_{i+1} \}|$.
One can easily show that
$\wt_I(t_{i+1}+k) = \wt_I(t_i + k) +s$ and
$\wt_J(t_{i+1}) = \wt_J(t_i) + s +1$,
which proves the claim.

Observe that $\wt_I(t_1+k) =1$ 
and 
$$\wt_I(t_1+k) = 
\begin{cases}
    |I \cap [k-1]| + 1 &\text{ if } k \notin I,\\
    |I \cap [k-1]| + 2 &\text{ if } k \in I.
\end{cases}
$$
This tells us that  
\begin{align}\label{eq: initial value of wt difference}
    ( 1- \nu^{\wt_I(t_1+k) - {\wt_{J}(t_1)} }) = 
\begin{cases}
    1-\nu^{|I \cap [k-1]|  } &\text{ if } k \notin I, \\
    1-\nu^{|I \cap [k-1]|+1} &\text{ if } k \in I.
\end{cases}
\end{align}

{\it Case 1: $k \notin I$.} 
In view of~\cref{eq: wt difference is 1}, one sees that $\slb \mathbb{Y}_i \srb$ is equal to 
\[
\sum_{J: I \cap [k+1,n-1] -k \subseteq J} \left( \prod_{j \in J \setminus (I \cap [k+1,n-1] -k)} \nu^{\wt_J(j)} \right)c_\nu(|I \cap [k-1]|, |J \setminus (I \cap [k+1,n-1] -k)|),
\]
where $c_\nu(|I \cap [k-1]|, |J \setminus (I \cap [k+1,n-1] -k)|)$ is the value obtained from $c_q(|I \cap [k-1]|, |J \setminus (I \cap [k+1,n-1] -k)|)$ by the specialization $q=\nu$.
Therefore, by~\cref{eq: initial value of wt difference}, 
\begin{align*}
&(1-\nu^{|I \cap [k-1]|+1})^{\delta_{k,I^{\mathrm c}}} \slb \mathbb{Y}_i \srb\\
&=\sum_{J: I \cap [k+1,n-1] -k \subseteq J} \nu^{g(I \cap [k+1,n-1] -k, J)} c_\nu(|I \cap [k-1]|+1, |J \setminus (I \cap [k+1,n-1] -k)|+1).
\end{align*}

{\it Case 2: $k \in I$.}  
As in the above case, we see that 
\begin{align*}
&(1-\nu^{|I \cap [k-1]|+1})^{\delta_{k,I^{\mathrm c}}} \slb \mathbb{Y}_i \srb\\
&=\sum_{J: I \cap [k+1,n-1] -k \subseteq J} \nu^{g(I \cap [k+1,n-1] -k, J)} c_\nu(|I \cap [k-1]|+1, |J \setminus (I \cap [k+1,n-1] -k)|).
\end{align*}

Set $\alpha := \comp(I \cap [k-1])$, $\beta := \comp(I \cap [k+1,n-1])$, and $\beta':=\comp(J)$. Then it holds that $|I \cap [k-1]| = \ell(\alpha)-1$ and $|J \setminus (I \cap [k+1,n-1] -k)| = \ell(\beta') - \ell(\beta)$.
In view of~\cref{eq: concatenation and near concatenation in set}, we see that the coproduct formula in consideration holds for $\nu$. 
Now, the the assertion is straightforward as $\nu$ ranges over the infinite set $\{2,3,\ldots \}$. 
\end{proof}

\section{A $q$ analogue of monomial quasisymmetric function}
\label{Section: q basis M for QSym}

Recall that the monomial quasisymmetric function $ M_{\alpha}$ is expanded in terms of the fundamental quasisymmetric functions as follows:
\[
M_{\alpha} = \sum_{\beta \preceq \alpha} (-1)^{\ell(\beta) -\ell(\alpha)} L_{\beta}
\]
(see~\cite[Proposition 5.2.8]{GR20}).
Motivated by this expansion, we here introduce a $q$-analogue of $ M_{\alpha}$
defined by  
\begin{align}\label{eq: definition of q basis M} 
M_{\alpha}(q) := \sum_{\beta \preceq \alpha} (-q)^{\ell(\beta) - \ell(\alpha)} L_{\beta}.
\end{align}
In view of~\cref{eq: mono to funda}, we have   
\begin{align*}
M_{\alpha}(q) &= \sum_{\beta \preceq \alpha} (1-q)^{\ell(\beta) - \ell(\alpha)} M_{\beta}\\
&= \sum_{\substack{1 \le i_1 \le i_2 \le \cdots \le i_n \\ i_j < i_{j+1} \text{ for } j \in \set(\alpha) }}(1-q)^{l - |\{i_1,i_2,\ldots,i_n \}| } x_{i_1}x_{i_2} \cdots x_{i_n}.
\end{align*}
It is evident that $\{ M_{\alpha}(q)\}$ is a basis of $\QSym_{\mathbb{C}(q)}$ since the transition matrix between $\{ M_{\alpha}(q) \}$ and $\{ L_{\alpha} \}$ is a triangular matrix with non-zero diagonals. Additionally, when $q$ is specialized to $0$ and $1$, $M_{\alpha}(q)$ corresponds to $L_{\alpha}$ and $M_{\alpha}$, respectively.

In this section, we derive product and coproduct formulas for $\{ M_{\alpha}(q) \}$ in the same way as in the previous section.

\begin{theorem}\label{thm: structure constant for M}
The following formulas hold.
    \begin{enumerate}[label = {\rm (\alph*)}]
        \item Let $I \subseteq [m-1]$ and $J \subseteq [n-1]$. Then
        \begin{align*}
        &M_{\comp(I)}(q) M_{\comp(J)}(q)\\
        &= \sum_{A \in \binom{[m+n]}{n}} (1-q)^{|\overline{\e}(A) \setminus I \#_A J|-1} \left( \sum_{S \subseteq \e(A^{\mathrm c})} q^{|S|} M_{\comp( (I \shuffle_A J)  \sqcup S)}(q) \right).
        \end{align*}
        \item Let $\gamma \in \Comp$. Then
        \[
        \triangle M_{\gamma}(q) = \sum_{\alpha \cdot \beta = \gamma \text{ or } \alpha \odot \beta = \gamma} (1-q)^{\gamma(\alpha, \beta)} \, M_{\alpha}(q) \otimes M_{\beta}(q),
        \]
        where $\gamma(\alpha, \beta) = 
        \begin{cases}
        0 &\text{ if } \alpha \odot \beta = \gamma \\
        1 &\text{ otherwise. }
        \end{cases}$
    \end{enumerate}
    
\end{theorem}
\begin{proof}
As in the proof of~\cref{thm: str consts for B}, we have only to show that 
the equalities in (a) and (b) hold whenever $q$ is specialized into a positive integer $\nu> 1$.
Let us fix a positive integer $\nu>1$.

(a)
Define $\mathbbm{M}_I(\nu)=[(\mathbbm{M}_I(\nu)_i)_{i\in [m-1]}]\in \scf(\mathcal{N}_m(\nu))$ as follows:
\[
\mathbbm{M}_I(\nu)_i:= 
\begin{cases}
    \mathbbm{1}  & \text{if } i \in I,\\
    \overline{\mathbbm{reg}} - \nu \mathbbm{1}             & \text{if } i \in [n-1] \setminus I.
\end{cases}
\]
Due to~\cref{eq: definition of q basis M}, this definition yields that 
$\ch_{\nu}(\mathbbm{M}_I(\nu)) = M_{\comp(I)}(\nu)$.
Furthermore, by applying the identity
$$
    \overline{\mathbbm{reg}} = \left(\overline{\mathbbm{reg}} - \nu \mathbbm{1}\right) +\nu \mathbbm{1}
$$
to the $z_A$th entry of ${\bf s}_A(\mathbbm{M}_I(\nu), \mathbbm{M}_J(\nu))$, we derive that 
\begin{equation}\label{equality for SA}
{\bf s}_A(\mathbbm{M}_I(\nu), \mathbbm{M}_J(\nu)) = \mathbbm{M}_{I \#_A J }(\nu) + \nu \mathbbm{M}_{(I \#_A J) \sqcup \{z_A\} }(\nu)
\end{equation}
(for the definition of ${\bf s}_A$, see~\cref{def of s(phi psi)}).
Recall that 
\begin{align}\label{recall of mA}
{\bf m}_A(\mathbbm{M}_I(\nu), \mathbbm{M}_J(\nu))=\dot\chi^{\e(A)}(\nu) \otimes_{\overline{\e}(A)} 
\left( {\bf s}_A(\mathbbm{M}_I(\nu), \mathbbm{M}_J(\nu))  \big\downarrow^{Q_{m+n+1}(\nu)}_{Q_{[m+n-1] \setminus \overline{\e}(A)}(\nu)} \right)
\end{align}
(see~\cref{def: product of scf}).
Plugging the right hand side of~\cref{equality for SA} into~\cref{recall of mA}, we can derive that 
\begin{equation} \label{expansion of mA}
\begin{aligned}
&{\bf m}_A(\mathbbm{M}_I(\nu), \mathbbm{M}_J(\nu)) \\
&=
\left((1-\nu)^{|\overline{\e}(A) \setminus I \#_A J|} + \nu(1-\nu)^{|\overline{\e}(A) \setminus I \#_A J|-1}\right) (\dot\chi^{\e(A)}(\nu) \otimes \mathbbm{M}_{I \#_A J \setminus \overline{\e}(A)}(\nu))\\
&=(1-\nu)^{|\overline{\e}(A) \setminus I \#_A J|-1} (\dot\chi^{\e(A)}(\nu) \otimes \mathbbm{M}_{I \#_A J \setminus \overline{\e}(A)}(\nu)).
\end{aligned}
\end{equation}
By the definitions of $\dot\chi$ and $\mathbbm{M}(\nu)$, it follows that    
$$\dot\chi^{\e(A)}(\nu) \otimes \mathbbm{M}_{I \#_A J \setminus \overline{\e}(A)}(\nu) = \slb ( \xi_i)_{i\in [m+n-1]} \srb,$$ 
where  
\begin{align*}
    \xi_i :=
    \begin{cases}
    \mathbbm{1}  & \text{if } i \in I \shuffle_A J,\\
    \overline{\mathbbm{reg}}             & \text{if } i \in \e(A^{\mathrm c}),\\
    \overline{\mathbbm{reg}} - \nu \mathbbm{1}             & \text{otherwise}.
    \end{cases}
\end{align*}
Now the assertion can be obtained by replacing $\dot\chi^{\e(A)}(\nu) \otimes \mathbbm{M}_{I \#_A J \setminus \overline{\e}(A)}(\nu)$ by  
$\slb ( \xi_i)_{i\in [m+n-1]} \srb$ in~\cref{expansion of mA}.

(b) 
Let $I = \set(\gamma)$.
It holds that 
\[
\mathbbm{M}_I(\nu) \big\downarrow^{Q_n(\nu)}_{Q_{[k-1] \sqcup [k+1,n-1] }(\nu)} = 
\begin{cases}
      \mathbbm{M}_{I \cap [k-1]}(\nu) \otimes_{[k-1]} \mathbbm{M}_{ (I \cap [k+1,n-1])}(\nu) &\text{ if } k \notin I, \\
      (\nu-1) \mathbbm{M}_{I \cap [k-1]}(\nu) \otimes_{[k-1]} \mathbbm{M}_{ (I \cap [k+1,n-1])}(\nu) &\text{ if } k \in I.
\end{cases}
\]
Since 
$$(\iota^\ast_{[k+1,n-1]})^{-1} (\mathbbm{M}_{ (I \cap [k+1,n-1])}(\nu)) = \mathbbm{M}_{ (I \cap [k+1,n-1]-k)}(\nu),$$ 
using~\cref{when simple tensor}, one can derive that 
\begin{align*}
\blacktriangle_k(\mathbbm{M}_I(\nu)) = 
\begin{cases}
\mathbbm{M}_{I \cap [k-1]}(\nu) \otimes (\mathbbm{M}_{I \cap [k+1, n-1]-k}(\nu)) &\text{ if } k \notin I,\\
(\nu-1) \mathbbm{M}_{I \cap [k-1]}(\nu) \otimes  (\mathbbm{M}_{I \cap [k+1, n-1]-k}(\nu)) &\text{ if } k \in I.
\end{cases}
\end{align*}
Hence the assertion is immediate from~\cref{eq: concatenation and near concatenation in set}.
\end{proof}

\bibliographystyle{abbrv}
\bibliography{references}
\end{document}